\documentclass[11pt]{article}
\usepackage{fullpage, amssymb}
\usepackage{caption}
\usepackage{subcaption}
\usepackage{graphicx}
\usepackage{amsmath}
\usepackage{float}
\usepackage{listings}
\usepackage{xcolor}
\usepackage{amsthm}
\usepackage{algorithmic}
\usepackage{algorithm}
\usepackage[colorlinks,
linkcolor=blue,
anchorcolor=blue,
citecolor=blue]{hyperref}
\usepackage{natbib}
\usepackage{comment}
\usepackage{bbm}
\usepackage[shortlabels]{enumitem}
\usepackage{romannum}
\AtBeginDocument{\pagenumbering{arabic}}
\usepackage{multirow}
\usepackage{stmaryrd}

\setcitestyle{authoryear,round}

\usepackage{epstopdf}
\newtheorem{theorem}{Theorem}
\newtheorem{lemma}{Lemma}

\newtheorem{proposition}{Proposition}
\newtheorem{corollary}{Corollary}
\newtheorem{assumption}{Assumption}
\newtheorem{remark}{Remark}

\newcommand{\iid}{\textit{i}.\textit{i}.\textit{d}. }

\def\rank{\text{rank}}
\def\wt{\widetilde}
\def\Var{{\textrm{Var}}\,}

\newcommand\fro[1]{\| #1 \|_{\rm{F}}}
\newcommand\fror[1]{\| #1 \|_{\rm{F},r}}
\newcommand\frorr[1]{\| #1 \|_{\rm{F},\r}}
\newcommand\op[1]{\| #1 \|}

\newcommand\lone[1]{\| #1 \|_{1}}

\newcommand{\inp}[2]{\langle #1,#2\rangle}

\def\calN{{\mathcal N}}

\def\calP{{\mathcal P}}

\def\bSigma{{\boldsymbol{\Sigma}}}

\def\bcalE{{\boldsymbol{\mathcal E}}}

\def\BB{{\mathbb B}}

\def\EE{{\mathbb E}}
\def\FF{{\mathbb F}}

\def\MM{{\mathbb M}}

\def\OO{{\mathbb O}}
\def\PP{{\mathbb P}}

\def\RR{{\mathbb R}}
\def\SS{{\mathbb S}}
\def\TT{{\mathbb T}}

\def\r{{\mathbf r}}

\def\u{{\mathbf u}}
\def\v{{\mathbf v}}

\def\A{{\mathbf A}}
\def\B{{\mathbf B}}
\def\C{{\mathbf C}}

\def\G{{\mathbf G}}

\def\I{{\mathbf I}}

\def\L{{\mathbf L}}
\def\M{{\mathbf M}}
\def\N{{\mathbf N}}

\def\R{{\mathbf R}}

\def\U{{\mathbf U}}
\def\V{{\mathbf V}}

\def\X{{\mathbf X}}

\def\Z{{\mathbf Z}}

\def\lmax{{l_{\textsf{max}}}}
\def\tauc{\tau_{\textsf{\tiny comp}}}
\def\taus{\tau_{\textsf{\tiny stat}}}
\def\muc{\mu_{\textsf{\tiny comp}}}
\def\mus{\mu_{\textsf{\tiny stat}}}
\def\Lc{L_{\textsf{\tiny comp}}}
\def\Ls{L_{\textsf{\tiny stat}}}
\def\eps{\varepsilon}

\def\bd{\boldsymbol{\Delta}}
\def\tm{\tilde{\M}}

\newcommand\bpsi[1]{\big\|#1\big\|_{\psi_2}}
\def\mins{\underline{\lambda}}
\def\maxs{\bar{\lambda}}
\def\dmax{\bar{d}}
\def\rmax{\bar{r}}
\def\gl{\nabla L_n}
\def\gL{\nabla L}
\def\gcL{\nabla_{\C} L}
\def\gcl{\nabla_{\C} L_n}

\begin{document}
\pagenumbering{arabic}

\title{Computationally Efficient and Statistically Optimal Robust Low-rank Matrix and Tensor Estimation}

\author{Yinan Shen, Jingyang Li, Jian-Feng Cai\footnote{Jian-Feng Cai’s research was partially supported by Hong Kong RGC Grant GRF 16310620 and GRF 16309219.}  and Dong Xia\footnote{Dong Xia's research was partially supported by Hong Kong RGC Grant ECS 26302019, GRF 16303320 and GRF 16300121.}\\
%Department of Mathematics, 
{\small Hong Kong University of Science and Technology}}

\date{(\today)}

\maketitle

\begin{abstract}

Low-rank matrix estimation under heavy-tailed noise is challenging,  both computationally and statistically.  Convex approaches have been proven statistically optimal but suffer from high computational costs,  especially since robust loss functions are usually non-smooth.   More recently,  computationally fast non-convex approaches via sub-gradient descent are proposed,  which,  unfortunately,  fail to deliver a statistically consistent estimator even under sub-Gaussian noise.  In this paper,  we introduce a novel  Riemannian sub-gradient (RsGrad) algorithm which is not only computationally efficient with linear convergence but also is statistically optimal,  be the noise Gaussian or heavy-tailed with a finite $1+\eps$ moment.  Convergence theory is established for a general framework and specific applications to absolute loss,  Huber loss and quantile loss are investigated.  Compared with existing non-convex methods,  ours reveals a surprising phenomenon of {\it dual-phase convergence}.  In phase one,  RsGrad behaves  as in a typical non-smooth optimization that requires gradually decaying stepsizes.  However,  phase one only delivers a statistically sub-optimal estimator which is already observed in existing literature.  Interestingly,   during phase two,  RsGrad converges linearly {\it as if} minimizing a smooth and strongly convex objective function and thus a constant stepsize suffices.  Underlying the phase-two convergence is the {\it smoothing effect} of random noise to the non-smooth robust losses in an area {\it close but not too close} to the truth. Lastly, RsGrad is applicable for low-rank tensor estimation under heavy-tailed noise where statistically optimal rate is attainable with the same phenomenon of dual-phase convergence, and a novel shrinkage-based second-order moment method is guaranteed to deliver a warm initialization.   Numerical simulations confirm our theoretical discovery and showcase the superiority of RsGrad over prior methods.  

\end{abstract}

\section{Introduction}\label{sec:intro}
Let $\M^{\ast}$ be a $d_1\times d_2$ matrix of rank $r\ll d_2$ where, without loss of generality, we assume $d_1\geq d_2$. Denote $\{(\X_i, Y_i)\}_{i=1}^n$ the collection of i.i.d.  observations satisfying 
\begin{equation}\label{eq:tr_model}
Y_i=\langle \X_i, \M^{\ast}\rangle+\xi_i
\end{equation}
with $\X_i\in\RR^{d_1\times d_2}$ known as the {\it measurement matrix}. Here, $\langle \X_i, \M^{\ast}\rangle={\rm tr}(\X_i^{\top}\M^{\ast})$ and hence  (\ref{eq:tr_model}) is often called the {\it trace regression} model. See, e.g., \cite{koltchinskii2011neumann,koltchinskii2015optimal,rohde2011estimation} and references therein. The latent noise $\xi_i$ has mean zero and is usually assumed {\it random} for ease of analysis. 
{\it Noisy low-rank matrix estimation} or {\it matrix sensing} refers to the goal of recovering $\M^{\ast}$ from the observations $\{(\X_i, Y_i)\}_{i=1}^n$. This problem and its variants naturally arise in diverse fields --- quantum state tomography \citep{gross2010quantum,xia2016estimation},  multi-task learning \citep{chen2011integrating, negahban2011estimation},  structured model inference \citep{chiu2021low, siddiqi2010reduced},  to name but a few.  Oftentimes,  the dimensions $d_1, d_2$ are large.  It is of great interest to effectively,  both computationally and statistically,  recover $\M^{\ast}$ using as few measurements as possible.    

The past two decades have witnessed a vast literature on designing computationally efficient methods and investigating their statistical performances and limits for estimating the underlying low-rank matrix.   At the common ground of these methods is a carefully selected loss function that leverages the low-rank structure.  Estimating $\M^{\ast}$ then boils down to solving an optimization program.  Towards that end,  all these methods take an either {\it convex} or {\it non-convex} approach.  Typical convex approaches introduce an additional penalization by matrix nuclear norm into the loss function to promote low-rank solutions.  This convexity renders well-understood algorithms for convex programming immediately applicable.  Therefore statisticians  can just focus on studying the statistical performances without paying too much attention to the computational implementations.  For instance,  by assuming i.i.d.  {\it sub-Gaussian} noise with variance $\sigma^2$ and the so-called {\it restricted isometry property} (RIP) or {\it restricted strong convexity} (RSC),  it has been proved \citep{candes2011tight, rohde2011estimation,negahban2011estimation,cai2015rop,chen2015fast,davenport2016overview,chandrasekaran2012convex} that the nuclear-norm penalized least square estimator,  referred to as the convex-$\ell_2$ approach,  attains the error rate $\tilde{O}_p(\sigma^2rd_1n^{-1})$\footnote{Here $\tilde{O}_p(\cdot)$ stands for the typical big-O notation up to logarithmic factors and holds with high probability.} in squared Frobenius norm.  This rate is shown to be optimal in the minimax sense.  See,  e.g.,  \cite{xia2014optimal,ma2015volume} and references therein.   Despite the appealing theoretical performances,  convex approaches suffer two major drawbacks.  First,  convex methods operate directly on a matrix of size $d_1\times d_2$ making them run slowly and unscalable to ultra-high dimensional problems.   Secondly,  while the established error rate is theoretically minimax optimal,  the implicit constant factor seems large and the resultant error rate is,  {\it in practice},  often inferior to that by non-convex approaches.  Fortunately,  non-convex approaches can nicely address these two drawbacks.  The starting point of non-convex methods is to re-parametrize the optimization program by writing $\M^{\ast}=\U^{\ast}\V^{\ast\top}$ with $\U^{\ast}$ and $\V^{\ast}$ both having $r$ columns,  i.e.,  requiring the knowledge of rank.  It then suffices to minimize the loss with respect to the factors $\U$ and $\V$ by,  for instance,  the projected gradient descent \citep{chen2015fast},  Burer-Monteiro type gradient descent \citep{burer2003nonlinear,zheng2016convergent},  rotation-calibrated gradient descent \citep{zhao2015nonconvex,xia2021statistical},  Riemannian gradient descent \citep{wei2016guarantees} and etc.  Compared to the convex counterparts,  these non-convex algorithms operate on matrix of size $d_1\times r$ directly and easily scales to large-scale setting even when $d_1$ is about millions.  Under similar conditions and with a non-trivial initialization,  it has been demonstrated that these non-convex algorithms converge fast (more exactly,  linearly) and also deliver minimax optimal estimators under sub-Gaussian noise.

The recent boom of data technology poses new challenges to noisy low-rank matrix estimation,  among which heavy-tailed noise has appeared routinely in numerous applications such as diffusion-weighted imaging \citep{chang2005restore},  on-line advertising \citep{sun2017provable},  and gene-expression data analysis \citep{sun2020adaptive}.  Actually,  the aforementioned convex and non-convex approaches with a square loss become vulnerable or even completely useless when the noise has a heavy tail.  It is because that heavy-tailed noise often generates a non-negligible set of outliers,  to which the square loss is sensitive.   To mitigate the effect of heavy-tailed noise,   a natural solution is to replace the square loss by more robust but {\it non-smooth} ones.  Notable examples include the $\ell_1$-loss \citep{candes2011robust, cambier2016robust},  the renowned Huber loss \citep{huber1965robust} and quantile loss \citep{koenker2001quantile},  all of which are convex but non-smooth.   For instance,  \cite{elsener2018robust} proposes a convex approach based on nuclear-norm penalized $\ell_1$ or Huber loss,  and proves that their estimator, under a mild condition,  attains the error rate $\tilde{O}_p(rd_1n^{-1})$ in squared Frobenius norm as long as the noise has a non-zero density in a neighbourhood of origin.  The rate is minimax optimal in terms of the model complexity.  A more general framework requiring only Lipschitz and convex loss but imposing no assumption on the noise is investigated by \cite{alquier2019estimation},  though the scenario of heavy-tailed noise is not specifically discussed.   See also \cite{klopp2017robust}.  The application of foregoing approaches still rely on convex programming  and suffer from the computational issues explained above.  In fact,  the issue is severer here due to the non-smoothness of objective function \citep{boyd2004convex}.  In consideration of computational efficiency,  non-convex approaches have been proposed and investigated for minimizing non-smooth but convex loss functions.  The sub-Gradient descent algorithms based on matrix factorization were studied recently by  \cite{charisopoulos2019lowrank, li2020nonconvex,tong2021low},  and all of them converge linearly if equipped with a good initialization and with properly chosen stepsizes.  While being computationally fast,  the statistical performances of these non-convex approaches under heavy-tailed noise are either largely missing or,  generally,  sub-optimal.  To be more specific,    \cite{li2020nonconvex} only proves the exact recovery under sparse outliers {\it without noise}; the rate in squared Frobenius norm established in \cite{charisopoulos2019lowrank} and \cite{li2020nonconvex} turns into $\tilde{O}_p(\sigma^2)$ even if  the noise are Gaussian with variance $\sigma^2$.  This rate has none statistical significance since it implies that more data (i.e.,  $n$ increases) will not improve the estimate.  Put it differently,  these estimators are not even statistically consistent.  Besides the aforementioned optimization-oriented approaches,  \cite{minsker2018sub} and \cite{fan2021shrinkage} propose computationally fast estimators based on proper shrinkage and low-rank approximation.  Though the rate $\tilde{O}_p(rd_1n^{-1})$ is achievable,  it is not proportional to the noise size,  even for sub-Gaussian noise,  leaving a considerable gap as noise level varies.    
These prior works all  present unpleasant defects in treating low-rank matrix estimation under heavy-tailed noise.  

To bridge the aforementioned gap,  in this paper,  we propose a computationally efficient non-convex algorithm based on Riemannian sub-gradient (RsGrad) optimization and demonstrate its statistical optimality for several mainstream robust loss functions including absolute loss,  Huber loss and quantile loss.    At the core of RsGrad is to view the set of rank-$r$ matrices as a smooth manifold so that established manifold-based optimization methods are readily applicable \citep{vandereycken2013low,cambier2016robust},  among which Riemannian (sub-)gradient descent is perhaps the most popular.  See,  e.g. ,  \cite{wei2016guarantees,cai2021generalized} and references therein.  Similarly as \cite{charisopoulos2019lowrank,li2020nonconvex,tong2021low},  with a proper scheduling of stepsizes,  RsGrad algorithm admits fast computation and converges linearly.  Surprisingly,  through a more sophisticated analysis,  we discover that RsGrad exhibits an intriguing,  referred to as {\it dual-pahse convergence},  phenomenon.   In phase one,   RsGrad behaves like a typical {\it non-smooth} optimization \citep{tong2021low},  e.g.,  requiring gradually decaying stepsizes through iterations.  For example,  if the noise is Gaussian with variance $\sigma^2$ and an absolute loss is equipped,  the phase-one iterations reach an estimator whose squared Frobenius norm is $O_p(\sigma^2)$.  This rate has been achieved by \cite{tong2021low} whereas theirs is based on a scaled sub-gradient (ScaledSM) descent algorithm.  However,  we reveal a phase two convergence of RsGrad.  In phase two,  while the objective function is still non-smooth (but locally smooth in a small region),  the convergence of RsGrad behaves like a smooth optimization in that a constant stepsize guarantees a linear convergence.   As a result,  after phase-two iterations,  RsGrad can output an estimator which is statistically optimal under mild conditions,  e.g.  it achieves the rate $O_p(\sigma^2rd_1n^{-1})$ for the previous example.   Similar results are also established for other robust loss functions under general heavy-tailed noise.

The merits of RsGrad are further demonstrated on robust low-rank (Tucker) tensor estimation. A tensor is a multi-dimensional array, e.g., a matrix is a $2$nd-order tensor. The trace regression model (\ref{eq:tr_model}) is frequently studied for low-rank tensor estimation where $\M^{\ast}$ and all $\X_i$'s are tensors. By unfolding a tensor into matrices, \cite{tomioka2013convex,mu2014square} proposed convex methods using matrix nuclear norms which, unfortunately, only achieve statistically sub-optimal rates under Gaussian noise and demand high computational costs. A delicate tensor nuclear norm was introduced by \cite{raskutti2019convex} which, albeit statistically appealing, is computationally NP-hard in general. In comparison, non-convex methods via direct tensor decomposition often simultaneously enjoy the statistical optimality and computational efficiency. Representative works in this line include the projected gradient descent algorithm \citep{chen2019non}, penalized jointly gradient descent \citep{han2022optimal}, Grassmannian gradient descent \citep{lyu2021latent}, Riemannian gradient descent \citep{cai2021generalized} and etc. Statistical performances of the aforementioned works are guaranteed under sub-Gaussian noise and considerably deteriorate when the noise have heavy tails.  During the preparation of this work, we noticed that \cite{tongaccelerating} extends the ScaledSM algorithm to robust low-rank tensor estimation that naturally handles heavy-tailed noise. However, as discussed above, the statistical rate achieved by ScaledSM even under Gaussian noise is sub-optimal and stays still even if the sample size $n$ increases. Interestingly, our RsGrad algorithm equipped with robust loss functions is easily applicable to heavy-tailed low-rank tensor estimation. We establish a similar dual-phase convergence and derive statistically optimal rates under a minimal $1+\eps$ moment condition and a nearly optimal sample size requirement. This clearly fills the void in robust low-rank tensor estimation under heavy-tailed noise.

Our contributions are multi-fold.   First,  we propose a computationally efficient method for robust low-rank matrix estimation that is applicable to a wide class of non-smooth loss functions.  While Riemannian sub-gradient descent (RsGrad) has been introduced by \cite{cambier2016robust} for minimizing the absolute loss,  to our best knowledge,  there exists no theoretical guarantees for its convergence.  Under mild conditions,  we prove that RsGrad converges linearly regardless of the condition number of $\M^{\ast}$ making it the preferable algorithm for estimating an ill-conditioned matrix.  Secondly,  we demonstrate the statistical optimality of the final estimator delivered by RsGrad,  be the noise Gaussian or heavy-tailed.   Specific applications to the absolute loss,  Huber loss and quantile loss confirm that a rate $O_p(rd_1n^{-1})$ is attainable as long as  the noise has a fine $1+\eps$ moment and its density satisfies mild regularity conditions.  The same conditions have appeared in \cite{elsener2018robust}.  Unlike \cite{fan2021shrinkage} and \cite{minsker2018sub},  our rate is proportional of the noise size.  Thirdly,  our analysis reveals a new phenomenon of {\it dual-phase convergence} that enables us to achieve statistically optimal error rates,  which is a significant improvement over existing literature \citep{charisopoulos2019lowrank, li2020nonconvex, tong2021low}.  While the phase-one convergence is typical for a non-smooth optimization,  the phase-two convergence,  interestingly,  behaves like a smooth optimization.  It seems that the random noise has an effect of smoothing in the phase-two convergence.  Though our dual-phase convergence is only established for RsGrad,  we believe that this should also occur for the factor-based sub-gradient descent \citep{charisopoulos2019lowrank} and the scaled sub-gradient descent \citep{li2020nonconvex}. Finally, we derive the statistically optimal rate of low-rank tensor estimation under heavy-tailed noise. Under a slightly stronger $2+\eps$ moment condition, we propose a novel shrinkage-based second-order moment method that guarantees a warm initialization under a nearly optimal sample size and signal-to-noise condition. To our best knowledge, these are the first results of this kind in tensor-related literature. 

The rest of paper is organized as follows.  In Section~\ref{sec:RsGrad},  we introduce examples of non-smooth robust loss functions and our Riemannian sub-gradient descent (RsGrad) algorithm in a general framework.  Section~\ref{sec:theory} presents the general convergence theory of RsGrad,  i.e.,  dual-phase convergence,  under the dual-phase regularity condition of loss function.  Specific applications to absolute loss,  Huber loss and quantile loss are investigated for both Gaussian noise and heavy-tailed noise in Section~\ref{sec:app}.  Methods of initialization and discussions about stepsize selection are provided in Section~\ref{sec:discussion}. Section~\ref{sec:tensor} extends RsGrad to robust tensor estimation and derives statistically optimal rates under heavy-tailed noise. We showcase the results of numerical experiments and comparison with prior methods in Section~\ref{sec:simulation}.  All the proofs are relegated to the Appendix.

\section{Robust Loss  and Riemannian Sub-gradient Descent}\label{sec:RsGrad}
Without loss of generality, we first focus on the case of matrix estimation. Extension to robust tensor estimation can be found in Section~\ref{sec:tensor}. 
Denote $\MM_r:=\{\M\in\RR^{d_1\times d_2},  \textrm{rank}(\M)\leq r\}$ the set of matrices with rank bounded by $r$.  With a properly selected loss function $\rho(\cdot): \RR\mapsto \RR_+$,  we aim to solve
\begin{align}\label{eq:loss}
\hat\M:=\underset{\M\in\MM_r}{\arg\min}\ f(\M),  \quad \textrm{ where } f(\M) := \sum_{i=1}^n\rho\big(\langle \M,  \X_i\rangle-Y_i\big)
\end{align}
The statistical property of $\hat\M$ crucially relies on the loss function.  For instance,  $\hat\M$ attained by the $\ell_2$-loss ({\it square loss}),  i.e.,  $\rho(x):=x^2$,  has been proven effective \citep{chen2015fast,wei2016guarantees,xia2021statistical} in dealing with sub-Gaussian noise but is well-recognized sensitive to outliers and heavy-tailed noise \citep{huber1965robust}.  Fortunately,  there are the so-called {\it robust} loss functions which are relatively immune to outliers and heavy-tailed noise and capable to deliver a more reliable estimate $\hat\M$.  Examples of robust loss function include:
{\it
\begin{enumerate}[1.]
\item absolute loss ($\ell_1$-loss):  $\rho(x):=|x|$ for any $x\in\RR$;
\item Huber loss: $\rho_{H,\delta}(x):=x^2\mathbbm{1}(|x|\leq \delta)+(2\delta|x|-\delta^2)\mathbbm{1}(|x|>\delta)$ for any $x\in\RR$ where $\delta>0$ is a tuning parameter.
\item quantile loss: $\rho_{Q,\delta}(x):=\delta x\mathbbm{1}(x\geq 0)+(\delta-1)x\mathbbm{1}(x<0)$ for any $x\in \RR$ with $\delta:=\PP(\xi\leq 0)$.
\end{enumerate}
}
\noindent The absolute loss,  Huber loss and quantile loss are all convex.  They have appeared in the literature for noisy low-rank matrix estimation including the convex approach based on nuclear norm penalization \citep{elsener2018robust,klopp2017robust,candes2011robust} and non-convex approach based on gradient-style algorithms \citep{li2020nonconvex,tong2021accelerating}.  
While our theory developed in Section~\ref{sec:theory} applies to general (robust) loss functions,  the cases of absolute loss \citep{tong2021low, charisopoulos2019lowrank},  Huber loss \citep{elsener2018robust,sun2020adaptive} and quantile loss \citep{alquier2019estimation, chenrobust} will be specifically investigated in Section~\ref{sec:app}.    Compared to the square loss,  most robust loss functions are non-smooth,  i.e.,  their derivatives are dis-continuous,  which brings new challenge to its computation.  Indeed,  optimizing program (\ref{eq:loss}) usually exploits the sub-gradient of robust loss functions,  written as $\partial f(\M)$.  See,  for instance,  \cite{charisopoulos2019lowrank} and references therein. 

Due to the low-rank constraint,  program (\ref{eq:loss}) is non-convex and solvable only locally.  One popular tactic to enforce the low-rank constraint is to reparametrize the program (\ref{eq:loss}) by factorization $\M=\U\V^{\top}$ with $r$-columned matrices $\U$ and $\V$.  Then it suffices to update $(\U,  \V)$ sequentially to minimize (\ref{eq:loss}) by local algorithms,  e.g.,  by sub-Gradient descent.  These algorithms \citep{zheng2016convergent,zhao2015nonconvex} often run fast if $\M^{\ast}$ is {\it well-conditioned},  namely $\sigma_1(\M^{\ast})\sigma_r^{-1}(\M^{\ast})$ is small,  where $\sigma_j(\M)$ denotes the $j$-th singular value of  $\M$ so that $\sigma_1(\M)^{\ast}$ is the operator norm of $\M$.   However,  they suffer a great loss of computational efficiency if $\M^{\ast}$ is {\it ill-conditioned}.  Though the issue can be theoretically remedied by a proper inverse scaling \citep{tong2021accelerating},  it can cause a potential computational instability especially when $\M^{\ast}$ has small non-zero singular values.  Recently,  it is discovered that optimizing (\ref{eq:loss}) directly on the low-rank manifold $\MM_r$,   called {\it Riemannian},  enjoys the fast computational speed of factorization-based approaches and,  meanwhile,  converges linearly regardless of the condition number of $\M^{\ast}$.  See,  e.g.,  \cite{cai2021generalized} and \cite{cai2021provable},  for the convergence of Riemannian gradient descent (RGrad) algorithms in minimizing {\it strongly convex and smooth} functions with tensor-related applications.     We note that RGrad is similar to the projected gradient descent (PGD, \cite{chen2015fast}) except that RGrad utilizes the Riemannian gradient while PGD takes the vanilla one.  

We adapt Riemannian optimization to solve program (\ref{eq:loss}).  Since $f(\cdot)$ can be non-smooth so that the sub-gradient is employed,  we refer our algorithm to as the Riemannian sub-gradient (RsGrad) descent.  At the $l$-th iteration with a current low-rank estimate $\M_l$,  the algorithm consists of two major steps.  It begins with computing the Riemannian sub-gradient,  which is the projection of a {\it vanilla} sub-gradient $\G_l\in\partial f(\M_r)\subset \RR^{d_1\times d_2}$ onto the tangent space,  denoted by $\TT_l$,  of $\MM_r$ at the point $\M_r$.  The benefit of using Riemannian sub-gradient, written as $\calP_{\TT_l}(\G_l)$,  instead of the vanilla one is that the Riemannian sub-gradient is low-rank while the vanilla one is full-rank \citep{cambier2016robust},  which significantly boosts up the subsequent singular value decomposition (SVD).    Interested readers are suggested to refer to \cite{vandereycken2013low,wei2016guarantees,cai2021generalized} and references therein for more discussions.  The second step is to  update the low-rank estimate along the direction of negative Riemannian sub-gradient and then retract it back to the manifold $\MM_r$,  for which,  it suffices to take the SVD.  Here ${\rm SVD}_r(\cdot)$ returns the best rank-$r$ approximation by SVD. The details of RsGrad are presented in Algorithm~\ref{alg:RsGrad}.   Note that the Riemannian sub-gradient algorithm has been introduced by \cite{cambier2016robust} for minimizing the absolute loss without convergence and statistical analysis.  In contrast,  our framework covers general robust loss functions,  and  we prove its computational efficiency and statistical optimality for several important applications. 

\begin{algorithm}
\caption{Riemannian Sub-gradient Descent (RsGrad)}\label{alg:RsGrad}
\begin{algorithmic}
	\STATE{\textbf{Input}: observations $\{(\X_i, Y_i)\}_{i=1}^n$,  max iterations $\lmax$,  step sizes $\{\eta_l\}_{l=0}^{\lmax}$.}
	\STATE{Initialization: $\M_0\in\MM_r$}
	\FOR{$l = 0,\ldots,\lmax$}
	\STATE{Choose a vanilla subgradient:  $\G_l\in\partial f(\M_l)$}
	\STATE{Compute Riemannian sub-gradient: $\wt\G_l = \calP_{\TT_l}(\G_l)$}
	\STATE{Retraction to $\MM_r$: $\M_{l+1} = \text{SVD}_r(\M_l - \eta_{l}\wt\G_l)$}
	\ENDFOR
	\STATE{\textbf{Output}: $\hat\M=\M_{\lmax}$}
\end{algorithmic}
\end{algorithm}

\paragraph*{Computation.} The Riemannian sub-gradient is fast computable.  Let $\M_l=\U_l\bSigma_l\V_l^{\top}$ be the thin SVD of $\M_l$.  It is well-known \citep{absil2009optimization,  vandereycken2013low} that the tangent space $\TT_l$ can be characterized by $\TT_l:=\{\Z\in\RR^{d_1\times d_2}: \Z=\U_l\R^{\top}+\L\V_l^{\top},  \R\in\RR^{d_2\times r},  \L\in\RR^{d_1\times r}\}$.  Then,  for any $d_1\times d_2$ matrix $\G_l$,  the projection onto $\TT_l$ is
$$
\calP_{\TT_l}(\G_l)=\U_l\U_l^{\top}\G_l+\G_l\V_l\V_l^{\top}-\U_l\U_l^{\top}\G_l\V_l\V_l^{\top},
$$
 which is of rank at most $2r$.  Consequently,  the final step of retraction only requires the SVD of a $2r\times 2r$ matrix.  See,  e.g.,  \cite{vandereycken2013low} and \cite{mishra2014fixed} for more details.

\section{General Convergence Performance}\label{sec:theory}
We now present the general convergence performance of RsGrad Algorithm~\ref{alg:RsGrad},  which essentially relies on regularity conditions of the objective function.   These conditions,  in spirit,  largely inherit those from existing literature \citep{charisopoulos2019lowrank,elsener2018robust,alquier2019estimation,tong2021accelerating}.  However,  as explained in Section~\ref{sec:intro},  these prior works only delivered statistically sub-optimal estimates.  It turns out that more delicate characterizations of these conditions are necessary for our purpose.  More exactly,  we discover that the aforementioned robust functions exhibit strikingly different regularity conditions near and   far away from $\M^{\ast}$,  referred to as the {\it dual-phase} regularity conditions. 
\begin{assumption}\label{assump:dual-phase}(Dual-phase regularity conditions)
Let $\tauc>\taus>0$ and define two regions around the truth $\M^{\ast}$
$$
\BB_1:=\left\{\M\in\MM_r: \|\M-\M^{\ast}\|_{\rm F}\geq \tauc\right\}\quad {\rm and}\quad \BB_2:=\left\{\M\in\MM_r: \tauc>\|\M-\M^{\ast}\|_{\rm F}\geq \taus\right\}
$$
,  where $\|\cdot\|_{\rm F}$ represents the Frobenius norm of a matrix.  Suppose the following conditions hold.  
\begin{enumerate}[1.]
\item {\bf (Dual-phase sharpness)} The function $f(\cdot): \RR^{d_1\times d_2}\mapsto \RR_+$ is said to satisfy rank-$r$ restricted $(\tauc,\taus,\muc,\mus)$ dual-phase sharpness with respect to $\M^{\ast}$ if
$$
f(\M)-f(\M^{\ast})\geq 
\begin{cases}
\muc \|\M-\M^{\ast}\|_{\rm F}, & \textrm{ for }\ \  \M\in\BB_1; \\
\mus \|\M-\M^{\ast}\|_{\rm F}^2, & \textrm{ for }\ \  \M\in\BB_2.
\end{cases}
$$

\item {\bf (Dual-phase sub-gradient bound)} The function $f(\cdot): \RR^{d_1\times d_2}\mapsto \RR_+$ has rank-r restricted $(\tauc,\taus,\Lc,\Ls)$ dual-phase sub-gradient bound with respect to $\M^{\ast}$ meaning that for any sub-gradient $\G\in\partial f(\M)$,
$$
\|\G\|_{\rm F,  r}\leq 
\begin{cases}
\Lc, & \textrm{ for }\ \ \M\in\BB_1;\\
\Ls\|\M-\M^{\ast}\|_{\rm F},& \textrm{ for }\ \ \M\in\BB_2.
\end{cases}
$$
Here, the truncated Frobenius norm $\|\G\|_{\rm F, r}:=\|{\rm SVD}_r(\G)\|_{\rm F}$ where ${\rm SVD}_r(\cdot)$ returns the best rank-$r$ approximation of a matrix.  See \cite{tong2021low}.  
\end{enumerate}
\end{assumption}
Basically,  Assumption~\ref{assump:dual-phase} dictates the distinct behaviors of  $f(\cdot)$ in two neighbourhoods of $\M^{\ast}$.  The quantity $\taus$ reflects the statistical limit while $\tauc$ is usually the rate achieved by the computational analysis in existing literature \citep{charisopoulos2019lowrank,tong2021accelerating} without assuming noise distributions.  These prior works only reveal the first phase regularity conditions,  i.e.,  on $\BB_1$.  Interestingly,  $\tauc\gg \taus$ under mild condition on noise distribution,  in which case  these prior works only deliver statistically sub-optimal estimates.  %For example,  if $\rho(\cdot)$ is the absolute loss,  $\X_i$ has i.i.d.  $N(0,1)$ entries and $\xi_i\sim N(0, \sigma^2)$,  we show that $\tauc$ is of order $O(\sigma)$ whereas $\taus$ is of order $O(\sigma(rd_1/n)^{1/2})$.  
Deriving the second phase regularity condition is challenging,  for which more precise calculations are necessary.  We remark that,  in statistics literature,  the sub-gradient bound is related to the Lipschitz continuity \citep{alquier2019estimation} and the  sharpness condition is called the {\it one-point-margin condition} \citep{elsener2018robust}.  However,  their estimators are built upon convex program so that the analysis is only made in a small neighbour of $\M^{\ast}$ where the dual-phase regularity conditions are not pivotal.  

\begin{remark}
The second phase regularity conditions are similar to those required in smooth optimization,  e.g.,  the square loss \citep{cai2021generalized,zhao2015nonconvex} and logistic loss \citep{lyu2021latent}.  For many applications (see Section~\ref{sec:app}),  the second phase occurs when the individual random noise starts to (stochastically) dominate $\|\M_l-\M^{\ast}\|_{\rm F}$.  This suggests that, when close enough to $\M^{\ast}$,  the random noise has the effect of smoothing the objective function.  Randomized smoothing has been observed in optimization and statistics literature.   See,  e.g. ,  \cite{duchi2012randomized,zhang2020edgeworth} and references therein.   
\end{remark}

The following proposition establishes the general convergence performance of Algorithm~\ref{alg:RsGrad}.  Note that the results are deterministic under Assumption~\ref{assump:dual-phase}.  Recall that $\sigma_r:=\sigma_r(\M^{\ast})$,  the smallest non-zero singular value of $\M^{\ast}$.    

\begin{proposition}\label{prop:main}
Suppose Assumption~\ref{assump:dual-phase} holds,  the initialization $\M_0$ satisfies $\fro{\M_0-\M^*}\leq c_0\sigma_r\cdot\min\{\mus^2\Ls^{-2},  \muc^{2}\Lc^{-2}\}$ for a small but absolute constant $c_0>0$,  and the initial stepsize $\eta_0\in\left[0.2\fro{\M_0-\M^*}\muc\Lc^{-2}, \ 0.3\fro{\M_0-\M^*}\muc\Lc^{-2}\right] $.   At the $l$-th iteration of Algorithm~\ref{alg:RsGrad},  
 \begin{enumerate}[(1)]
 			\item When $\fro{\M_{l}-\M^*}\geq\tauc $,  namely in phase one,  take the stepsize $\eta_{l}=\big(1-0.04\muc^2\Lc^{-2}\big)^{l}\eta_{0}$,  then we have
 			\begin{align}
 				\Vert \M_{l}-\M^{*}\Vert_{\mathrm{F}}\leq \big(1-0.04\muc^2\Lc^{-2}\big)^{l}\cdot \| \M_{0}-\M^{*}\|_{\rm F}.
 			\end{align}
 		\item When $\tauc>\fro{\M_l-\M^*}\geq\taus$,  namely in phase two,  take stepsize $\eta_l\in \big[0.125\mus\Ls^{-2},\ 0.75\mus\Ls^{-2} \big]$, then we have \begin{align}
 			\fro{\M_{l+1} -\M^{*}} \leq \left(1-\frac{\mus^2}{32\Ls^2}\right)\cdot \fro{\M_{l}-\M^{*}}.
 		\end{align}
 \end{enumerate}
\end{proposition}

By Proposition~\ref{prop:main},  the RsGrad Algorithm~\ref{alg:RsGrad} outputs an estimate $\hat\M$ within a $O(\taus)$ distance in Frobenius norm from the truth.  
As discussed above,  phase one is typical non-smooth optimization,  for which the stepsize needs to be adaptive.  For simplicity,  we apply the geometrically decay stepsize \citep{tong2021accelerating,  goffin1977convergence}.  Interestingly,  phase two is essentially smooth optimization so that a fixed stepsize suffices to yield linear convergence.  See, for instance,  \cite{cai2021generalized,wei2016guarantees} and references therein.  Note that the convergence dynamic is free of the matrix condition number,  a benefit of Riemannian-type algorithms.   Lastly,  we remark that the convergence dynamic is valid only when $\sigma_r=\Theta(\taus)$,  otherwise the initialization does not belong to either of the two phases.  

\paragraph*{Comparison with prior works.} The algorithmic dynamic of RsGrad in phase one is similar to that in \cite{charisopoulos2019lowrank} and \cite{tong2021low},  where the geometrically decay of stepsizes in phase one will eventually bring the stepsize into the level $O(\mus\Ls^{-2})$ desired by phase-two convergence.   However,  \cite{charisopoulos2019lowrank} and \cite{tong2021low} will continue to shrink the stepsize geometrically and fail to deliver a statistically optimal estimator.   In contrast,  our RsGrad sets a constant stepsize in phase two and the ultimate estimator is statistically optimal.  The difference is observed in numerical experiments.  See Section~\ref{sec:simulation} for more details.

\section{Applications}\label{sec:app}
In this section,  specific applications of Proposition~\ref{prop:main} are studied and their statistical performances are presented.  Since our major interest is on the loss functions and their effectiveness against heavy-tailed noise,  for simplicity,  we assume the measurement matrix $\X$ has i.i.d. $N(0,1)$ entries.  While this condition is relaxable to more general sub-Gaussian distributions,  they inevitably further complicate the subsequent calculations and is hence not pursued here.  

\subsection{Absolute loss with Gaussian noise}
We begin to demonstrate that the absolute loss,  though motivated for heavy-tailed noise,  can deliver a statistically optimal estimator even when the noise is Gaussian.  Assume $\xi_1,\cdots,\xi_n $ are i.i.d.  $N(0, \sigma^2)$ and take the absolute loss so that the objective function 
\begin{align}\label{eq:l1-loss}
f(\M)=\sum_{i=1}^n \big|Y_i-\langle \M,  \X_i\rangle \big|.
\end{align}
The following lemma affirms the dual-phase regularity properties for the loss function in (\ref{eq:l1-loss}). 
\begin{lemma}\label{lem:Gaussian-l1}
Assume $\xi_1,\cdots, \xi_n \stackrel{i.i.d.}{\sim} N(0,\sigma^2)$.  There exist absolute constants $C_1, C_2,  C_3, C_4,  c_1>0$ such that if $n\geq C_1 rd_1$,  then absolute loss (\ref{eq:l1-loss}) satisfies Assumption~\ref{assump:dual-phase} with probability at least $1-\exp(-c_1rd_1)$ where 
$$
\tauc=\sigma, \ \taus=C_2\sigma\Big(\frac{rd_1}{n}\Big)^{1/2},\ \muc=\frac{n}{12}, \ {\rm and}\ \mus=\frac{n}{12\sigma}
$$
, moreover,  the dual-phase sub-gradient bounds are $\Lc\leq 2 n$ and $\Ls\leq C_4n\sigma^{-1}$,  respectively.  This implies the second phase step size $\eta\asymp \sigma n^{-1}$ for Algorithm~\ref{alg:RsGrad}.  
\end{lemma}

By putting together Lemma~\ref{lem:Gaussian-l1} and  Proposition~\ref{prop:main}, we  immediately obtain the convergence and statistical performance of RsGrad Algorithm~\ref{alg:RsGrad}.  The proof is a combination of Proposition~\ref{prop:main} and Lemma~\ref{lem:Gaussian-l1},  and is hence omitted.  

\begin{theorem}\label{thm:Gaussian-l1}
Suppose the conditions of Lemma~\ref{lem:Gaussian-l1} hold.  There exist absolute constants $c_0,  c_1,  c_2,c_3\in(0,1 ), C_2>0$ such that if the initialization $\fro{\M_0-\M^*}\leq c_0\sigma_r$,  the stepsizes are chosen as in Proposition~\ref{prop:main},  then with probability at least $1-\exp(-c_1rd_1)$,  Algorithm~\ref{alg:RsGrad} has the following dynamics: 
	\begin{enumerate}[(1)]
		\item during phase one when $\|\M_{l-1}-\M^{\ast}\|_{\rm F}\geq \sigma$,   the updated estimate satisfies $\fro{\M_{l}-\M^{*}}\leq (1-c_{2})^{l}\fro{\M_{0}-\M^{*}}$;  after $l_{1} = \Theta\big(\log(\sigma_{r}/\sigma)\big)$ iterations in phase one,  it achieves the statistically sub-optimal rate $\fro{\M_{l_{1}} - \M^*} \leq \sigma $; 
		\item during phase two when $\|\M_{l_1+l}-\M^{\ast}\|_{\rm F}\leq \sigma$,  the updated estimate satisfies $\fro{\M_{l_{1}+l+1}-\M^{*}}\leq(1-c_{3})\fro{\M_{l_{1}+l}-\M^{*}}$;  after $l_{2}=\Theta\big(\log\big(n/(d_1r)\big)\big)$ iterations in phase two, it achieves the statistically optimal rate $\fro{\M_{l_{2}+l_{1}} - \M^*} \leq C_2\sigma(rd_1/n)^{1/2}$.
	\end{enumerate}
\end{theorem}

%Note that the constant $C_3$ in Lemma~\ref{lem:Gaussian-l1} can be set as $1.1(2/\pi)^{1/2}$.   
Note that,  by Proposition~\ref{prop:main},  the stepsizes are chosen as $\eta_l=(1-0.04\muc^2\Lc^{-2})^l\eta_0$ during phase one and set to $\eta_l\asymp \sigma n^{-1}$ during phase.  The initial stepsize $\eta_0$ can be as large as $O(\sigma_rn^{-1})$ depending on the initialization.  
By Theorem~\ref{thm:Gaussian-l1},  the final estimate $\hat\M$,  output by Algorithm~\ref{alg:RsGrad} with a warm initialization,  properly chosen stepsizes and after $\Theta\big(\log(\sigma_r/\sigma)+\log(n/(rd_1))\big)$ iterations,  achieves the rate $\|\hat \M -\M^{\ast}\|_{\rm F}^2=O_p(\sigma^2rd_1n^{-1})$ that is minimax optimal.  See \cite{ma2015volume,xia2014optimal} for a matching minimax lower bound.  In comparison,  prior works \citep{charisopoulos2019lowrank, tong2021low} only achieve the rate $O_p(\sigma^2)$.  The two phase dynamics are observed in numerical experiments.  See Section~\ref{sec:simulation} for more details.  

\subsection{Absolute loss with heavy-tailed noise}
In this section,  we demonstrate the effectiveness of absolute loss in handling heavy-tailed noise.  More specifically,  the following assumption of noise is required.  Denote $h_{\xi}(\cdot)$ and $H_{\xi}(\cdot)$ the density and  distribution function of the noise,   respectively.  
\begin{assumption}(Heavy-tailed noise \Romannum{1}) \label{assump:heavy-tailed}
There exists an $\varepsilon>0$ such that $\EE |\xi|^{1+\varepsilon}<+\infty$.  The noise has median zero,  i.e.,  $H_{\xi}(0)=1/2$.  Denote $\gamma=\EE|\xi|$.  There exist constants $b_0, b_1>0$ (may be dependent on $\gamma$) such that 
\begin{align*}
h_{\xi}(x)\geq b_0^{-1}, &\ \ \  \textrm{ for all } |x|\leq 30\gamma;\\
h_{\xi}(x)\leq b_1^{-1}, &\ \ \ \forall x\in \RR.
\end{align*}
\end{assumption}
By Assumption~\ref{assump:heavy-tailed},  a simple fact is $b_0\geq 30 \gamma$.  The noise is required to have a finite $1+\eps$ moment,  which is to bound $\sum_{i=1}^n |\xi_i|$.  This is fairly weak compared with existing literature.  For instance,  \cite{minsker2018sub} requires a finite second-order moment condition on noise;  \cite{fan2021shrinkage} imposes a $2+\eps$ moment condition.  The lower bound on density function is similar to that required by \cite{elsener2018robust}.  The upper bound condition on density function is also mild.  For example,  a Lipschitz distribution function ensures such a uniform upper bound.  

The dual-phase regularity condition of the absolute loss with heavy-tailed noise is guaranteed as follows. Notice if a constant factor depends on $\EE|\xi|$ or/and $\EE|\xi|^{1+\epsilon}$, a star sign is placed on top left of it.

\begin{lemma}\label{lem:heavytail-l1}
Suppose Assumption~\ref{assump:heavy-tailed} holds.  There exist absolute constants $C_1, C_2, C_3, {}^*\!c_1,c_2>0$ such that if $n\geq C_1rd_1$,  then  the absolute loss (\ref{eq:l1-loss}) satisfies Assumption~\ref{assump:dual-phase} with probability at least $1-{}^*\!c_1n^{-\eps}-\exp(-c_2rd_1)$ where 
$$
\tauc=30\gamma, \ \taus=C_2b_0\Big(\frac{rd_1}{n}\Big)^{1/2},\ \muc=\frac{n}{6}, \ {\rm and}\ \mus=\frac{n}{12b_0}
$$
,  moreover,  the dual-phase sub-gradient bounds are $\Lc\leq 2n$ and $\Ls\leq C_3nb_1^{-1}$,  respectively.  This implies  the second phase step size $\eta\asymp b_1^2(nb_0)^{-1}$ for Algorithm~\ref{alg:RsGrad}.
\end{lemma}

Viewing $b_0/\gamma$ as a constant,  then $\taus\ll \tauc$ if the sample size is large.  Equipped with Lemma~\ref{lem:heavytail-l1},  the convergence and statistical performance of RsGrad Algorithm~\ref{alg:RsGrad} under heavy-tailed noise is guaranteed by the following theorem.  Note that if a constant factor depends on $b_0$ or/and $b_1$,    a star sign is placed on top right of it.

\begin{theorem}\label{thm:heavytail-l1}
Suppose Assumption~\ref{assump:heavy-tailed} and the conditions of Lemma~\ref{lem:heavytail-l1} hold.  There exist constants $c_0^{\ast},  {}^*\!c_1,  c_2,c_3, c_4^{\ast}\in(0,1 ), C_1^{\ast},  C_2>0$ such that if the initialization $\fro{\M_0-\M^*}\leq c_0^{\ast}\sigma_{r}$,  the stepsizes are chosen as in Proposition~\ref{prop:main},  then with probability at least $1-{}^*\!c_1n^{-\eps}-\exp(-c_2rd_1)$,  Algorithm~\ref{alg:RsGrad} has the following dynamics: 
	\begin{enumerate}[(1)]
		\item during phase one when $\|\M_{l-1}-\M^{\ast}\|_{\rm F}\geq 30\gamma$,   the updated estimate satisfies $\fro{\M_{l}-\M^{*}}\leq (1-c_{3})^{l}\fro{\M_{0}-\M^{*}}$;  after $l_{1} = \Theta\big(\log(\sigma_r/\gamma)\big)$ iterations in phase one,  it achieves the rate $\fro{\M_{l_{1}} - \M^*} \leq 30\gamma $; 
		\item during phase two when $\|\M_{l_1+l}-\M^{\ast}\|_{\rm F} \leq 30\gamma$,  the updated estimate satisfies $\fro{\M_{l_{1}+l+1}-\M^{*}}\leq(1-c_{4}^{\ast})\fro{\M_{l_{1}+l}-\M^{*}}$;  after $l_{2}=\Theta\big(C_1^{\ast}\log\big(\gamma n/(d_1r)\big)\big)$ iterations in phase two, it achieves the rate $\fro{\M_{l_{2}+l_{1}} - \M^*} \leq C_2b_0(rd_1/n)^{1/2}$.
	\end{enumerate}
\end{theorem}

By Theorem~\ref{thm:heavytail-l1},  if $b_0/\gamma$ is a constant,  Algorithm~\ref{alg:RsGrad} outputs a final estimate achieving the rate $\|\hat \M-\M^{\ast}\|_{\rm F}=O_p\big(\EE |\xi|\cdot (rd_1/n)^{1/2}\big)$.  Note that Gaussian noise also satisfy Assumption~\ref{assump:heavy-tailed},  in which case the aforesaid rate has a matching minimax lower bound.  Therefore,  we claim this rate to be minimax optimal.  Moreover,  our estimator computes fast and requires only a logarithmic factor of iterations.  

\subsection{Huber loss with heavy-tailed noise}
Huber loss is prevalent in robust statistics \citep{huber1965robust,sun2020adaptive,elsener2018robust} and defined by $\rho_{H,\delta}(x):=x^2\mathbbm{1}(|x|\leq \delta)+(2\delta|x|-\delta^2)\mathbbm{1}(|x|>\delta)$ where $\delta>0$ is referred to as the robustification parameter.   Clearly,  the function $\rho_{H,\delta}$ is Lipschitz with a constant $2\delta$.  Then the loss function is given by 
\begin{align}\label{eq:loss-huber}
f(\M):=\sum_{i=1}^n \rho_{H,\delta}(Y_i-\langle \M,  \X_i\rangle)
\end{align}
Huber loss is robust to heavy-tailed noise but it turns out that a slightly different assumption is needed for our purpose.
\begin{assumption}\label{assump:heavytail-huber}(Heavy-tailed noise \Romannum{2}) 
There exists an $\eps>0$ such that $\EE |\xi|^{1+\eps}<+\infty$.  The noise has a symmetric distribution,  i.e.,  $H_{\xi}(x)=1-H_{\xi}(-x)$. Denote $\gamma=\EE|\xi|$. There exist constants $b_0, b_1$ (may be dependent on $\gamma$ and $\delta$) such that 
\begin{align*}
H_{\xi}(x+\delta)-H_{\xi}(x-\delta)\geq 2\delta b_0^{-1} &\ \ \  \textrm{ for all } |x|\leq 24\gamma+2\delta;\\
H_{\xi}(x+\delta)-H_{\xi}(x-\delta)\leq 2\delta b_1^{-1}, &\ \ \ \forall x\in \RR
\end{align*}
,  where $\delta$ is the Huber loss parameter.   
\end{assumption}   
Compared with Assumption~\ref{assump:heavy-tailed},  here the noise is required to be symmetric but the condition on density function is relaxed.  The dual-phase regularity condition of Huber loss is ensured by the following lemma. If a constant factor depends on $\EE|\xi|$ or/and $\EE|\xi|^{1+\epsilon}$, a star sign is placed on top left of it.

\begin{lemma}\label{lem:heavytail-huber}
Suppose Assumption~\ref{assump:heavytail-huber} holds.  There exist absolute constants $C_1,C_2, C_3, {}^*\!c_1,c_2>0$ such that if $n\geq C_1rd_1$,  then  the Huber loss (\ref{eq:loss-huber}) satisfies Assumption~\ref{assump:dual-phase} with probability at least $1-{}^*\!c_1n^{-\eps}-\exp(-c_2rd_1)$ where 
$$
\tauc=24\gamma+2\delta, \ \taus=C_{\delta,1}b_0\Big(\frac{rd_1}{n}\Big)^{1/2},\ \muc=\frac{\delta n}{4}, \ {\rm and}\ \mus=\frac{\delta n}{3b_0}
$$
,  moreover,  the dual-phase sub-gradient bounds are $\Lc\leq 2\delta n$ and $\Ls\leq C_{\delta,2}nb_1^{-1}$,  respectively,  where $C_{\delta,1}=C_3\cdot \max\{1, \delta^{-1}\}$ and $C_{\delta,2}=C_3\cdot \max\{1,\delta\}$.    This implies  the second phase step size $\eta\asymp b_1^2(nb_0)^{-1}$ for Algorithm~\ref{alg:RsGrad}.
\end{lemma}

Proposition~\ref{prop:main} and Lemma~\ref{lem:heavytail-huber}  lead to the following convergence and statistical performance of RsGrad Algorithm~\ref{alg:RsGrad} for Huber loss.  Similarly,  the constants dependent on $b_0$ or/and $b_1$ are marked with star.

\begin{theorem}\label{thm:heavytail-huber}
Suppose Assumption~\ref{assump:heavytail-huber} and the conditions of Lemma~\ref{lem:heavytail-huber} hold.  There exist constants $c_0^{\ast},  {}^*\!c_1,  c_2,c_3, c_4^{\ast}\in(0,1 ), C_{1,\delta}^{\ast},  C_{2}>0$ such that if the initialization $\fro{\M_0-\M^*}\leq c_0^{\ast}\sigma_r$,  the stepsizes are chosen as in Proposition~\ref{prop:main},  then with probability at least $1-{}^*\!c_1n^{-\eps}-\exp(-c_2rd_1)$,  Algorithm~\ref{alg:RsGrad} has the following dynamics: 
	\begin{enumerate}[(1)]
		\item during phase one when $\|\M_{l-1}-\M^{\ast}\|_{\rm F}\geq 24\gamma+2\delta$,   the updated estimate satisfies $\fro{\M_{l}-\M^{*}}\leq (1-c_{3})^{l}\fro{\M_{0}-\M^{*}}$;  after $l_{1} = \Theta\big(\log(\sigma_{r}/\gamma)\big)$ iterations in phase one,  it achieves the rate $\fro{\M_{l_{1}} - \M^*} \leq 24\gamma +2\delta$; 
		\item during phase two when $\|\M_{l_1+l}-\M^{\ast}\|_{\rm F}\leq 24\gamma+2\delta$,  the updated estimate satisfies $\fro{\M_{l_{1}+l+1}-\M^{*}}\leq(1-c_{4}^{\ast})\fro{\M_{l_{1}+l}-\M^{*}}$;  after $l_{2}=\Theta\big(C_1^{\ast}\log\big(\gamma n/(d_1r)\big)\big)$ iterations in phase two, it achieves the rate $\fro{\M_{l_{2}+l_{1}} - \M^*} \leq C_2\max\{1,\delta^{-1}\}\cdot b_0(rd_1/n)^{1/2}$.
	\end{enumerate}
\end{theorem}

 By Theorem~\ref{thm:heavytail-huber},  RsGrad algorithm outputs a final estimate $\hat \M$ with Frobenius-norm error $O\big(\max\{1,\delta^{-1}\}\cdot b_0(rd_1/n)^{1/2}\big)$.   Assumption~\ref{assump:heavytail-huber} implies that $b_0$ is at least lower bounded by $2(\gamma+\delta)$ and thus the error rate is lower bounded by $\max\{\delta+\gamma,1+\gamma\delta^{-1}\} (rd_1/n)^{1/2}\big)$.   This suggests an interesting transition with respect to the robustification parameter: when $\delta\leq \gamma$,  increasing $\delta$ does not affect the error rate; when $\delta\geq \max\{\gamma, 1\}$,  the error rate appreciates if $\delta$ becomes larger.    Therefore,  an appropriate choice can be $\delta\asymp \gamma$.  In contrast,  by Theorem~\ref{thm:heavytail-l1},  the absolute loss achieves an ultimate estimator with a comparable error rate but is free of the tuning parameter $\delta$.  Note that,  though absolute loss is a special case of Huber loss (i.e.,  $\delta=0$),  our Theorem~\ref{thm:heavytail-huber} is not directly applicable to the absolute loss due to proof technicality.   
Finally,  \cite{sun2020adaptive} derived the statistical performance of Huber loss for linear regression without imposing regularity conditions on noise density,  resulting into 
 an error rate that is possibly slower than the traditional $O(n^{-1/2})$ rate.

\subsection{Quantile loss with heavy-tailed noise}
Quantile loss was initially proposed by \cite{koenker1978regression} and has been a popular loss function for robust statistics  \citep{welsh1989m,koenker2001quantile,wang2012quantile}.  Denote the quantile loss by $\rho_{Q,\delta}:=\delta x\mathbbm{1}(x\geq 0)+(\delta-1)x\mathbbm{1}(x<0)$ for any $x\in \RR$ with $\delta:=\PP(\xi\leq 0)$. Notice that the function $\rho_{Q,\delta}$ is Lipschitz with a constant $\max\{\delta,1-\delta\}$. The loss function is given by \begin{align}\label{eq:loss-quantile}
	f(\M):=\sum_{i=1}^n \rho_{Q,\delta}(Y_i-\langle \M,  \X_i\rangle).
\end{align}
Absolute loss  could be viewed as a special case of quantile loss when $\PP(\xi\leq0)=1/2$,  i.e.,  noise has median zero.   Similarly,  a slightly different assumption on the heavy-tailed noise is necessary for quantile loss.  
\begin{assumption}(Heavy-tailed noise \Romannum{3}) \label{assump:heavy-tailed quantile}
	There exists an $\varepsilon>0$ such that $\EE |\xi|^{1+\varepsilon}<+\infty$.  Suppose $\delta:=H_{\xi}(0)$ lies in $(0,1)$. Denote $\gamma=\EE|\xi|$.  There exist constants $b_0, b_1>0$ (may be dependent on $\gamma$) such that
	\begin{align*}
		h_{\xi}(x)\geq b_0^{-1}, &\ \ \  \textrm{ for all } |x|\leq 15\gamma\max\{\delta,1-\delta\};\\
		h_{\xi}(x)\leq b_1^{-1}, &\ \ \ \forall x\in \RR.
	\end{align*}
\end{assumption}
When the noise has median zero,  namely $\delta=1/2$, Assumption~\ref{assump:heavy-tailed quantile} becomes identical to Assumption~\ref{assump:heavy-tailed}. The dual-phase regularity of the quantile loss is provided by the following lemma. Notice if a constant factor depends on $\EE|\xi|$ or/and $\EE|\xi|^{1+\epsilon}$, a star sign is placed on top left of it.

\begin{lemma}\label{lem:heavytail-quantile}
	Suppose Assumption~\ref{assump:heavy-tailed quantile} holds.  There exist absolute constants $C_1, C_2, C_3, {}^*\!c_1,c_2>0$ such that if $n\geq C_1rd_1$,  then  the absolute loss (\ref{eq:loss-quantile}) satisfies Assumption~\ref{assump:dual-phase} with probability at least $1-{}^*\!c_1n^{-\eps}-\exp(-c_2rd_1)$ where 
	$$
	\tauc=15\gamma\max\{\frac{1}{\delta},\frac{1}{1-\delta}\}, \ \taus=C_2b_0\Big(\frac{rd_1}{n}\Big)^{1/2},\ \muc=\frac{n}{6}\min\{\delta,1-\delta\}, \ {\rm and}\ \mus=\frac{n}{12b_0}
	$$
	,  moreover,  the dual-phase sub-gradient bounds are $\Lc\leq 2n\max\{\delta,1-\delta\}$ and $\Ls\leq C_3nb_1^{-1}$,  respectively.  This implies  the second phase step size $\eta\asymp b_1^2(nb_0)^{-1}$ for Algorithm~\ref{alg:RsGrad}.
\end{lemma}

By combining Proposition~\ref{prop:main} with Lemma~\ref{lem:heavytail-quantile},  we get the convergence dynamics and statistical accuracy of RsGrad Algorithm~\ref{alg:RsGrad} for quantile loss.  Same as previous,  those constants dependent on $b_0$ or/and $b_1$ are marked with a star at top.

\begin{theorem}\label{thm:heavytail-quantile}
	Suppose Assumption~\ref{assump:heavy-tailed quantile} and the conditions of Lemma~\ref{lem:heavytail-quantile} hold.  There exist constants $c_0^{\ast},  {}^*\!c_1,  c_2,c_3, c_4^{\ast}\in(0,1 ), C_1^{\ast},  C_2>0$ such that if the initialization $\fro{\M_0-\M^*}\leq c_0^{\ast}\sigma_r$,  the stepsizes are chosen as in Proposition~\ref{prop:main},  then with probability at least $1-{}^*\!c_1n^{-\eps}-\exp(-c_2rd_1)$,  Algorithm~\ref{alg:RsGrad} has the following dynamics: 
	\begin{enumerate}[(1)]
		\item during phase one when $\|\M_{l-1}-\M^{\ast}\|_{\rm F}\geq 15\gamma\max\{\delta^{-1}, (1-\delta)^{-1}\}$,   the updated estimate satisfies $\fro{\M_{l}-\M^{*}}\leq (1-c_{3}\min\{\delta(1-\delta)^{-1}, (1-\delta)\delta^{-1}\})^{l}\fro{\M_{0}-\M^{*}}$;  after $l_{1} = \Theta\big(\log(\min\{\delta,1-\delta\}\cdot\sigma_{r}/\gamma)\big)$ iterations in phase one,  it achieves the rate $\fro{\M_{l_{1}} - \M^*} \leq 15\gamma\max\{\delta^{-1}, (1-\delta)^{-1}\} $; 
		\item during phase two when $\|\M_{l_1+l}-\M^{\ast}\|_{\rm F}\leq 15\gamma\max\{\delta^{-1}, (1-\delta)^{-1}\}$,  the updated estimate satisfies $\fro{\M_{l_{1}+l+1}-\M^{*}}\leq(1-c_{4}^{\ast})\fro{\M_{l_{1}+l}-\M^{*}}$;  after $l_{2}=\Theta\big(C_1^{\ast}\log\big(\max\{\delta^{-1}, (1-\delta)^{-1}\}\cdot\gamma n/(d_1r)\big)\big)$ iterations in phase two, it achieves the rate $\fro{\M_{l_{2}+l_{1}} - \M^*} \leq C_2b_0(rd_1/n)^{1/2}$.
	\end{enumerate}
\end{theorem}
Parameter $\delta$ is kept in all results to show the explicit dependence on $\delta$. Theorem~\ref{thm:heavytail-quantile} shows the final estimator $\hat{\M}$ output by RsGrad Algorithm~\ref{alg:RsGrad} for quantile loss achieves Frobenius-norm error $O(C_2b_0(d_1r/n)^{1/2})$.

\section{Initialization and Algorithmic Parameters Selection}\label{sec:discussion}

\paragraph*{Initialization.} Convergence of Algorithm~\ref{alg:RsGrad} crucially relies on the warm initialization $\M_0$.  Towards that end,  one can simply apply the shrinkage low-rank approximation as in \cite{fan2021shrinkage} where a finite $2+\eps$ moment condition on noise is required.  Here,  for simplicity,  we investigate the performance of vanilla low-rank approximation.  

\begin{theorem}(Initialization Guarantees)\label{thm:initialization}
	Suppose there exists some constant $\varepsilon\in(0,1]$ such that $\gamma_1:=\EE|\xi|^{1+\varepsilon}<+\infty$. Initialize $\M_0:=\operatorname{SVD}_r(n^{-1}\sum_{i=1}^{n}\X_{i}Y_i)$. 
For any small $c_0>0$,  there exist constants $C, c_1,c_2,c_3>0$ depending only on $c_0$ such that if the sample size  
$$n>C\max\big\{\kappa^2d_1r^2\log d_1, (d_1r)^{\frac{1+\epsilon}{2\epsilon}}\sigma_r^{-\frac{1+\epsilon}{\epsilon}}(\gamma_1\log d_1)^{\frac{1}{\epsilon}}\big\},$$
% (\gamma\sigma_r^{-1}\sqrt{d_1r})^{1+\eps^{-1}}\log^{\eps^{-1}} d_1\},$$ 
 with probability over $1-c_1d_1^{-1}-c_2\log^{-1} d_1-c_3e^{-d_1}$, the initialization satisfies $\fro{\M_0-\M^*}\leq c_0\sigma_{r}$,  where $\kappa:=\sigma_1(\M^{\ast})\sigma_r^{-1}(\M^{\ast})$ denotes the condition number.  
\end{theorem}

Compared with \cite{fan2021shrinkage},  Theorem~\ref{thm:initialization} only requires a finite $1+\eps$ moment condition and a similar sample size condition.  However,  the vanilla low-rank approximation demands a much larger signal-to-noise ratio condition.    Since the shrinkage-based approaches \citep{fan2021shrinkage,minsker2018sub} are already effective to provide warm initializations,  we spare no additional efforts to improve Theorem~\ref{thm:initialization}. 

\paragraph*{Selection of initial stepsize.} According to Proposition~\ref{prop:main},  the initial stepsize depends on $\|\M_0-\M^{\ast}\|_{\rm F}$,  $\muc$ and $\Lc$.  Fortunately,  as shown in Lemma~\ref{lem:Gaussian-l1} - \ref{lem:heavytail-quantile} for specific applications,  the quantities $\muc$ and $\Lc$ are often of order $O(n)$ up to a factor of loss-related parameters.  It suffices to obtain an estimate of $\|\M_0-\M^{\ast}\|_{\rm F}$.  Towards that end,  one appropriate method is to take the average $n^{-1}\sum_{i=1}^n |Y_i-\langle \M_0, \X_i\rangle|$,  which,  under mild conditions,  is of the same order of $\|\M_0-\M^{\ast}\|_{\rm F}$ with high probability.   Besides,  a simpler way is to directly use the operator norm $c\|\M_0\|$ to replace $\|\M_0-\M^{\ast}\|_{\rm F}$ where $c>0$ is a tuning parameter,  if we believe that $\|\M_0-\M^{\ast}\|_{\rm F}$ is indeed of order $O(\sigma_r)$.

\paragraph*{Determine the phase.} Due to the geometrically decay of stepsizes during phase one,  after some iterations,  the stepsize will enter the level $O(\mus\Ls^{-2})$ desired by phase two convergence.  On the other hand,  if phase one iterations continue,  the stepsize will diminish fast and the value of objective function becomes stable.  This is indeed observed in numerical experiments.  See Figure~\ref{fig:conv-Gaussian} and Figure~\ref{fig:conv-t2} in Section~\ref{sec:simulation}.  Therefore,  once the stepsize falls below a pre-chosen small threshold,  the phase two iterations can be initiated and the stepsizes are fixed afterwards until convergence.

\section{Extension to Robust Low-rank Tensor Estimation}\label{sec:tensor}
We now extend RsGrad to low-rank tensor estimation equipped with a robust loss function and investigate its performance under Gaussian or heavy-tailed noise. Without loss of generality, only the absolute loss is considered here.

%, namely, from two-dimension matrix to multidimensional array. Refer readers to \

	\paragraph*{Preliminaries in Tucker tensors.}
	An $m$-th order tensor is an $m$-dimensional array. 
    For an $m$-th order tensor $\M\in \RR^{d_1\times d_2\times \cdots \times d_m}$, denote its mode-$j$ matricization as $\M_{(j)}\in \RR^{d_j\times d_j^{-}}$, where $d_j^{-}:=\prod_{l\neq j} d_l$. 
    The mode-$j$ marginal multiplication between $\M$ and a matrix $\U^{\top}\in\RR^{r_j\times d_j}$ results into an $m$-th order tensor of size $d_1\times\cdots d_{j-1}\times r_j\times d_{j+1}\cdots d_m$, whose elements are $(\M\times_j \U^{\top})_{i_1\cdots i_{j-1}l i_{j+1}\cdots i_m}:=\sum_{i_j=1}^{d_j} \M_{i_1\cdots i_{j-1} i_{j} i_{j+1}\cdots i_m}\U_{i_jl}.$  A simple fact is $\N_{(j)}=\U^{\top} \M_{(j)}$. 
     There exist multiple definitions of tensor ranks. We focus on Tucker ranks and the associated Tucker decomposition \citep{tucker1966some}.  $\M$ is said to have Tucker rank $\r:=(r_1,r_2,\cdots, r_m)$ if its mode-$j$ matricization has rank $r_j$, i.e., $r_j=\text{rank}(\M_{(j)})$. 
     Then $\M$ admits the Tucker decomposition $\M=\C\cdot \llbracket \U_1,\cdots,\U_m\rrbracket:=\C\times_1\U_1\times_2\cdots\times_m\U_m$ where the core tensor $\C$ is of size $r_1\times\cdots\times r_m$ and $\U_j\in\RR^{d_j\times r_j}$ has orthonormal columns. 
Denote $\MM_{\r}$ be set of all tensors with Tucker rank at most $\r$, namely, 
     $\MM_{\r}=\{\M\in\RR^{d_1\times\dots\times d_m}: \rank(\M_{(j)})\leq r_j,\, \text{ for all }j =1,\cdots,m\}.$ Interested readers are suggested to refer to \cite{kolda2009tensor,de2008tensor,de2000multilinear} for more details about Tucker ranks and Tucker decomposition. 
     
     \hspace{0.1cm}
     
Let $\{(\X_i, Y_i)\}_{i=1}^n$ be i.i.d. observations satisfying the trace regression model (\ref{eq:tr_model}) with $\X_i\in\RR^{d_1\times\cdots\times d_m}$ that, for simplicity, is assumed to have i.i.d. $N(0, 1)$ entries. Suppose the underlying tensor $\M^{\ast}\in\MM_{\r}$. Our goal is to reliably estimate $\M^{\ast}$ under possibly heavy-tailed noise and using as few observations as possible. 
Towards that end, we aim to solve 
\begin{align}\label{eq:tensor:loss}
	\hat{\M}:=\underset{\M\in\MM_\r}{\arg\min}\ f(\M),  \quad \textrm{ where } f(\M) := \sum_{i=1}^n\rho\big(\langle \M,  \X_i\rangle-Y_i\big)
\end{align}
with $\rho(\cdot):\RR\mapsto\RR_+$ a proper chosen robust loss function. RsGrad is readily applicable to solve program~(\ref{eq:tensor:loss}) except that now it operates on the tensor Riemannian and its sub-gradient needs to be taken carefully. 

The tensor-version RsGrad is presented in Algorithm~\ref{alg:tensor:RsGrad}.  Note that the Riemannian sub-gradient is explicitly computed as follows. Let $\M_l=\C_l \cdot \llbracket \U_1^{(l)}, \cdots, \U_m^{(l)} \rrbracket$ be its Tucker decomposition. Given the vanilla sub-gradient $\G_l$, we have 
$$
\calP_{\TT_l}(\G_l)=\G_{l} \cdot \llbracket \U_1^{(l)}\U_1^{(l)\top},\cdots, \U_m^{(l)}\U_m^{(l)\top}\rrbracket+\sum_{i=1}^{m}\C_l\times_{j\in[m]\backslash i}\U_{j}^{(l)}\times_{i}\dot{\U}_i^{(l)}
$$
where $\dot{\U}_i^{(l)}$ is defined by $(\I-\U_i^{(l)}\U_i^{(l)\top})(\G_l)_{(i)}(\otimes_{j\neq i} \U_j^{l} )(\C_l)_{(i)}^{\dagger}$. The Tucker rank of $\calP_{\TT_l}(\G_l)$ is at most $2\r$. The retraction step in Algorithm~\ref{alg:tensor:RsGrad} relies on high-order SVD (HOSVD, \cite{de2000multilinear}, \cite{xia2019sup}),  which is obtained by $\text{HOSVD}_{\r}(\M):=\M\cdot \llbracket \U_1\U_1^{\top},\cdots, \U_m\U_m^{\top}\rrbracket$ with $\U_j$ being the top-$r_j$ left singular vectors of $\M_{(j)}$ for all $j\in[m]$.

\begin{algorithm}
	\caption{Riemannian Sub-gradient Descent for Tensor (RsGrad)}\label{alg:tensor:RsGrad}
	\begin{algorithmic}
		\STATE{\textbf{Input}: observations $\{(\X_i, Y_i)\}_{i=1}^n$,  max iterations $\lmax$,  step sizes $\{\eta_l\}_{l=0}^{\lmax}$.}
		\STATE{Initialization: $\M_0\in\MM_\r$}
		\FOR{$l = 0,\ldots,\lmax$}
		\STATE{Choose a vanilla subgradient:  $\G_l\in\partial f(\M_l)$}
		\STATE{Compute Riemannian sub-gradient: $\wt\G_l = \calP_{\TT_l}(\G_l)$}
		\STATE{Retraction to $\MM_\r$: $\M_{l+1} = \text{HOSVD}_\r(\M_l - \eta_{l}\wt\G_l)$}
		\ENDFOR
		\STATE{\textbf{Output}: $\hat\M=\M_{\lmax}$}
	\end{algorithmic}
\end{algorithm}

% and has Tucker decomposition $\M^*=\C^* \times_{1} \U_1^* \times_2\dots\times_m \U_m^*$, with core tensor $\C^*\in\RR^{r_1\times\dots\times r_m}$ and orthogonal matrices $\U_{j}^*\in\OO_{d_j, r_j}$.
%We introduce operator $\text{SVD}_{\r}(\cdot)$: $\text{SVD}_{\r}(\M):=\underset{\hat{\M}\in\MM_{\r}}{\arg\min}\fro{\M-\hat\M},$ which returns best rank $\r$ approximation and is discussed in Lemma~\ref{teclem:tensor:partial F norm}. Denote $\underline{\lambda}:=\min_{i}\op{\sigma_{r_i}(\M_{(i)}^*)}$.

\subsection{General Convergence Performance}

Similarly, we establish the general convergence performance of Algorithm~\ref{alg:tensor:RsGrad} under the dual-phase regularity conditions. The following assumption is almost identical to Assumption~\ref{assump:dual-phase} except that the matrix SVD is replaced by tensor low-rank approximation. 
\begin{assumption}\label{assump:tensor:dual-phase}(Dual-phase regularity conditions for tensor)
	Let $\tauc>\taus>0$ and define two regions around the truth $\M^{\ast}$
	$$
	\BB_1:=\left\{\M\in\MM_\r: \|\M-\M^{\ast}\|_{\rm F}\geq \tauc\right\}\quad {\rm and}\quad \BB_2:=\left\{\M\in\MM_\r: \tauc>\|\M-\M^{\ast}\|_{\rm F}\geq \taus\right\}
	$$
	,  where $\|\cdot\|_{\rm F}$ represents the Frobenius norm of a tensor.  Suppose the following conditions hold.  
	\begin{enumerate}[1.]
		\item {\bf (Dual-phase sharpness)} The function $f(\cdot): \RR^{d_1\times \cdots\times d_m}\mapsto \RR_+$ is said to satisfy rank-$\r$ restricted $(\tauc,\taus,\muc,\mus)$ dual-phase sharpness with respect to $\M^{\ast}$ if
		$$
		f(\M)-f(\M^{\ast})\geq 
		\begin{cases}
			\muc \|\M-\M^{\ast}\|_{\rm F}, & \textrm{ for }\ \  \M\in\BB_1; \\
			\mus \|\M-\M^{\ast}\|_{\rm F}^2, & \textrm{ for }\ \  \M\in\BB_2.
		\end{cases}
		$$
		
		\item {\bf (Dual-phase sub-gradient bound)} The function $f(\cdot): \RR^{d_1\times \cdots\times d_m}\mapsto \RR_+$ has rank-r restricted $(\tauc,\taus,\Lc,\Ls)$ dual-phase sub-gradient bound with respect to $\M^{\ast}$ meaning that for any sub-gradient $\G\in\partial f(\M)$,
		$$
		\|\G\|_{\rm F,  2\r}\leq 
		\begin{cases}
			\Lc, & \textrm{ for }\ \ \M\in\BB_1;\\
			\Ls\|\M-\M^{\ast}\|_{\rm F},& \textrm{ for }\ \ \M\in\BB_2.
		\end{cases}
		$$
		The truncated Frobenius norm $\|\G\|_{\rm F, 2\r}:=\sup_{\U_j} \|\G\cdot \llbracket \U_1^{\top},\cdots, \U_m^{\top}\rrbracket\|_{\rm F}$ where $\U_j\in\RR^{d_1\times 2r_j}$ has orthonormal columns.  %Truncated Frobenius norm for tensor is discused in Lemma~\ref{teclem:tensor:partial F norm}.  
	\end{enumerate}
\end{assumption}

Under Assumption~\ref{assump:tensor:dual-phase} and with a warm initialization, the following Proposition~\ref{prop:tensor:main} shows that RsGrad converges linearly and attains the statistical error rate $O(\taus)$. The signal strength of $\M^{\ast}$ is defined by $\underline{\lambda}:=\min_{j=1,\cdots,m}\{\sigma_{r_j}(\M_{(j)}^*)\}$ that is the smallest nonzero singular value among all matricizations of $\M^*$.

\begin{proposition}\label{prop:tensor:main}
Suppose Assumption~\ref{assump:tensor:dual-phase} holds,  the initialization $\M_0$ satisfies 
	$$\fro{\M_0-\M^*}\leq c_0\underline{\lambda}\cdot\min\left\{\frac{1}{m+1}\frac{\mus^2}{\Ls^2},\ \frac{1}{m+1}\frac{\muc^2}{\Lc^2},\ \frac{1}{m(m+3)}\frac{\mus}{\Ls},\ \frac{1}{m(m+3)}\frac{\muc}{\Lc}\right\}
	$$
for a small but absolute constant $c_0>0$,  and the initial stepsize 
$$
\eta_0\in\left[0.25(m+1)^{-1}\fro{\M_0-\M^*}\muc\Lc^{-2}, \ 0.75(m+1)^{-1}\fro{\M_0-\M^*}\muc\Lc^{-2}\right]. 
$$  
At the $l$-th iteration of Algorithm~\ref{alg:tensor:RsGrad},  
	\begin{enumerate}[(1)]
		\item When $\fro{\M_{l}-\M^*}\geq\tauc $,  namely in phase one,  take the stepsize $\eta_{l}=\big(1-\frac{1}{16(m+1)}\frac{\muc^2}{\Lc^2}\big)^{l}\eta_{0}$,  then we have
		\begin{align}
			\Vert \M_{l}-\M^{*}\Vert_{\mathrm{F}}\leq \left(1-\frac{1}{16(m+1)}\cdot\frac{\muc^2}{\Lc^2}\right)^{l}\cdot \| \M_{0}-\M^{*}\|_{\rm F}.
		\end{align}
		\item When $\tauc>\fro{\M_l-\M^*}\geq\taus$,  namely in phase two,  take stepsize $\eta_l\in \big[0.25(m+1)^{-1}\mus\Ls^{-2},\ 0.75(m+1)^{-1}\mus\Ls^{-2} \big]$, then we have \begin{align}
			\fro{\M_{l+1} -\M^{*}} \leq \left(1-\frac{1}{16(m+1)}\cdot\frac{\mus^2}{\Ls^2}\right)\cdot \fro{\M_{l}-\M^{*}}.
		\end{align}
	\end{enumerate}
\end{proposition}
Note that $m$ is kept explicitly in Proposition~\ref{prop:tensor:main} to underscore its role in the convergence of RsGrad algorithm.  The convergence dynamic of RsGrad during phase one is similar to that of ScaledSM \citep{tongaccelerating}, both of which achieve the computational accuracy $O(\tauc)$. However, under the dual-phase regularity condition, RsGrad presents a two-phase dynamic which converges linearly with a constant step size and eventually achieves a statistical error rate $O(\taus)$, which is usually significantly smaller than $O(\tauc)$.

\subsection{Applications}
We now specifically investigate the statistical performances of RsGrad for absolute loss under Gaussian noise and heavy-tailed noise. The results of Huber loss and quantile loss can be similarly established and skipped here. 

\subsubsection{Absolute loss with Gaussian noise}
Suppose that the noise $\xi_{1},\cdots,\xi_{n}$ are i.i.d. $N(0,\sigma^2)$. Take the absolute loss and get the objective function $
	f(\M)=\sum_{i=1}^n |Y_i-\inp{\M}{\X_{i}}|.
$ Denote $\textsf{DoF}_m:=2^m\cdot r_1r_2\cdots r_m+2\sum_{j=1}^{m}d_jr_j$, reflecting the model complexity. If $m$ is a constant, then $\textsf{DoF}_m$ can be treated as the degree of freedom. The dual-phase regularity condition is guaranteed by Lemma~\ref{lem:tensor:Gaussian-l1}. 

\begin{lemma}\label{lem:tensor:Gaussian-l1}
	Assume $\xi_1,\cdots, \xi_n \stackrel{i.i.d.}{\sim} N(0,\sigma^2)$.  There exist absolute constants $C_1, C_2,  C_3, C_4,  c_1>0$ such that if $n\geq C_1 \cdot \textsf{DoF}_m$,  then absolute loss satisfies Assumption~\ref{assump:tensor:dual-phase} with probability at least $1-\exp(-c_1\cdot\textsf{DoF}_m)$ where 
	$$
	\tauc=\sigma, \ \taus=C_2\sigma\Big(\frac{\textsf{DoF}_m}{n}\Big)^{1/2},\ \muc=\frac{n}{12}, \ {\rm and}\ \mus=\frac{n}{12\sigma}
	$$
	, moreover,  the dual-phase sub-gradient bounds are $\Lc\leq 2 n$ and $\Ls\leq C_4n\sigma^{-1}$,  respectively.  This implies the second phase step size $\eta\asymp \sigma n^{-1}$ for Algorithm~\ref{alg:tensor:RsGrad}.
\end{lemma}

The following theorem presents the convergence dynamic and statistical accuracy of Algorithm~\ref{alg:tensor:RsGrad} for absolute loss under Gaussian noise.

\begin{theorem}\label{thm:tensor:Gaussian-l1}
	Suppose the conditions of Lemma~\ref{lem:tensor:Gaussian-l1} hold.  There exist absolute constants $c_0,  c_1,  c_2,c_3\in(0,1 ), C_2>0$ such that if the initialization $\fro{\M_0-\M^*}\leq c_0m^{-1}(m+3)^{-1}\underline{\lambda}$,  the stepsizes are chosen as in Proposition~\ref{prop:tensor:main},  then with probability at least $1-\exp(-c_1\cdot \textsf{DoF}_m)$,  Algorithm~\ref{alg:tensor:RsGrad} exhibits the following dynamics: 
	\begin{enumerate}[(1)]
		\item during phase one when $\|\M_{l-1}-\M^{\ast}\|_{\rm F}\geq \sigma$,   the updated estimate satisfies $\fro{\M_{l}-\M^{*}}\leq (1-c_{2})^{l}\fro{\M_{0}-\M^{*}}$;  after $l_{1} = \Theta\big(\log(\underline{\lambda}/\sigma)\big)$ iterations in phase one,  it achieves the statistically sub-optimal rate $\fro{\M_{l_{1}} - \M^*} \leq \sigma $; 
		\item during phase two when $\|\M_{l_1+l}-\M^{\ast}\|_{\rm F}\leq \sigma$,  the updated estimate satisfies $\fro{\M_{l_{1}+l+1}-\M^{*}}\leq(1-c_{3})\fro{\M_{l_{1}+l}-\M^{*}}$;  after $l_{2}=\Theta\big(\log\big(n/\textsf{DoF}_m\big)\big)$ iterations in phase two, it achieves the statistically optimal rate $\fro{\M_{l_{2}+l_{1}} - \M^*} \leq C_2\sigma(\textsf{DoF}_m\cdot n^{-1})^{1/2}$.
	\end{enumerate}
\end{theorem}
By Theorem~\ref{thm:tensor:Gaussian-l1}, Algorithm~\ref{alg:tensor:RsGrad} outputs an estimator achieving statistical rate $\fro{\hat{\M}-\M^*}^2=O_p(\sigma^2\cdot \textsf{DoF}_m\cdot n^{-1})$. The rate matches the existing ones attained by the square loss \citep{cai2021generalized,han2022optimal, zhang2020islet}. Theorem~\ref{thm:tensor:Gaussian-l1} only requires a sample complexity $O(\textsf{DoF}_m)$ to ensure the algorithmic convergence. Oftentimes, the sample size requirement is more stringent to guarantee a warm initialization. See Section~\ref{sec:tensor_init}. 

\subsubsection{Absolute loss with heavy-tailed noise}
Now consider the performance of absolute loss under heavy-tailed noise. 
The following lemma certifies the dual-phase regularity under a finite $1+\eps$ moment condition on noise. Hereafter, a star sign on the top left of a constant indicates that it depends on the $\EE|\xi|$ or/and $\EE |\xi|^{1+\eps}$. 

\begin{lemma}\label{lem:tensor:heavytail-l1}
	Suppose Assumption~\ref{assump:heavy-tailed} holds.  There exist absolute constants $C_1, C_2, C_3, {}^*\!c_1,c_2>0$ such that if $n\geq C_1\cdot \textsf{DoF}_m$,  then  the absolute loss satisfies Assumption~\ref{assump:tensor:dual-phase} with probability at least $1-{}^*\!c_1n^{-\eps}-\exp(-c_2\cdot \textsf{DoF}_m)$ where 
	$$
	\tauc=30\gamma, \ \taus=C_2b_0\Big(\frac{\textsf{DoF}_m }{n}\Big)^{1/2},\ \muc=\frac{n}{6}, \ {\rm and}\ \mus=\frac{n}{12b_0}
	$$
	,  moreover,  the dual-phase sub-gradient bounds are $\Lc\leq 2n$ and $\Ls\leq C_3nb_1^{-1}$,  respectively.  This implies  the second phase step size $\eta\asymp b_1^2(nb_0)^{-1}$ for Algorithm~\ref{alg:tensor:RsGrad}.
\end{lemma}

 Proposition~\ref{prop:tensor:main} and Lemma~\ref{lem:tensor:heavytail-l1} immediately imply the convergence dynamic and statistical error as presented in  the following theorem. Similarly, constant factors with a star sign on its top right corner depends on $b_0$ or/and $b_1$.
 
\begin{theorem}\label{thm:tensor:heavytail-l1}
	Suppose Assumption~\ref{assump:heavy-tailed} and the conditions of Lemma~\ref{lem:tensor:heavytail-l1} hold.  There exist constants $c_0^{\ast},  {}^*\!c_1,  c_2,c_3, c_4^{\ast}\in(0,1 ), C_1^{\ast},  C_2>0$ such that if the initialization $\fro{\M_0-\M^*}\leq c_0^{\ast}m^{-1}(m+3)^{-1}\underline{\lambda}$,  the stepsizes are chosen as in Proposition~\ref{prop:tensor:main},  then with probability at least $1-{}^*\!c_1n^{-\eps}-\exp(-c_2\cdot \textsf{DoF}_m)$,  Algorithm~\ref{alg:tensor:RsGrad} has the following dynamics: 
	\begin{enumerate}[(1)]
		\item during phase one when $\|\M_{l-1}-\M^{\ast}\|_{\rm F}\geq 30\gamma$,   the updated estimate satisfies $\fro{\M_{l}-\M^{*}}\leq (1-c_{3})^{l}\fro{\M_{0}-\M^{*}}$;  after $l_{1} = \Theta\big(\log(\underline{\lambda}/\gamma)\big)$ iterations in phase one,  it achieves the rate $\fro{\M_{l_{1}} - \M^*} \leq 30\gamma $; 
		\item during phase two when $\|\M_{l_1+l}-\M^{\ast}\|_{\rm F}\leq 30\gamma$,  the updated estimate satisfies $\fro{\M_{l_{1}+l+1}-\M^{*}}\leq(1-c_{4}^{\ast})\fro{\M_{l_{1}+l}-\M^{*}}$;  after $l_{2}=\Theta\big(C_1^{\ast}\log\big(\gamma n/\textsf{DoF}_m\big)\big)$ iterations in phase two, it achieves the rate $\fro{\M_{l_{2}+l_{1}} - \M^*} \leq C_2b_0(\textsf{DoF}_m\cdot n^{-1})^{1/2}$.
	\end{enumerate}
\end{theorem}

By Theorem~\ref{thm:tensor:heavytail-l1}, as long as the noise has a finite $1+\eps$ moment, RsGrad outputs an estimator achieving the statistical rate $\|\hat\M-\M^{\ast}\|_{\rm F}^2=O_p(b_0^2 \textsf{DoF}_m\cdot n^{-1})$. If $m$ is a constant, this rate is proportional to the degree of freedom and optimal in terms of tensor dimensions. To our best knowledge, this is the first result of this kind for robust low-rank tensor regression under heavy-tailed noise.

\subsection{Initialization by Shrinkage-based  Second Order Moment}\label{sec:tensor_init}
The convergence Theorems~\ref{thm:tensor:Gaussian-l1} and ~\ref{thm:tensor:heavytail-l1} relies crucially on the warm initialization. With heavy-tailed noise, a simple spectral initialization, say, by HOSVD, has limited performances. Adapting from the ideas in  \cite{cai2021provable,xia2017polynomial,xia2021statistically}, we propose a new initialization by shrinkage-based second order moment method. For each response $Y_i$, define its shrinkage by $\tilde Y_i := \text{sign}(Y_i)(|Y_i|\vee \tau)$ with a threshold $\tau > 0$ to be chosen later.

With the truncated responses, construct the second-order U-statistic :
$$
\tilde\N_j = \frac{1}{n(n-1)}\sum_{1\leq i\neq i'\leq n}\tilde Y_i\tilde Y_{i'}(\X_{i(j)}\X_{i'(j)}^\top + \X_{i'(j)}\X_{i(j)}^\top),\quad j=1,\ldots,m
$$
We will prove that $\tilde{\N}_j$ is a good estimate of $\M_{(j)}^*\M_{(j)}^{*\top}$ with a properly chosen $\tau$ and a reasonably large sample size. The top-$r_j$ left singular vectors of $\tilde\N_j$, denoted by $\U_j^{(0)}$, serve as a warm initial estimate of the mode-$j$ singular vectors of $\M^{\ast}$. Then the initial core tensor is obtained by minimizing the sum of squares that admits an explicit form by
$$
\C^{(0)}=\left(\sum_{i=1}^{n}\text{vec}(\X_{i}\times_{j\in[m]}\U_{j}^{(0)\top})\cdot \text{vec}(\X_{i}\times_{j\in[m]}\U_{j}^{(0)\top})^{\top}\right)^{\dagger}\cdot \text{vec}\left(\sum_{i=1}^{n}Y_i\X_{i}\times_{j\in[m]}\U_{j}^{(0)\top}\right).
$$
Finally, the initialization for Algorithm~\ref{alg:tensor:RsGrad} is constructed by $\M_0=\C^{(0)}\cdot \llbracket \U_1^{(0)},\cdots, \U_m^{(0)}\rrbracket$.  

By assuming a finite $2+\eps$ moment condition, the following theorem confirms the closeness between $\M_0$ and $\M^{\ast}$. 
For ease of exposition, denote $\dmax = \max_{i=1}^md_i$, $\rmax = \max_{i=1}^m r_i$, $d^* = d_1\cdots d_m$ and $r^{\ast}=r_1\cdots r_m$. Recall that $\textsf{DoF}_m=2^mr^{\ast}+2\sum_{j}r_jd_j$ and the signal strength $\mins$. Denote $\maxs:=\max_{1\leq j\leq m} \|\M^{\ast}_{(j)}\|$ and the tensor condition number $\kappa:=\maxs \mins^{-1}$. 
\begin{theorem}\label{thm:init:tensor} 
Let $c_0\in (0,1)$ be a constant. 
Suppose $\op{\xi}_{2+\eps} := \left(\EE|\xi|^{2+\eps}\right)^{1/(2+\eps)}<+\infty$ for some $\eps>0$ and the following conditions hold:
	\begin{itemize}
		\item[(1)] sample size requirement: $n\geq  C_1 m^2\kappa^6\rmax^2(d^{\ast})^{1/2}\log\dmax $;
		\item[(2)] signal to noise ratio: $\mins\geq  C_2m^{1/2}\kappa^{1/2}(\rmax\log \dmax)^{1/4} \frac{(d^{\ast})^{1/4}}{n^{1/2}} \op{\xi}_{2+\eps}$,
	\end{itemize}
where $C_1, C_2>0$ are some constants depending only on $c_0$. There exist absolute constants $c_1,c_2,c_3>0$ depending only on $c_0$, such that, if $\tau = n^{1/2}{(d^*)^{-1/4}}(\sqrt{\rmax}\maxs + \op{\xi}_{2+\eps})$, then with probability exceeding $1-c_1\exp(-\textsf{DoF}_m)-c_2n^{-\min\{\eps/2,1\}}-c_3m\dmax^{-10}$, the initialization $\M_0$ satisfies
	$$
	\fro{\M_0-\M^*}\leq c_0\mins.
	$$
\end{theorem}

If $m, \bar r$ and $\kappa$ are all $O(1)$, Theorem~\ref{thm:init:tensor} requires a sample size $(d^{\ast})^{1/2}\log\bar d$,  which matches the best known result in existing literature \citep{han2022optimal, tongaccelerating,xia2021statistically}, be it noiseless or under Gaussian noise. The SNR condition becomes $\mins/\|\xi\|_{2+\xi}\gtrsim n^{-1/2}(d^{\ast}\log \bar d)^{1/4}$ that reduces to the typical SNR condition under Gaussian noise in the aforementioned prior works. These conditions are thus likely minimal but here only a finite $2+\eps$ moment is necessary.

\section{Numerical Experiments}\label{sec:simulation}
\subsection{Low-rank Matrix Estimation}
In this section,  numerical simulation results are presented which showcase the dual-phase convergence of Algorithm~\ref{alg:RsGrad} and demonstrate its merit over competing methods \citep{tong2021low,elsener2018robust,candes2011tight}.  Both RsGrad and ScaledSM \citep{tong2021low} require a careful design of stepsize schedule.  
Practically,  for ScaledSM and phase one of RsGrad,  we initialize and update the stepsizes by  $\eta_0 = c_1 \|\M_0\|n^{-1}$ and $\eta_l=q\cdot \eta_{l-1},$ where $c_1>0$, $0<q<1$ are parameters to be specified and $\|\cdot\|$ denotes the matrix operator norm.   As \cite{tong2021low} suggests,  a great choice is $q=0.91$.  For a fair comparison,  all the following  simulations are implemented with $q=0.91$.  As suggested by Proposition~\ref{prop:main},  RsGrad has a second phase convergence,  for which we set the stepsize at $\eta_l=c_2\EE|\xi| n^{-1}$ where $c_2>0$ is the parameter to be tuned.  Note that we assume $\EE|\xi|$ is known for simplicity.  In practice,  $\EE|\xi|$ is unknown and can be estimated by $n^{-1}\sum_{i=1}^n|Y_i-\langle \M_{l_1},\X_i\rangle|$ where $\M_{l_1}$ is the output after phase one iterations.  
Following \cite{tong2021low},  we define signal-to-noise ratio by $\text{SNR}:=20\log_{10}\big(\fro{\M^*}/ \EE\vert\xi\vert\big)$.

\paragraph*{Convergence dynamics}
We compare the convergence dynamics between RsGrad and ScaledSM \citep{tong2021low}.  Without loss of generality,  we only experiment RsGrad with absolute loss  (RsGrad-$\ell_1$) and Huber loss (RsGrad-Huber).  ScaledSM is proposed only for absolute loss.  The relative error is calculated by $\|\M_l-\M^{\ast}\|_{\rm F}\|\M^{\ast}\|_{\rm F}^{-1}$. 

We fix dimension $d_1=d_2=80$ and rank $r=5$.  The first set of simulations is to showcase the convergence performance of RsGrad and ScaledSM under Gaussian noise.  The SNR is set at $\{40,  80\}$ and the sample sizes are varied among $\{2.5,  5\}\times rd_1$.  The convergence dynamics are displayed in Figure~\ref{fig:conv-Gaussian},  which clearly show a dual-phase convergence of RsGrad.  We note that ScaledSM performs poorly when sample size is small ($n=2.5rd_1$).  It is possibly due to the instability of inverse scaling under small sample size.  The bottom two plots in Figure~\ref{fig:conv-Gaussian}  show that,  after the phase one iterations,  RsGrad achieves a similar perfomrance as ScaledSM.  However,  with a second phase convergence,  RsGrad (using either Huber or absolute loss) eventually delivers a more accurate estimate than ScaledSM.

\begin{figure}
\centering
	\begin{subfigure}[b]{0.45\textwidth}
		\centering
		\includegraphics[width=\textwidth]{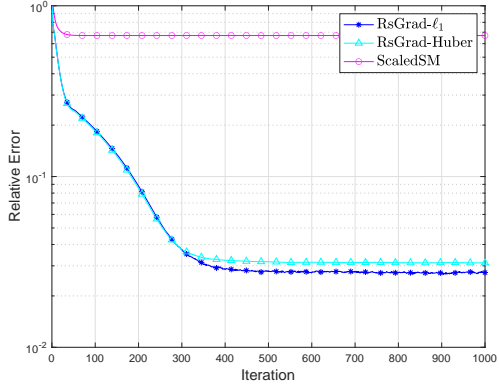}
		\caption{SNR=40}
		\label{fig:11}
	\end{subfigure}
	\hfill
	\begin{subfigure}[b]{0.45\textwidth}
		\centering
		\includegraphics[width=\textwidth]{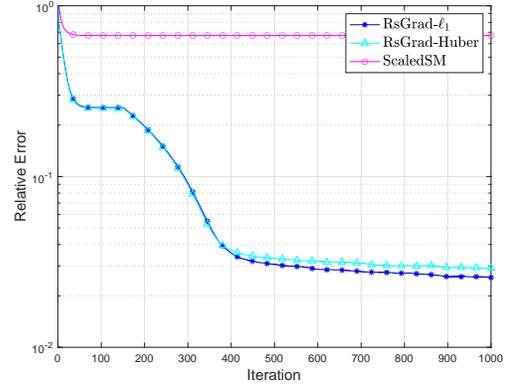}
		\caption{SNR=80}
		\label{fig:12}
	\end{subfigure}

	\begin{subfigure}[b]{0.45\textwidth}
		\centering
		\includegraphics[width=\textwidth]{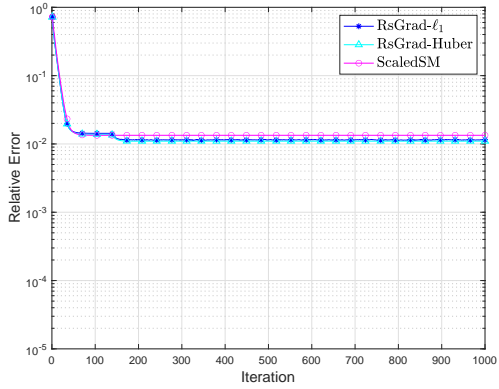}
		\caption{SNR=40}
		\label{fig:31}
	\end{subfigure}
	\hfill
	\begin{subfigure}[b]{0.45\textwidth}
		\centering
		\includegraphics[width=\textwidth]{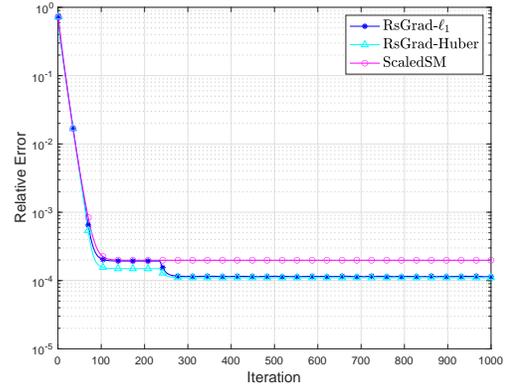}
		\caption{SNR=80}
		\label{fig:32}
	\end{subfigure}
	
	\caption{Convergence dynamics of RsGrad and ScaledSM \citep{tong2021low} under Gaussian noise.  Top two figures:  $n=2.5rd_1$;  bottom two figures: $n=5rd_1$.}
	\label{fig:conv-Gaussian}
\end{figure}

The second set of simulation is to test the convergence performance of RsGrad and ScaledSM under heavy-tailed noise.  For simplicity,  the noise is sampled independently from a Student's $t$-distribution with ${\rm d.f.}$ $\nu=2$,  which has a finite first (absolute) moment and an infinite second moment.  Other parameters are selected the same as in the Gaussian case.   Figure~\ref{fig:conv-t2} presents the convergence performance showing that both RsGrad (either absolute loss or Huber loss)  and ScaledSM are robust to heavy-tailed noise.  
Similarly,  we can observe dual-phase convergence of RsGrad whose second phase iterations deliver a more accurate estimate than ScaledSM,  especially when SNR is large.  

\begin{figure}
\centering
	\begin{subfigure}[b]{0.45\textwidth}
		\centering
		\includegraphics[width=\textwidth]{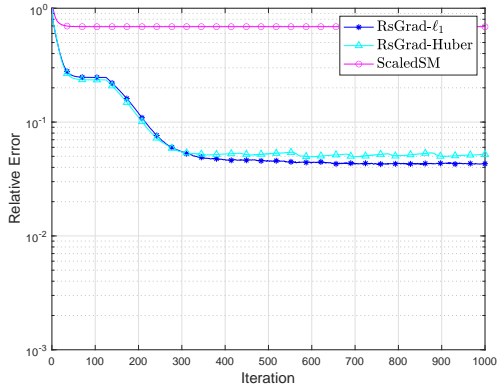}
		\caption{SNR=40}
		\label{fig:21}
	\end{subfigure}
	\hfill
	\begin{subfigure}[b]{0.45\textwidth}
		\centering
		\includegraphics[width=\textwidth]{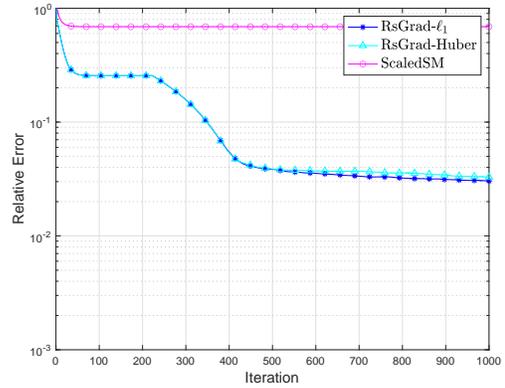}
		\caption{SNR=80}
		\label{fig:22}
	\end{subfigure}
	
	\begin{subfigure}[b]{0.45\textwidth}
		\centering
		\includegraphics[width=\textwidth]{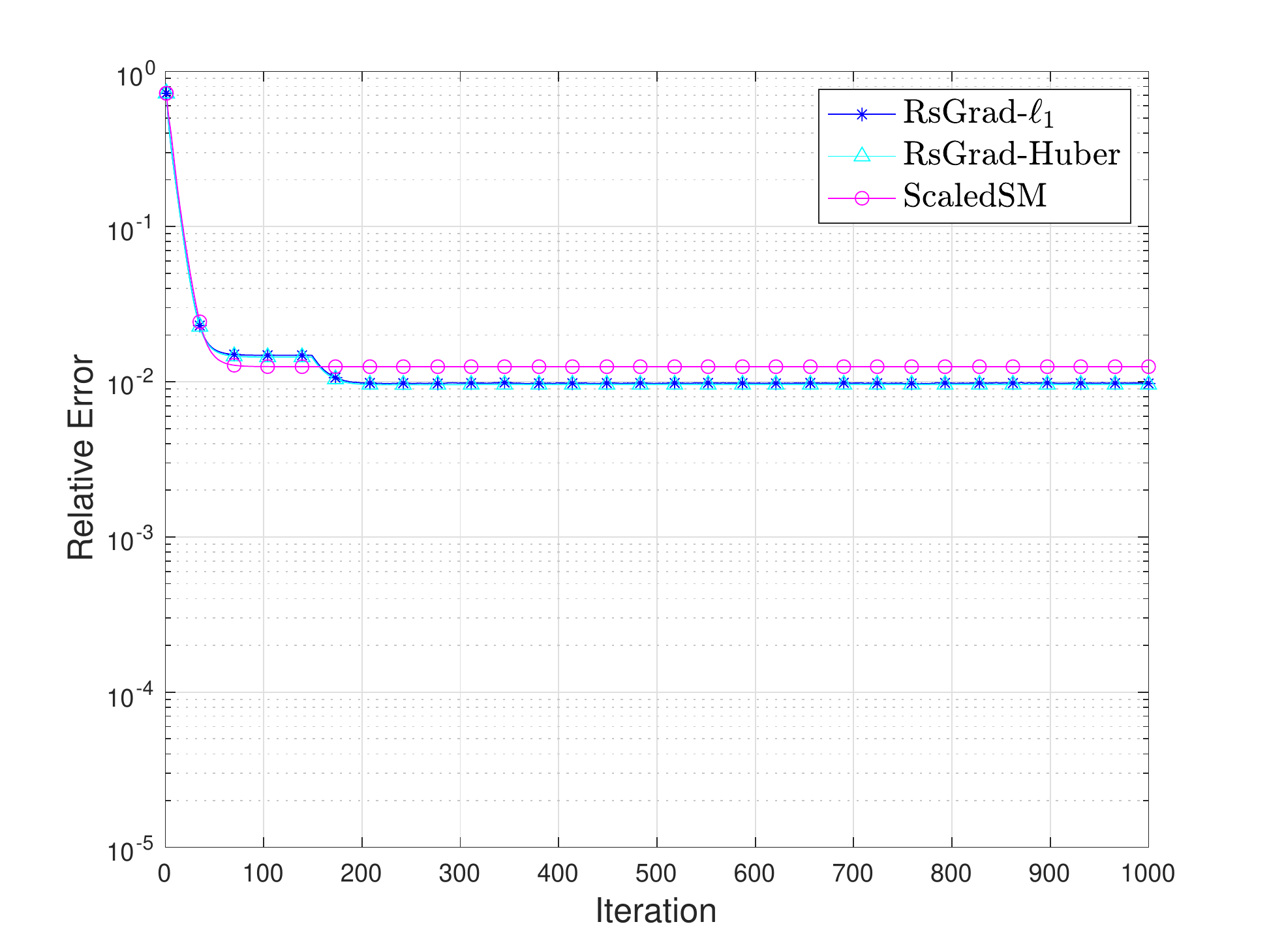}
		\caption{SNR=40}
		\label{fig:41}
	\end{subfigure}
	\hfill
	\begin{subfigure}[b]{0.45\textwidth}
		\centering
		\includegraphics[width=\textwidth]{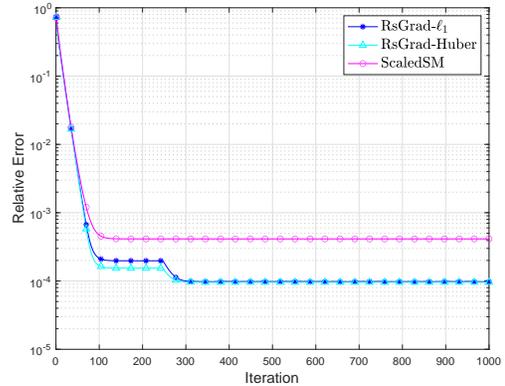}
		\caption{SNR=80}
		\label{fig:42}
	\end{subfigure}

	\caption{Convergence dynamics of RsGrad and ScaledSM \citep{tong2021low} under Student's $t$-distribution with d.f.  $\nu=2$.  Top two figures:  $n=2.5rd_1$;  bottom two figures: $n=5rd_1$.}
	\label{fig:conv-t2}
\end{figure}

\paragraph*{Statistical accuracy} We now compare the statistical accuracy of the final estimator output by RsGrad,  ScaledSM and those by convex approaches including the  nuclear norm penalized \citep{elsener2018robust} absolute loss (convex $\ell_1$),  Huber loss (convex Huber) and the square loss ($\ell_2$-loss) by seminal work \cite{candes2011tight}.   Functions from Matlab library cvx \citep{grant2014cvx} are borrowed to implement the convex methods for convex $\ell_1$,  convex Huber and $\ell_2$-loss.  For each setting,  every method is repeated for $10$ times and the box-plots of error rates are presented.

The parameter settings are similar as above where both the Gaussian noise and Student's $t$-distribution are experimented.  The comparison of statistical accuracy under Gaussian noise is displayed in Figure~\ref{fig:accu-Gaussian}. The top two plots suggest that all the convex methods and ScaledSM performs poorly when sample size is small ($n=2.5rd_1$).  On the other hand,  ScaledSM becomes much better when sample size is large but still under-performs RsGrad,  as predicted by our theory.   Similar results are also observed when noise has a Students' $t$-distribution with d.f.  $\nu=2$.  See Figure~\ref{fig:accu-t2}.  

\begin{figure}
\centering
	\begin{subfigure}[b]{0.45\textwidth}
		\centering
		\includegraphics[width=\textwidth]{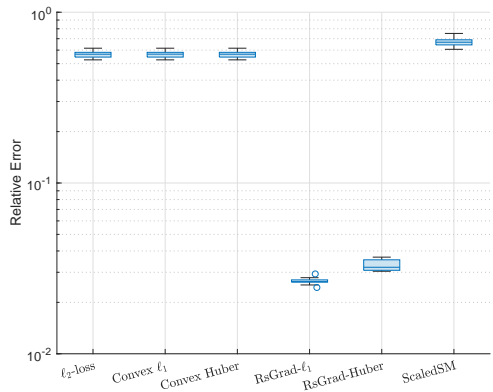}
		\caption{SNR=40}
		\label{fig:71}
	\end{subfigure}
	\hfill
	\begin{subfigure}[b]{0.45\textwidth}
		\centering
		\includegraphics[width=\textwidth]{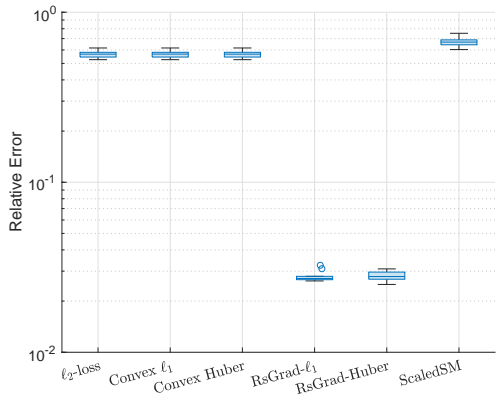}
		\caption{SNR=80}
		\label{fig:72}
	\end{subfigure}
	
	\begin{subfigure}[b]{0.45\textwidth}
		\centering
		\includegraphics[width=\textwidth]{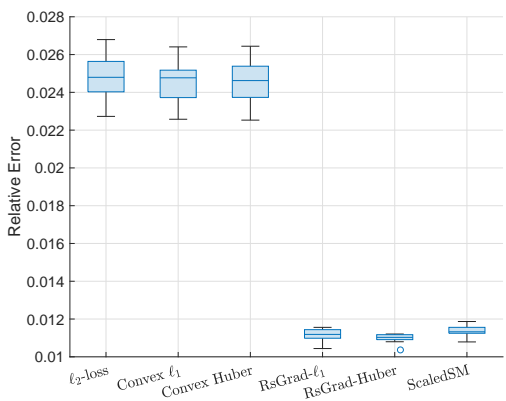}
		\caption{SNR=40}
		\label{fig:91}
	\end{subfigure}
	\hfill
	\begin{subfigure}[b]{0.45\textwidth}
		\centering
		\includegraphics[width=\textwidth]{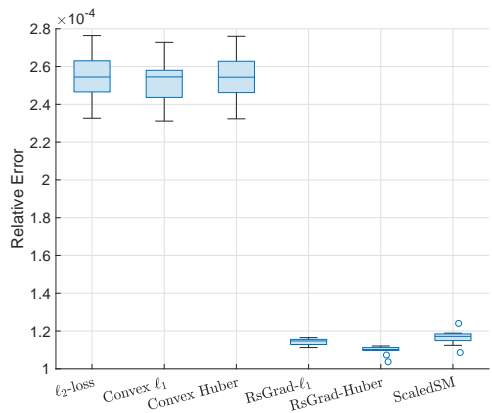}
		\caption{SNR=80}
		\label{fig:92}
	\end{subfigure}
	
	\caption{Accuracy comparisons of RsGrad, ScaledSM, robust convex models \citep{elsener2018robust} and convex $\ell_2$ loss \citep{candes2011tight} under Gaussian noise.  Top two figures:  $n=2.5rd_1$;  bottom two figures: $n=5rd_1$.}
	\label{fig:accu-Gaussian}
\end{figure}

\begin{figure}
\centering
	\begin{subfigure}[b]{0.45\textwidth}
		\centering
		\includegraphics[width=\textwidth]{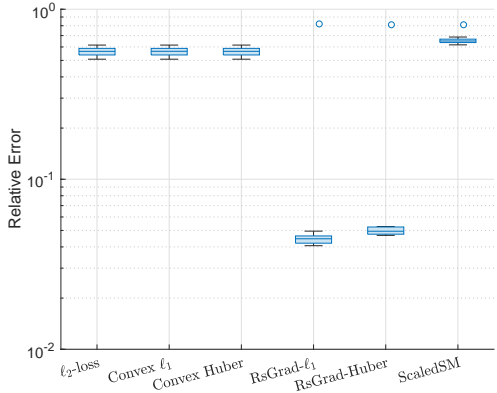}
		\caption{SNR=40}
		\label{fig:81}
	\end{subfigure}
	\hfill
	\begin{subfigure}[b]{0.45\textwidth}
		\centering
		\includegraphics[width=\textwidth]{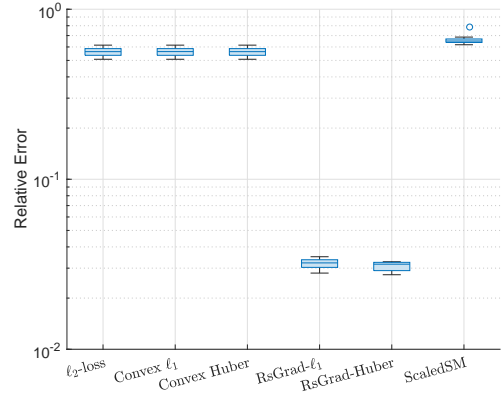}
		\caption{SNR=80}
		\label{fig:82}
	\end{subfigure}
	
	\begin{subfigure}[b]{0.45\textwidth}
		\centering
		\includegraphics[width=\textwidth]{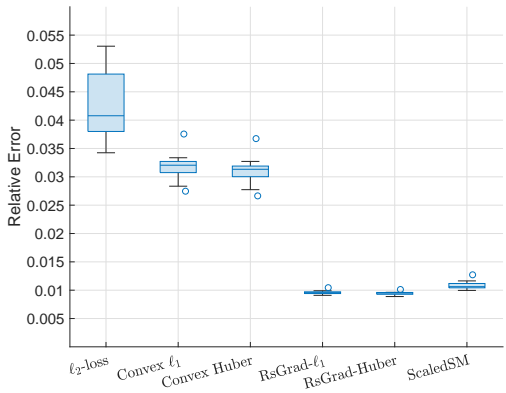}
		\caption{SNR=40}
		\label{fig:101}
	\end{subfigure}
	\hfill
	\begin{subfigure}[b]{0.45\textwidth}
		\centering
		\includegraphics[width=\textwidth]{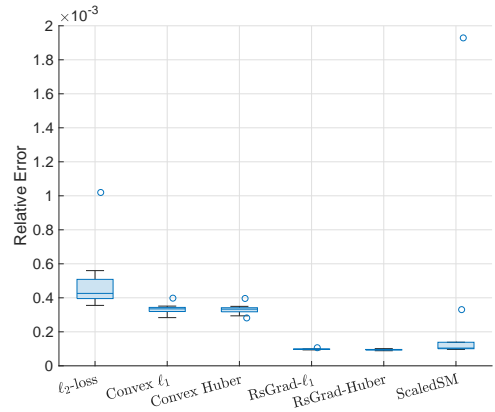}
		\caption{SNR=80}
		\label{fig:102}
	\end{subfigure}
	
	\caption{Accuracy comparisons of RsGrad, ScaledSM, robust convex models \citep{elsener2018robust} and convex $\ell_2$ loss \citep{candes2011tight} under Students' $t$-distribution noise with d.f.  $\nu=2$.   Top two figures:  $n=2.5rd_1$;  bottom two figures: $n=5rd_1$.}
	\label{fig:accu-t2}
\end{figure}

From Figure~\ref{fig:accu-Gaussian} and Figure~\ref{fig:accu-t2},  we may conclude that nonconvex models are statistically more accurate in practice for both Gaussian noise and heavy-tailed noise,  though convex approaches are proven statistically optimal in theory.  Moreover,  convex approaches are computationally much more demanding.  For instance,  in the case $n=5rd_1$,  convex $\ell_1$, convex Huber, convex $\ell_2$ need more than 1.5h, 1.5h, 3h to complete the ten repetitions,  respectively.  In sharp contrast,  RsGrad and ScaledSM only take around ten minutes.  When it comes to an even larger sample size like $n=10d_1r$,  the Matlab package on our computing platform needs 3h and 6h to complete the simulations (10 repetitions) for convex $\ell_1$ and convex Huber,  respectively,  whereas RsGrad and ScaledSM each takes about 20 minutes to do the job.

\subsection{Low-rank Tensor Estimation}
This section presents the results of numerical experiments on RsGrad for low-rank tensor estimation. The stepsizes schedule, relative error and SNR are similar to the matrix simulations.  For simplicity, the tensor size is fixed at $30\times 30\times 30$ with Tucker rank $(2,2,2)$.

\paragraph*{Convergence dynamics}
RsGrad is compared with the $\ell_2$-loss Riemannian gradient algorithm, RGrad \citep{cai2020provable} under both Gaussian and heavy-tailed noise. The results are displayed in Figure~\ref{fig:conv-tensor}. Under Gaussian noise, RsGrad slightly under-performs RGrad. However, when the noise has a Student's t-distribution with d.f. $\nu=2.1$, RsGrad significantly outperforms RGrad. Clearly, the dual-phase convergence of RsGrad is observed confirming our theory.

\begin{figure}
	\centering
	\begin{subfigure}[b]{0.45\textwidth}
		\centering
		\includegraphics[width=\textwidth]{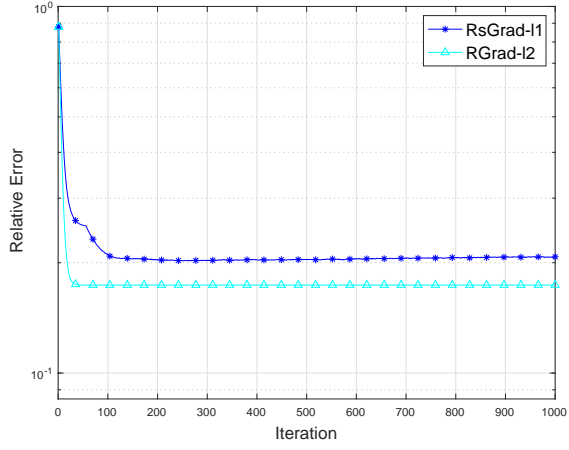}
		\caption{ $\frac{\frorr{\M^*}}{\EE|\xi|}=5$}
		\label{fig:111}
	\end{subfigure}
	\hfill
	\begin{subfigure}[b]{0.45\textwidth}
		\centering
		\includegraphics[width=\textwidth]{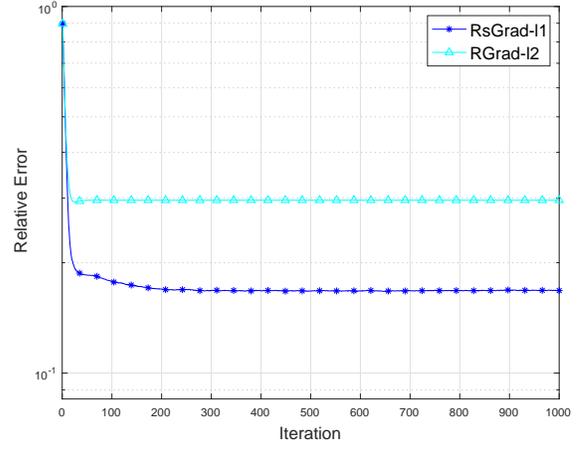}
		\caption{ $\frac{\frorr{\M^*}}{\EE|\xi|}=5$}
		\label{fig:112}
	\end{subfigure}

	\begin{subfigure}[b]{0.45\textwidth}
		\centering
		\includegraphics[width=\textwidth]{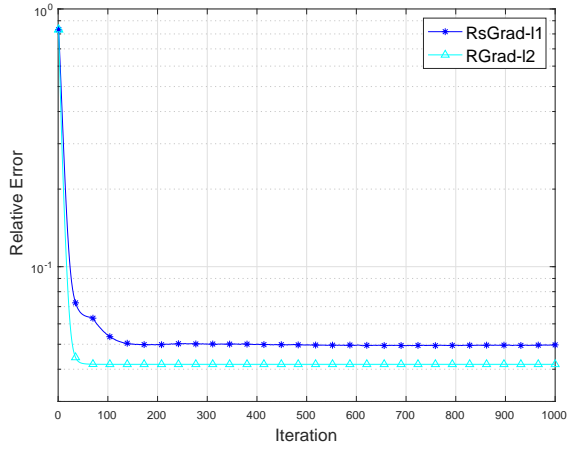}
		\caption{$\frac{\frorr{\M^*}}{\EE|\xi|}=20$}
		\label{fig:113}
	\end{subfigure}
	\hfill
	\begin{subfigure}[b]{0.45\textwidth}
		\centering
		\includegraphics[width=\textwidth]{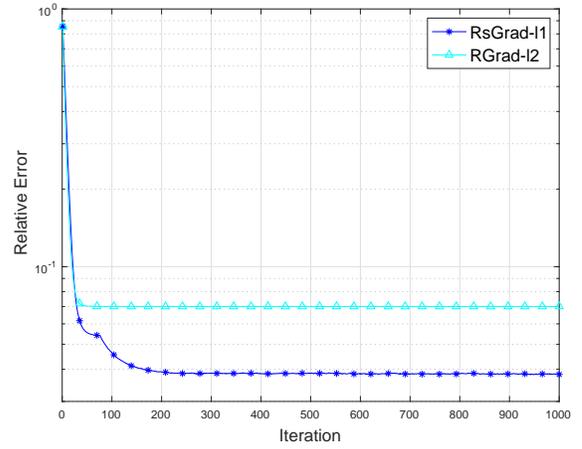}
		\caption{$\frac{\frorr{\M^*}}{\EE|\xi|}=20$}
		\label{fig:114}
	\end{subfigure}

	\caption{Conergence dynamics of RsGrad and RGrad \citep{cai2020provable} under $d_1=d_2=d_3=30$, $r_1=r_2=r_3=2$, $n = 600$. The left two figures: Gaussian noise; The right two figures; Students' t-distribution noise with d.f. $\nu=2.1$.}
	\label{fig:conv-tensor}.
\end{figure}

\paragraph*{Initialization Comparison}
We now compare the shrinkage-based second order moment initialization with the naive spectral initialization by HOSVD: $\M_0^{\textsf{\tiny HOSVD}}=\text{HOSVD}_{\r}(n^{-1}\sum_{i=1}^{n} Y_i\X_i)$. The simulation results under small and large SNRs are presented in Figure~\ref{fig:tensor-init}. It shows that, under heavy-tailed noise, our initialization method attains a smaller relative error.

\begin{figure}
	\centering
	\begin{subfigure}[b]{0.45\textwidth}
		\centering
		\includegraphics[width=\textwidth]{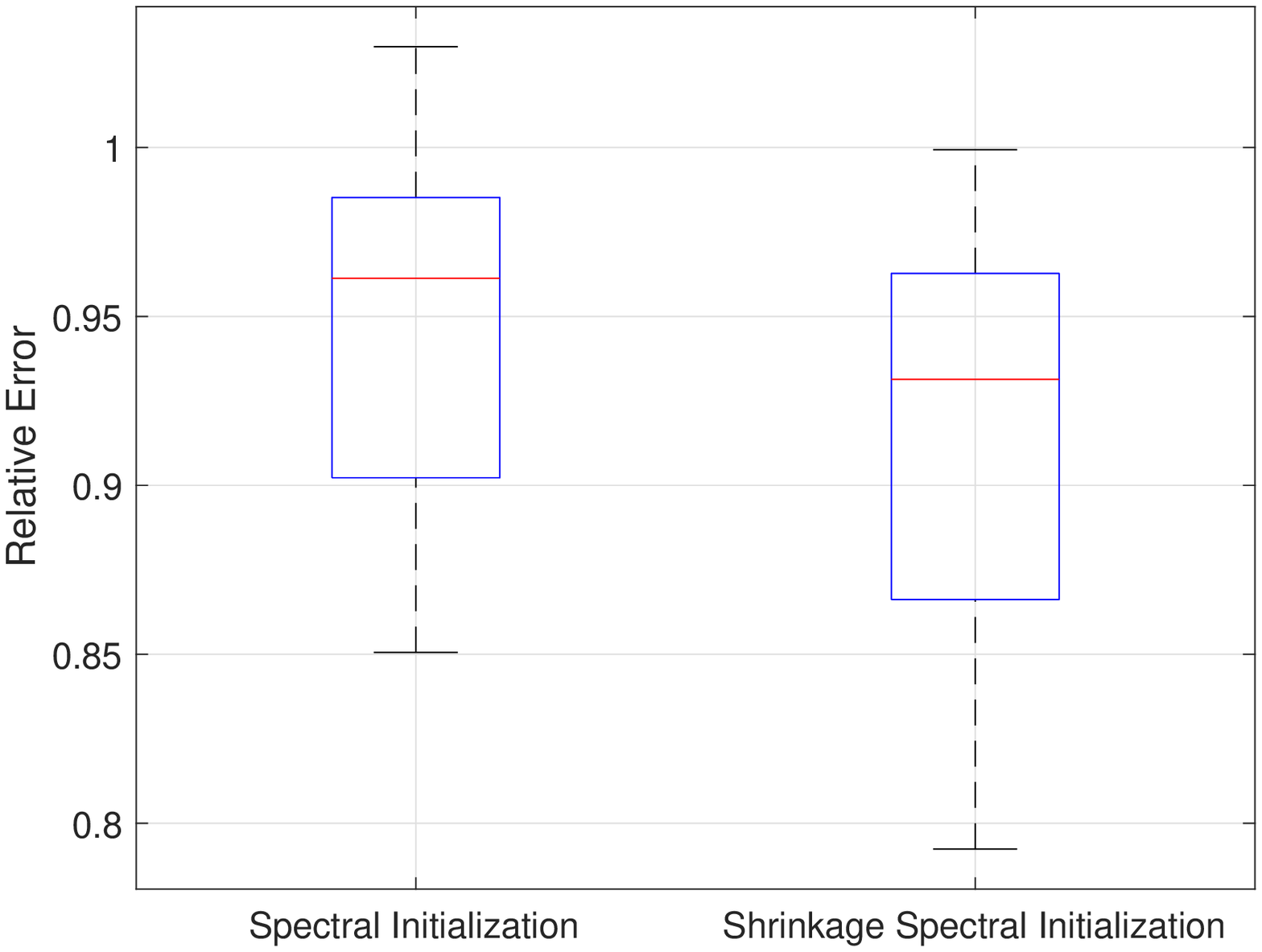}
		\caption{$\frac{\frorr{\M^*}}{\EE|\xi|}=5$}
		\label{fig:121}
	\end{subfigure}
	\hfill
	\begin{subfigure}[b]{0.45\textwidth}
		\centering
		\includegraphics[width=\textwidth]{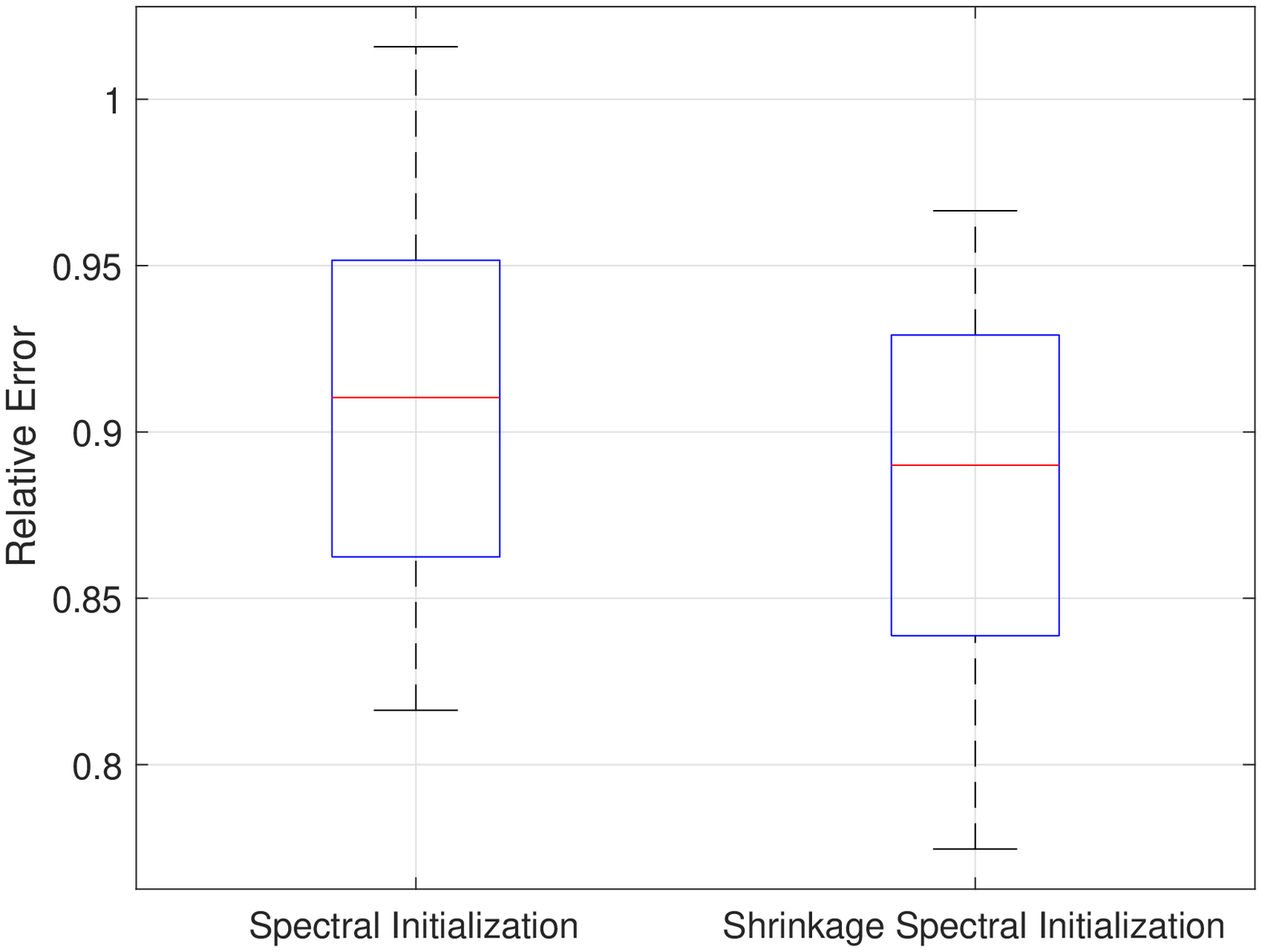}
		\caption{$\frac{\frorr{\M^*}}{\EE|\xi|}=20$}
		\label{fig:122}
		
	\end{subfigure}
	\caption{Initialization accuracy comparisons of shrinkage spectral inialization and spectral initialization under Students' t-distribution noise with d.f. $\nu=2.1$, $d_1=d_2=d_3=30$, $r_1=r_2=r_3=2$, $n=420$ }
	\label{fig:tensor-init}
\end{figure}

\newpage

\bibliographystyle{plainnat}
\bibliography{reference} % see references.bib for bibliography management

\newpage

\appendix

\section{Proofs for main results}
\subsection{Proof of Proposition~\ref{prop:main}}
\label{proof:prop:main}
For convenience, denote $d_{l}:=(1-0.04\muc^2\Lc^{-2})^{l}\cdot\fro{\M_0-\M^*}$ and proof of phase one is equivalent to verification of $\fro{\M_{l+1}-\M^*}\leq d_{l+1}$ when $\fro{\M_{l}-\M^*}\geq \tau_{\textsf{\tiny comp}}$. 

Prove by induction and when $l=0$, $\fro{\M_{0}-\M^*}\leq d_{0} $ is obvious. Suppose we already have $\fro{\M_l-\M^*}\leq d_{l}$ and we are going to prove $\fro{\M_{l+1}-\M^*}\leq d_{l+1}$. 

Notice that $\M_{l+1}=\text{SVD}_r(\M_{l}-\eta_{l}\mathcal{P}_{\TT_l}(\G_l))$ and we use Lemma~\ref{teclem:perturbation} to bound the distance between $\M_{l+1}$ and the ground truth matrix $\M^*$. First consider $\Vert \M_{l}-\eta_{l}\mathcal{P}_{\mathbb{T}_{l}}(\G_{l})-\M^{*}\Vert_{\mathrm{F}}^{2}$,
\begin{equation}\label{eq1}
	\begin{split}
		\Vert \M_{l}-\eta_{l}\mathcal{P}_{\mathbb{T}_{l}}(\G_{l})-\M^{*}\Vert_{\mathrm{F}}^{2}&= \fro{\M_{l}-\M^{*}}^{2} - 2\eta_{l}\langle \M_{l}-\M^{*}, \mathcal{P}_{\mathbb{T}_{l}}(\G_{l})\rangle + \eta_{l}^{2}\fro{ \mathcal{P}_{\mathbb{T}_{l}}(\G_{l})}^{2}.
	\end{split}
\end{equation} The second term can be bounded in the way of
\begin{equation*}
	\begin{split}
		\inp{\M_l-\M^*}{\mathcal{P}_{\mathbb{T}_{l}}(\G_{l}) }&=\inp{\M_l-\M^*}{\G_{l}}-\inp{\M_l-\M^*}{\mathcal{P}_{\mathbb{T}_{l}}^{\perp}(\G_l)}\\
		&\geq (f(\M_{l}) - f(\M^{*})) - \langle \M_{l}-\M^{*}, \mathcal{P}_{\mathbb{T}_{l}}^{\perp}(\mathbf{G}_{l})\rangle,\\
	\end{split}
\end{equation*}
where the last inequality is from definition of sub-gradient. Note that 
\begin{align*}
	\inp{\M_l-\M^\ast}{\mathcal{P}_{\mathbb{T}_{l}}^{\perp}(\G_{l})}&\leq\fro{(\U\U^{\top}-\U_l\U_l^{\top})(\M_l-\M^*)(\I-\V_l\V_l^{\top})}\fror{\G_{l}}\\
	&\leq 4\frac{\Lc}{\sigma_r}\fro{\M_l-\M^\ast}^2,
\end{align*}
which uses $\op{\U\U^{\top}-\U_l\U_l^\top}\leq \frac{4\fro{\M_l-\M^*}}{\sigma_{r}}$ (see Lemma~\ref{teclem:perturbation}) and $\fror{\G_{l}}\leq \Lc$. The above two equations lead to \begin{align*}
	\inp{\M_l-\M^*}{\mathcal{P}_{\mathbb{T}_{l}}(\G_{l}) }\geq(f(\M_{l}) - f(\M^{*}))-4\frac{\Lc}{\sigma_r}\fro{\M_l-\M^\ast}^2.
\end{align*}
Together with $\Vert \mathcal{P}_{\TT_l}(\G_l)\Vert_{\mathrm{F}}\leq\sqrt{2}\fror{\G_l}\leq\sqrt{2}\Lc$ (see Lemma~\ref{teclem:projected subgradient norm}), $f(\M_{l}) - f(\M^{*})\geq \muc \fro{\M_l-\M^*}$ and $\Vert \M_{l} - \M^{*}\Vert_{\mathrm{F}}\leq \fro{\M_0-\M^*}\leq\frac{1}{8}\frac{\muc}{\Lc}\sigma_{r}$, Equation~\ref{eq1} could be  bounded in the way of
\begin{align*}
	\fro{\M_l-\eta_l\mathcal{P}_{\TT_l}(\G_l)-\M^*}^2\leq \fro{\M_{l}-\M^*}^2-\eta_l\muc\fro{\M_l-\M^*}+2\eta_l^2\Lc^{2}.
\end{align*}
By induction, one already has $\fro{\M_{l}-\M^{*}}\leq  (1-0.04\muc^2/\Lc^2)^{l}\cdot\fro{\M_0-\M^*}=:d_{l}$ and $\eta_l$ satisfies $ \frac{1}{5}d_{l}\muc\Lc^{-2}\leq\eta_{l}=(1-0.04\muc^2/\Lc^2)^{l}\cdot\eta_{0}\leq \frac{3}{10}d_{l}\muc\Lc^{-2}$, which implies
\begin{equation}
	\fro{\M_l-\eta_l\mathcal{P}_{\TT_l}(\G_l)-\M^*}^2\leq (1-\frac{3}{25}\frac{\muc^2}{\Lc^2})d_{l}^2.
	\label{eq2}
\end{equation}
It shows $\fro{\M_l-\eta_l\mathcal{P}_{\TT_l}(\G_l)-\M^*}<\sigma_{r}/4$ and then by Lemma~\ref{teclem:perturbation}, one could bound $\Vert  \operatorname{SVD}_{r}(\mathbf{M}_{l} - \eta_{l}\mathcal{P}_{\mathbb{T}_{l}}(\mathbf{G}_{l}))- \mathbf{M}^{*}\Vert_{\mathrm{F}}$:
\begin{equation}
	\begin{split}
		\fro{\M_{l+1}-\M^*}^2
		&= \Vert  {\rm SVD}_{r}(\M_{l} - \eta_{l}\mathcal{P}_{\mathbb{T}_{l}}(\mathbf{G}_{l}))- \M^{*}\Vert_{\mathrm{F}}^{2}\\
		&\leq \fro{\M_l-\eta_l\mathcal{P}_{\TT_l}(\G_l)-\M^*}^2+\frac{50}{\sigma_{r}}\fro{\M_l-\eta_l\mathcal{P}_{\TT_l}(\G_l)-\M^*}^{3},
	\end{split}
	\label{eq3}
\end{equation}
Insert Equation~\ref{eq2} into Equation~\ref{eq3} :
\begin{align*}
	\fro{\M_{l+1}-\M^*}^2&\leq (1-\frac{3}{25}\frac{\muc^{2}}{\Lc^{2}})d_{l}^2  + \frac{50}{\sigma_{r}}(1-\frac{3}{25}\frac{\muc^{2}}{\Lc^{2}})^{\frac{3}{2}}d_{l}^{3}\\
	&\leq (1-\frac{2}{25}\frac{\muc^{2}}{\Lc^{2}})d_{l}^2,
\end{align*}
which uses $\frac{50}{\sigma_{r}}d_{l}\leq\frac{50}{\sigma_{r}}d_{0} \leq\frac{1}{25}\frac{\muc^{2}}{\Lc^{2}}$. It shows $$ \fro{\M_{l+1}-\M^*}\leq d_{l+1}= (1-0.04\muc^2/\Lc^2)^{l+1}\fro{\M_0-\M^*}. $$

Hence, analysis of phase one convergence has been finished. Then consider the second phase, $\taus\leq\fro{\M_l-\M^*}<\tauc$. The derivation is similar:
\begin{align*}
	&{~~~~}\fro{\M_l-\eta_l\mathcal{P}_{\TT_l}(\G_l)-\M^*}^2\\
	&=\fro{\M_l-\M^*}^{2} - 2\eta_l\langle \M_{l}-\M^{*}, \mathcal{P}_{\mathbb{T}_{l}}(\mathbf{G}_{l})\rangle + \eta_l^{2}\Vert \mathcal{P}_{\mathbb{T}_{l}}(\mathbf{G}_{l})\Vert_{\mathrm{F}}^{2}\\
	&\leq \fro{\M_l-\M^*}^{2} - 2\eta_l(f(\M_{l}) - f(\M^{*})) + 2\eta_l\frac{\Ls}{\sigma_{r}}\Vert \M_{l} - \M^{*}\Vert_{\mathrm{F}}^{3} + 2\eta_l^{2}\fror{\G_l}^2\\
	&\leq \fro{\M_l-\M^*}^{2} - 2\eta_l\mus\fro{\M_l-\M^*}^{2}+ 2\eta_l\frac{\Ls}{\sigma_{r}}\fro{\M_l-\M^*}^{3} + 2\eta_l^{2}\Ls^2\fro{\M_l-\M^*}^2\\
	&\leq(1-\frac{7}{4}\eta_l\mus)\fro{\M_l-\M^*}^{2}+2\eta_l^{2}\Ls^2\fro{\M_l-\M^*}^2,
\end{align*}
where the second inequality follows from the condition $f(\M_l)-f(\M^*)\geq\mus\fro{\M_l-\M^*}^2$, $\fror{\G_l}\leq \Ls\fro{\M_l-\M^*}$ and the third inequality uses $\fro{\M_l-\M^*}\leq\fro{\M_0-\M^*}\leq \frac{1}{8}\frac{\mus}{\Ls}\sigma_r$.
Take stepsize $ \frac{1}{8}\cdot\frac{\mus}{\Ls^2}\leq\eta_l \leq \frac{3}{4}\cdot\frac{\mus}{\Ls^2}$ and one has
\begin{align*}
	\fro{\M_l-\eta\calP_{\TT_l}(\G_l)-\M^*}^2 \leq (1-\frac{1}{8}\frac{\mus^2}{\Ls^2}) \fro{\M_{l} -\M^{*}}^2.
\end{align*}
This implies $\fro{\M_l-\eta\calP_{\TT_l}(\G_l)-\M^*}^2<\sigma_{r}/4$ and then similarly, use Lemma~\ref{teclem:perturbation} to cope with the matrix perturbation,
\begin{align*}
	\Vert \M_{l+1}-\M^{*}\Vert_{\mathrm{F}}^{2} &= \Vert  {\rm SVD}_{r}(\M_{l} - \eta_{l}\mathcal{P}_{\mathbb{T}_{l}}(\mathbf{G}_{l}))- \M^{*}\Vert_{\mathrm{F}}^{2}\\
	&\leq \fro{\M_l-\eta\calP_{\TT_l}(\G_l)-\M^*}^2+\frac{50}{\sigma_{r}}\fro{\M_l-\eta\calP_{\TT_l}(\G_l)-\M^*}^3\\
	&\leq 	(1-\frac{1}{8}\frac{\mus^2}{\Ls^2}) \fro{\M_{l} -\M^{*}}^2+\frac{50}{\sigma_{r}}(1-\frac{1}{8}\frac{\mus^2}{\Ls^2})^{3/2} \fro{\M_{l} -\M^{*}}^3\\
	&\leq (1-\frac{1}{16}\frac{\mus^2}{\Ls^2}) \fro{\M_{l} -\M^{*}}^2,
\end{align*}
which uses the condition $\fro{\M_{l} -\M^{*}}\leq\fro{\M_{0} -\M^{*}} \leq\frac{1}{800}\cdot\frac{\mus^2}{\Ls^2}\sigma_{r}$. It verifies $$\fro{\M_{l+1}-\M^*}\leq(1-\frac{1}{32}\cdot\frac{\mus^2}{\Ls^2})\fro{\mathbf{M}_{l}-\mathbf{M}^{*}},$$
which finishes the proof.

\subsection{Proof of tensor case Proposition~\ref{prop:tensor:main}}
For convenience, denote $d_{l}:=(1-\frac{1}{16(m+1)}\frac{\muc^2}{\Lc^2})^{l}\cdot\fro{\M_0-\M^*}$ and proof of phase one is equivalent to verification of $\fro{\M_{l+1}-\M^*}\leq d_{l+1}$ when $\fro{\M_{l}-\M^*}\geq \tau_{\textsf{\tiny comp}}$. 

Prove by induction and when $l=0$, $\fro{\M_{0}-\M^*}\leq d_{0} $ is obvious. Suppose we already have $\fro{\M_l-\M^*}\leq d_{l}$ and we are going to prove $\fro{\M_{l+1}-\M^*}\leq d_{l+1}$. 

Notice that $\M_{l+1}=\text{HOSVD}_\r(\M_{l}-\eta_{l}\mathcal{P}_{\TT_l}(\G_l))$ and we use Lemma~\ref{teclem:perturbation} to bound the distance between singular value truncated matrix and the ground truth matrix $\M^*$. First consider $\Vert \M_{l}-\eta_{l}\mathcal{P}_{\mathbb{T}_{l}}(\G_{l})-\M^{*}\Vert_{\mathrm{F}}^{2}$,
\begin{equation}\label{eq30}
	\begin{split}
		\Vert \M_{l}-\eta_{l}\mathcal{P}_{\mathbb{T}_{l}}(\G_{l})-\M^{*}\Vert_{\mathrm{F}}^{2}&= \fro{\M_{l}-\M^{*}}^{2} - 2\eta_{l}\langle \M_{l}-\M^{*}, \mathcal{P}_{\mathbb{T}_{l}}(\G_{l})\rangle + \eta_{l}^{2}\fro{ \mathcal{P}_{\mathbb{T}_{l}}(\G_{l})}^{2}.
	\end{split}
\end{equation} The second term can be bounded in the way of
\begin{equation*}
	\begin{split}
		\inp{\M_l-\M^*}{\mathcal{P}_{\mathbb{T}_{l}}(\G_{l}) }&=\inp{\M_l-\M^*}{\G_{l}}-\inp{\M_l-\M^*}{\mathcal{P}_{\mathbb{T}_{l}}^{\perp}(\G_l)}\\
		&\geq (f(\M_{l}) - f(\M^{*})) - \langle \M_{l}-\M^{*}, \mathcal{P}_{\mathbb{T}_{l}}^{\perp}(\mathbf{G}_{l})\rangle,\\
	\end{split}
\end{equation*}
where the last inequality is from definition of sub-gradient. Note that the second term could be bounded with
\begin{align*}
	\inp{\M_l-\M^\ast}{\mathcal{P}_{\mathbb{T}_{l}}^{\perp}(\G_{l})}\leq \fro{\calP_{\TT_l}^{\perp}(\M_l-\M^*)}\Vert\G_l\Vert_{\mathrm{F,2\r}}\leq \frac{8m(m+3)\Lc}{\underline\lambda}\fro{\M_l-\M^\ast}^2,
\end{align*}
which uses Lemma~\ref{teclem:tensor:perp projected} and $\Vert\G_{l}\Vert_{\mathrm{F,2\r}}\leq \Lc$. The above two equations lead to \begin{align*}
	\inp{\M_l-\M^*}{\mathcal{P}_{\mathbb{T}_{l}}(\G_{l}) }\geq(f(\M_{l}) - f(\M^{*}))-8m(m+3)\frac{\Lc}{\underline{\lambda}}\fro{\M_l-\M^\ast}^2.
\end{align*}
Together with $\Vert \mathcal{P}_{\TT_l}(\G_l)\Vert_{\mathrm{F}}\leq\sqrt{m+1}\frorr{\G_l}\leq\sqrt{m+1}\Lc$ (see Lemma~\ref{teclem:tensor:projected subgrad}), $f(\M_{l}) - f(\M^{*})\geq \muc \fro{\M_l-\M^*}$ and $\Vert \M_{l} - \M^{*}\Vert_{\mathrm{F}}\leq \fro{\M_0-\M^*}\leq\frac{1}{16}\frac{\muc}{m(m+3)\Lc}\underline{\lambda}$, Equation~\ref{eq30} becomes
\begin{align*}
	\fro{\M_l-\eta_l\mathcal{P}_{\TT_l}(\G_l)-\M^*}^2\leq \fro{\M_{l}-\M^*}^2-\eta_l\muc\fro{\M_l-\M^*}+(m+1)\eta_l^2\Lc^{2}.
\end{align*}
By induction, one already has $\fro{\M_{l}-\M^{*}}\leq  (1-\frac{1}{16(m+1)}\frac{\muc^2}{\Lc^2})^{l}\cdot\fro{\M_0-\M^*}=:d_{l}$. Also, stepsize $\eta_l$ satisfies $ \frac{1}{4(m+1)}d_{l}\muc\Lc^{-2}\leq\eta_{l}=(1-\frac{1}{16(m+1)}\frac{\muc^2}{\Lc^2})^{l}\cdot\eta_{0}\leq \frac{3}{4(m+1)}d_{l}\muc\Lc^{-2}$ and it implies
\begin{equation}
	\fro{\M_l-\eta_l\mathcal{P}_{\TT_l}(\G_l)-\M^*}^2\leq (1-\frac{3}{16(m+1)}\frac{\muc^2}{\Lc^2})d_{l}^2.
	\label{eq31}
\end{equation}
It shows $\fro{\M_l-\eta_l\mathcal{P}_{\TT_l}(\G_l)-\M^*}<\underline\lambda/4$ and then by Lemma~B.2 of \cite{cai2021generalized}, one could bound $\Vert  \operatorname{HOSVD}_{\r}(\mathbf{M}_{l} - \eta_{l}\mathcal{P}_{\mathbb{T}_{l}}(\mathbf{G}_{l}))- \mathbf{M}^{*}\Vert_{\mathrm{F}}$:
\begin{equation}
	\begin{split}
		\fro{\M_{l+1}-\M^*}^2
		&= \Vert  {\rm HOSVD}_{\r}(\M_{l} - \eta_{l}\mathcal{P}_{\mathbb{T}_{l}}(\mathbf{G}_{l}))- \M^{*}\Vert_{\mathrm{F}}^{2}\\
		&\leq \fro{\M_l-\eta_l\mathcal{P}_{\TT_l}(\G_l)-\M^*}^2+\frac{c_1}{\underline{\lambda}}\fro{\M_l-\eta_l\mathcal{P}_{\TT_l}(\G_l)-\M^*}^{3},
	\end{split}
	\label{eq32}
\end{equation}
Insert Equation~\ref{eq31} into Equation~\ref{eq32} :
\begin{align*}
	\fro{\M_{l+1}-\M^*}^2&\leq (1-\frac{3}{16(m+1)}\frac{\muc^{2}}{\Lc^{2}})d_{l}^2  + \frac{c_1}{\underline{\lambda}}(1-\frac{3}{16(m+1)}\frac{\muc^{2}}{\Lc^{2}})^{\frac{3}{2}}d_{l}^{3}\\
	&\leq (1-\frac{1}{8(m+1)}\frac{\muc^{2}}{\Lc^{2}})d_{l}^2,
\end{align*}
which uses $\frac{c_1}{\underline{\lambda}}d_{l}\leq\frac{c_1}{\underline{\lambda}}d_{0} \leq\frac{1}{16(m+1)}\frac{\muc^{2}}{\Lc^{2}}\mins$. Take the square root and then we have $$ \fro{\M_{l+1}-\M^*}\leq d_{l+1}= (1-\frac{1}{16(m+1)}\frac{\muc^2}{\Lc^2})^{l+1}\fro{\M_0-\M^*}. $$

Hence, analysis of phase one convergence has been finished. Then consider the second phase, $\taus\leq\fro{\M_l-\M^*}<\tauc$. The derivation is similar:
\begin{align*}
	&{~~~~}\fro{\M_l-\eta_l\mathcal{P}_{\TT_l}(\G_l)-\M^*}^2\\
	&=\fro{\M_l-\M^*}^{2} - 2\eta_l\langle \M_{l}-\M^{*}, \mathcal{P}_{\mathbb{T}_{l}}(\mathbf{G}_{l})\rangle + \eta_l^{2}\Vert \mathcal{P}_{\mathbb{T}_{l}}(\mathbf{G}_{l})\Vert_{\mathrm{F}}^{2}\\
	&\leq \fro{\M_l-\M^*}^{2} - 2\eta_l(f(\M_{l}) - f(\M^{*})) + 16\eta_l\frac{m(m+3)\Ls}{\underline{\lambda}}\Vert \M_{l} - \M^{*}\Vert_{\mathrm{F}}^{3} + (m+1)\eta_l^{2}\frorr{\G_l}^2\\
	&\leq \fro{\M_l-\M^*}^{2} - 2\eta_l\mus\fro{\M_l-\M^*}^{2}+ 16\eta_l\frac{m(m+3)\Ls}{\underline{\lambda}}\fro{\M_l-\M^*}^{3} + (m+1)\eta_l^{2}\Ls^2\fro{\M_l-\M^*}^2\\
	&\leq(1-\eta_l\mus)\fro{\M_l-\M^*}^{2}+(m+1)\eta_l^{2}\Ls^2\fro{\M_l-\M^*}^2,
\end{align*}
where the second inequality follows from the condition $f(\M_l)-f(\M^*)\geq\mus\fro{\M_l-\M^*}^2$, $\Vert\G_l\Vert_{\mathrm{F,2\r}}\leq \Ls\fro{\M_l-\M^*}$ and the third inequality uses $\fro{\M_l-\M^*}\leq\fro{\M_0-\M^*}\leq \frac{1}{16m(m+3)}\frac{\mus}{\Ls}\underline{\lambda}$.
Take stepsize $ \frac{1}{4(m+1)}\cdot\frac{\mus}{\Ls^2}\leq\eta_l \leq \frac{3}{4(m+1)}\cdot\frac{\mus}{\Ls^2}$ and one has
\begin{align*}
	\fro{\M_l-\eta\calP_{\TT_l}(\G_l)-\M^*}^2 \leq (1-\frac{3}{16(m+1)}\frac{\mus^2}{\Ls^2}) \fro{\M_{l} -\M^{*}}^2.
\end{align*}
This implies $\fro{\M_l-\eta\calP_{\TT_l}(\G_l)-\M^*}^2<\underline{\lambda}/4$ and then similarly, use Lemma~B.2 \cite{cai2021generalized} to cope with the tensor perturbation,
\begin{align*}
	\Vert \M_{l+1}-\M^{*}\Vert_{\mathrm{F}}^{2} &= \Vert  {\rm HOSVD}_{\r}(\M_{l} - \eta_{l}\mathcal{P}_{\mathbb{T}_{l}}(\mathbf{G}_{l}))- \M^{*}\Vert_{\mathrm{F}}^{2}\\
	&\leq \fro{\M_l-\eta\calP_{\TT_l}(\G_l)-\M^*}^2+\frac{c_1}{\underline{\lambda}}\fro{\M_l-\eta\calP_{\TT_l}(\G_l)-\M^*}^3\\
	&\leq 	(1-\frac{3}{16(m+1)}\frac{\mus^2}{\Ls^2}) \fro{\M_{l} -\M^{*}}^2+\frac{c_1}{\underline{\lambda}}(1-\frac{3}{16(m+1)}\frac{\mus^2}{\Ls^2})^{3/2} \fro{\M_{l} -\M^{*}}^3\\
	&\leq (1-\frac{1}{8(m+1)}\frac{\mus^2}{\Ls^2}) \fro{\M_{l} -\M^{*}}^2,
\end{align*}
which uses the condition $\fro{\M_{l} -\M^{*}}\leq\fro{\M_{0} -\M^{*}} \leq\frac{1}{16c_1(m+1)}\cdot\frac{\mus^2}{\Ls^2}\underline{\lambda}$. It verifies $$\fro{\M_{l+1}-\M^*}\leq(1-\frac{1}{16(m+1)}\cdot\frac{\mus^2}{\Ls^2})\fro{\mathbf{M}_{l}-\mathbf{M}^{*}} .$$
We finish the proof.

\section{Proofs for Applications}
\subsection{Proof of absolute loss with Gaussian noise Lemma~\ref{lem:Gaussian-l1}}
\label{proof:lem:Gaussian-l1}
\begin{proof}
	First consider upper bound of sug-gradient $\G\in f(\M)$, where matrix $\M$ has rank at most $r$. Note that we have,
	\begin{align*}
		f(\M+\text{SVD}_r(\G))-f(\M)&= \big\vert \sum_{i=1}^{n}\vert Y_i-\inp{\M +\text{SVD}_r(\G)}{\X_i}\vert - \sum_{i=1}^{n}\vert Y_i-\inp{\M}{\X_i}\vert\big\vert\\
		&\leq \sum_{i=1}^{n}\vert\inp{\text{SVD}_r(\G)}{\X_i}\vert\leq 2n\fror{\G},
	\end{align*}
	where the last inequality holds with probability exceeding $1-4\exp(-cd_1r)$ (see Corollary~\ref{cor: empirical process}). On the other hand, by definition of sub-gradient, we have $f(\M+\text{SVD}_r(\G))-f(\M)\geq \fror{\G}^2$ and together with the above equation, we have $\fror{\G_l}\leq 2n$. It implies $\Lc=2n$.
	
	Next consider the lower bound for $f(\M) - f(\M^*)$. Its expectation $\EE[f(\M)-f(\M^{*})]$ is calculated in Lemma~\ref{teclem:gaussian-l1}. We shall proceed assuming the event $\bcalE = \{\sup_{\M\in\MM_r}|f(\M)-f(\M^*) - \EE(f(\M)-f(\M^*))|\cdot\fro{\M-\M^*}^{-1}\leq C_2\sqrt{nd_1r} \}$ holds and specifically Theorem \ref{thm:empirical process} proves $\bcalE$ holds with probability exceeding $1-\exp(-cd_1r)$.
	
	We discuss the two phases respectively, namely, phase one when $\fro{\M-\M^*}\geq \sigma$ and phase two when $C\sqrt{d_1r/n}\sigma\leq\fro{\M-\M^*}\leq \sigma$.
	
	\noindent\textit{Case 1:}
	When $\Vert \mathbf{M}-\mathbf{M}^{*}\Vert_{\mathrm{F}}\geq\sigma$, we have 
	\begin{align*}
		f(\mathbf{M})-f(\mathbf{M}^*)&\geq\mathbb{E}\left[f(\mathbf{M})-f(\mathbf{M}^*) \right] - C_2\sqrt{nd_1r}\Vert \mathbf{M}-\mathbf{M}^{*}\Vert_{\mathrm{F}}\\
		&= n\sqrt{2/\pi}\frac{1}{\sqrt{\Vert \mathbf{M}-\mathbf{M}^{*}\Vert_{\mathrm{F}}^{2}+\sigma^{2}}+\sigma}\Vert \mathbf{M}-\mathbf{M}^{*}\Vert_{\mathrm{F}}^{2} - C_2\sqrt{nd_1r}\Vert \mathbf{M}-\mathbf{M}^{*}\Vert_{\mathrm{F}}\\
		&\geq n\sqrt{2/\pi}\frac{1}{\sqrt{2}+1}\Vert \mathbf{M}-\mathbf{M}^{*}\Vert_{\mathrm{F}} -C_2\sqrt{nd_1r}\Vert \mathbf{M}-\mathbf{M}^{*}\Vert_{\mathrm{F}}\\
		&\geq \frac{1}{12}n\Vert \mathbf{M}-\mathbf{M}^{*}\Vert_{\mathrm{F}},
	\end{align*}
	where the penultimate line is from $\Vert \mathbf{M}-\mathbf{M}^{*}\Vert_{\mathrm{F}}\geq\sigma$ and the last line uses $n\geq Cd_1r$ for some large absolute constant $C>0$. It verifies $\muc=\frac{1}{12}n$.

	\noindent\textit{Case 2:}
	When $C\sqrt{d_1r/n}\sigma\leq\Vert \mathbf{M}-\mathbf{M}^{*}\Vert_{\mathrm{F}}<\sigma$, we have \begin{equation*}
		\begin{split}
			f(\mathbf{M})-f(\mathbf{M}^*)&\geq\mathbb{E}\left[f(\mathbf{M})-f(\mathbf{M}^*) \right] - C_2\sqrt{nd_1r}\Vert \mathbf{M}-\mathbf{M}^{*}\Vert_{\mathrm{F}}\\
			&= n\sqrt{2/\pi}\frac{1}{\sqrt{\Vert \mathbf{M}-\mathbf{M}^{*}\Vert_{\mathrm{F}}^{2}+\sigma^{2}}+\sigma}\Vert \mathbf{M}-\mathbf{M}^{*}\Vert_{\mathrm{F}}^{2} - C_2\sqrt{nd_1r}\Vert \mathbf{M}-\mathbf{M}^{*}\Vert_{\mathrm{F}}\\
			&\geq\frac{1}{6}\frac{n}{\sigma}\Vert \mathbf{M}-\mathbf{M}^{*}\Vert_{\mathrm{F}}^{2} -C_2\sqrt{nd_1r}\Vert \mathbf{M}-\mathbf{M}^{*}\Vert_{\mathrm{F}}\\
			&\geq \frac{1}{12}\frac{n}{\sigma}\Vert \mathbf{M}-\mathbf{M}^{*}\Vert_{\mathrm{F}}^{2},
		\end{split}
	\end{equation*}
	where the penultimate line uses $\Vert \mathbf{M}-\mathbf{M}^{*}\Vert_{\mathrm{F}}< \sigma$ and the last line is from $\fro{\M-\M^*}\geq C\sqrt{d_1r/n}\sigma$. It shows $\mus=\frac{1}{12}\frac{n}{\sigma}$.
	
	Finally, from the following Lemma \ref{lemma:upperboundsubgradient:gaussian}, we see that $\Ls \leq C_4n\sigma^{-1}$. And this finishes the proof of the lemma.
\end{proof}

\begin{lemma}[Upper bound for sub-gradient]\label{lemma:upperboundsubgradient:gaussian}
	Let $\M\in\RR^{d_1\times d_2}$ have rank at most $r$ and satisfy $\fro{\M-\M^*}\geq \sqrt{\frac{d_1r}{n}}\sigma$. Let $\G\in\partial f(\M)$ be the sub-gradient. Under the event $\bcalE=\{\sup_{\M,\M_1\in\MM_r}|f(\M+\M_1)-f(\M) - \EE(f(\M+\M_1)-f(\M))|\cdot\fro{\M_1}^{-1}\leq C_1\sqrt{nd_1r} \}$, we have $\fror{\G}\leq Cn\sigma^{-1}\fro{\M-\M^*}$ for some absolute constant $C>0$.
\end{lemma}
\begin{proof}
	Take $\M_1 = \M + \frac{\sigma}{2n}\text{SVD}_r(\G)$, where $\text{SVD}_r(\G)$ is the best rank $r$ approximation of $\G$. So $\rank(\M-\M_1) \leq r$. Then we have
	\begin{align*}
		\EE f(\M_{1})-\EE f(\M)=&\sqrt{2/\pi}n\frac{\Vert \M_{1}-\M^{*}\Vert_{\mathrm{F}}^{2} - \Vert \M-\M^{*}\Vert_{\mathrm{F}}^{2} }{\sqrt{\sigma^{2} + \Vert \M_{1}-\M^{*}\Vert_{\mathrm{F}}^{2} } + \sqrt{\sigma^2 + \Vert \M-\M^{*}\Vert_{\mathrm{F}}^{2} } }\\
		=&\sqrt{2/\pi}n\frac{\Vert \M_{1}-\M\Vert_{\mathrm{F}}^{2}+2\langle \M-\M^{*}, \M_{1}-\M\rangle }{\sqrt{\sigma^{2} + \Vert \M_{1}-\M^{*}\Vert_{\mathrm{F}}^{2} } + \sqrt{\sigma^2 + \Vert \M-\M^{*}\Vert_{\mathrm{F}}^{2} } }\\
		\leq&\sqrt{2/\pi}\frac{n}{\sigma}\left( \Vert \M_{1}-\M\Vert_{\mathrm{F}}^{2}+2\Vert \M-\M^{*}\Vert_{\mathrm{F}} \Vert\M_{1}-\M\Vert_{\mathrm{F}}\right),
	\end{align*}
	where the first equality comes from Lemma~\ref{teclem:gaussian-l1}. And from Theorem~\ref{thm:empirical process}, we have 
	\begin{align*}
		f(\M_{1}) - f(\M)\leq&\EE f(\M_{1})-\EE f(\M)+C_1\sqrt{ nd_1r}(\Vert \M_{1}-\M^{*}\Vert_{\mathrm{F}} + \Vert \M-\M^{*}\Vert_{\mathrm{F}})\\
		\leq&\EE f(\M_{1})-\EE f(\M)+C_1\frac{n}{\sigma}\Vert \M-\M^{*}\Vert_{\mathrm{F}}(2\Vert \M-\M^{*}\Vert_{\mathrm{F}} + \Vert\M_{1}-\M\Vert_{\mathrm{F}} )\\
		\leq& \frac{n}{\sigma} \Vert \M_{1}-\M\Vert_{\mathrm{F}}^{2} + (2+C_1) \frac{n}{\sigma}\Vert \M-\M^{*}\Vert_{\mathrm{F}} \Vert\M_{1}-\M\Vert_{\mathrm{F}} +2C_1\frac{n}{\sigma}\Vert \M-\M^{*}\Vert_{\mathrm{F}}^{2},
	\end{align*}
	where the second inequality is from the condition $\fro{\M-\M^*}\geq \sqrt{\frac{d_1r}{n}}\sigma$. Since $\M_1 = \M + \frac{\sigma}{2n}\text{SVD}_r(\G)$, we have 
	\begin{equation}\label{eq11}
		\begin{split}
			f(\M+\frac{\sigma}{2n}\text{SVD}_{r}(\G))&-f(\M)\\&\leq \frac{\sigma}{4n} \Vert\G\Vert_{\mathrm{F,r}}^{2} + (1+C_1/2) \Vert \M-\M^{*}\Vert_{\mathrm{F}} \Vert\G\Vert_{\mathrm{F,r}} +2C_1\frac{n}{\sigma}\Vert \M-\M^{*}\Vert_{\mathrm{F}}^{2}.
		\end{split}
	\end{equation}
	On the other hand, by the definition of sub-gradient, we have
	\begin{align}\label{eq12}
		f(\M+\frac{\sigma}{2n}\text{SVD}_{r}(\G))-f(\M)\geq \frac{\sigma}{2n}\fror{\G}^2.
	\end{align}
	Combine Equation~\ref{eq11} with Equation~\ref{eq12} and by solving the quadratic inequality we get $$\fror{\G}\leq Cn\sigma^{-1}\fro{\M-\M^*}.$$
\end{proof}

\subsection{Proof of absolute value loss with heavy-tailed noise Lemma~\ref{lem:heavytail-l1}}
\label{proof:lem:heavytail-l1}
The $\Lc=2n$ proof is the same as the one given in Section~\ref{proof:lem:Gaussian-l1}. Now we focus on lower bound of $f(\M)-f(\M^*)$.

First consider phase one when $\fro{\X-\X^*}\geq 30\EE\vert\xi\vert=30\gamma$. Use triangle inequality and then get
\begin{align}
	f(\mathbf{M})-f(\mathbf{M}^{*}) &= \sum_{i=1}^{n}|\xi_i-\inp{\M-\M^*}{\X_i}|-\sum_{i=1}^{n}|\xi_i| \geq\sum_{i=1}^{n}|\inp{\M-\M^*}{\X_i}|-2\Vert \boldsymbol{\xi}\Vert_{1}.
	\label{eq13}
\end{align}
Corollary~\ref{cor: empirical process} proves $ \sum_{i=1}^{n}|\inp{\M-\M^*}{\X_i}| \geq \sqrt{1/2\pi}n\fro{\M-\M^*}$ and Lemma~\ref{teclem:Contraction of Heavy Tailed Random Variables} proves with probability over $1- {}^*\!c_1n^{-\min\{1,\varepsilon\}}$, $\frac{1}{n}\lone{\boldsymbol{\xi}}\leq 3\EE\vert\xi\vert=3\gamma$. Then combined with $\fro{\M-\M^*}\geq 30\gamma$, Equation~\ref{eq13} could be further bounded with 
\begin{align*}
	f(\mathbf{M})-f(\mathbf{M}^{*})\geq \sqrt{1/2\pi}n\Vert \mathbf{M}-\mathbf{M}^{*}\Vert_{\mathrm{F}}-6n\gamma\geq \frac{n}{6}\fro{\M-\M^*},
\end{align*}
which verifies $\muc = \frac{n}{6}$.

Then consider the second phase when $C_2\sqrt{d_1r/n}b_0 \leq \fro{\M-\M^*}<30\gamma$. To bound expectation of $f(\M)-f(\M^*)$, we first consider the conditional expectation $\EE_\xi f(\M)-\EE_\xi f(\M^*)$. Notice for all $t_0\in\RR$,
\begin{align*}
	\EE|\xi-t_0| &= \int_{s\geq t_0}(s-t_0)dH_{\xi}(s) + \int_{s< t_0}(t_0-s)dH_{\xi}(s)\\
	&=2\int_{s\geq t_0}(s-t_0)dH_{\xi}(s) + \int_{-\infty}^{+\infty}(t_0-s)dH_{\xi}(s)\\
	&=2\int_{s\geq t_0}(1-H_{\xi}(s))ds + t_0 - \int_{-\infty}^{\infty}s dH_{\xi}(s).
\end{align*}
Set $t_0 = \inp{\X_i}{\M-\M^*}$ and define $g_{\X_i}(\M) := \EE_{\xi_i}|\xi_i - \inp{\X_i}{\M-\M^*}|$. It implies
\begin{align*}
	g_{\X_i}(\M) = 2\int_{s\geq \inp{\X_i}{\M-\M^*}}(1-H_{\xi}(s))ds + \inp{\X_i}{\M-\M^*} - \int_{-\infty}^{+\infty}s dH_{\xi}(s).
\end{align*}
This leads to 
\begin{align}
	g_{\X_i}(\M) -g_{\X_i}(\M^*) = 2\int_{0}^{\inp{\X_i}{\M-\M^*}}H_{\xi}(s)ds -  \inp{\X_i}{\M-\M^*}.
	\label{eq14}
\end{align}
Note that $\langle \mathbf{X}_{i}, \mathbf{M}-\mathbf{M}^{*}\rangle\sim\mathcal{N}(0,\Vert \mathbf{M}-\mathbf{M}^{*}\Vert_{\mathrm{F}}^{2}) $ and denote $z:=\langle \mathbf{X}_{i}, \mathbf{M}-\mathbf{M}^{*}\rangle$. Let $f_z(\cdot)$ be the density of $z$. Take expectation of $z$ on both sides of Equation~\ref{eq14}
\begin{align*}
	\mathbb{E}\left[g_{\mathbf{X}_{i}}(\mathbf{M})-g_{\mathbf{X}_{i}}(\mathbf{M}^{*}) \right]
	&=2\int_{-\infty}^{+\infty}\int_{0}^{t}\left(H_{\xi}(s)-0.5\right)f_{z}(t)\,ds\, dz\\
	&=2\int_{-\infty}^{+\infty}\int_{0}^{t}\int_{0}^{s}h_{\xi}(w)f_{z}(t)\,dw\,ds\,dt\\
	&\geq2\int_{-\Vert \mathbf{M} - \mathbf{M}^{*} \Vert_{\mathrm{F}}}^{\Vert \mathbf{M} - \mathbf{M}^{*} \Vert_{\mathrm{F}} }\int_{0}^{t}f_{z}(t)b_0^{-1}s\,ds\,dt\\
	&=b_0^{-1}\int_{-\fro{\M-\M^*}}^{\fro{\M-\M^*}}t^2f_z(t)\,dt = b_0^{-1}\fro{\M-\M^*}^2\int_{-1}^{1}t^2\cdot \frac{1}{\sqrt{2\pi}}e^{-t^2/2}\,dt\\
	&\geq\frac{1}{6b_0}\fro{\M-\M^*}^2,
\end{align*}
where the first inequality follows from Assumption \ref{assump:heavy-tailed} and the last is from $\int_{-1}^1t^2\cdot\frac{1}{\sqrt{2\pi}}e^{-t^2/2}\,dt\geq 1/6$. Therefore,
\begin{align*}
	\EE f(\M)- \EE f(\M^*) = \sum_{i=1}^n \EE g_{\X_i}(\M) - g_{\X_i}(\M^*) \geq \frac{n}{6b_0}\fro{\M-\M^*}^2.
\end{align*}
Invoke Theorem~\ref{thm:empirical process},
\begin{align*}
	f(\M)-f(\M^*)&\geq \EE[f(\M)-f(\M^*)] - C\sqrt{nd_1r}\fro{\M-\M^*}\\
	&\geq \frac{n}{6b_{0}} \fro{\M-\M^*}^2 - C\sqrt{nd_1r}\fro{\M-\M^*}\\
	&\geq\frac{n}{12b_0}\fro{\M-\M^*}^2,
\end{align*}
where the last inequality uses $\fro{\M-\M^{*}}\geq C_2\sqrt{d_1r/n}b_{0}$. This proves $\mus=\frac{n}{12b_0}$.

Finally, from the following lemma, we see that $\Ls \leq C_3nb_1^{-1}$.
And this finishes the proof.

\begin{lemma}[Upper bound for sub-gradient]\label{lemma:upperboundsubgradient:heavytail}
	Let $\M\in\RR^{d_1\times d_2}$ be rank at most $r$ matrix such that $\fro{\M-\M^*}\geq C_2\sqrt{\frac{d_1r}{n}}b_1$. Let $\G\in\partial f(\M)$ be the sub-gradient. Under the event $\bcalE=\{\sup_{\M,\M_1\in\MM_r}|f(\M+\M_1)-f(\M) - \EE(f(\M+\M_1)-f(\M))|\cdot\fro{\M_1}^{-1}\leq C_1\sqrt{nd_1r} \}$, we have $\fror{\G}\leq C_3nb_{1}^{-1}\fro{\M-\M^*}$ for some absolute constant $C_3>0$.
\end{lemma}
\begin{proof}
	Take $\M_1 = \M + \frac{b_1}{2n}\text{SVD}_r(\G)$, where $\text{SVD}_r(\G)$ is the best rank $r$ approximation of $\G$ and then $\rank(\M_1) \leq2r$. We finish the proof via combining lower bound and upper bound of $f(\M_{1})-f(\M)$. First consider $\EE f(\M_1)-\EE f(\M)$.
	Use the same notation as the above proof and similarly,
	\begin{align*}
		g_{\X_{i}}(\M_{1}) - g_{\X_{i}}(\M)=&2\int_{\langle \X_{i}, \M-\M^{*}\rangle}^{\langle \X_{i}, \M-\M^{*}\rangle+\langle \X_{i}, \M_{1}-\M\rangle}\int_{0}^{\xi} h_{\xi}(x)\,dx\,d\xi\\
		\leq& 2b_1^{-1}\int_{\langle \X_{i}, \M-\M^{*}\rangle}^{\langle \X_{i}, \M-\M^{*}\rangle+\langle \X_{i}, \M_{1}-\M\rangle}\xi \,d\xi\\
		=&b_{1}^{-1}\left[(\langle \X_{i}, \M-\M^{*}\rangle+\langle \X_{i}, \M_{1}-\M\rangle )^2 - (\langle \X_{i}, \M-\M^{*}\rangle)^2\right],
	\end{align*}
	where the inequality follows from the upper bound for $h_{\xi}(x)$. Then take expectation on each side and sum up over $i$ :
	\begin{align*}
		\EE f(\M_{1}) - \EE f(\M) =& \sum_{i=1}^{n}\EE[g_{\X_{i}}(\M_{1}) - g_{\X_{i}}(\M) ]\\
		\leq&nb_1^{-1} \EE\left[(\langle \X_{i}, \M-\M^{*}\rangle+\langle \X_{i}, \M_{1}-\M\rangle )^2 - (\langle \X_{i}, \M-\M^{*}\rangle)^2\right]\\
		\leq&nb_1^{-1}\left[ \Vert \M_{1} -\M\Vert_{\mathrm{F}}^{2} + 2\Vert \M_{1} -\M\Vert_{\mathrm{F}} \Vert \M -\M^{*}\Vert_{\mathrm{F}}\right]
	\end{align*}
	Bound $f(\M_{1})-f(\M)$ using Corollary~\ref{cor:thm:empirical process}:
	\begin{align*}
		f(\M_{1}) - f(\M)\leq&\EE f(\M_{1})-\EE f(\M)+C_1\sqrt{ nd_1r}(\Vert \M_{1}-\M^{*}\Vert_{\mathrm{F}} + \Vert \M-\M^{*}\Vert_{\mathrm{F}})\\
		\leq&\EE f(\M_{1})-\EE f(\M)+C_3nb_1^{-1}\Vert \M-\M^{*}\Vert_{\mathrm{F}}(2\Vert \M-\M^{*}\Vert_{\mathrm{F}} + \Vert\M_{1}-\M\Vert_{\mathrm{F}} )\\
		\leq& nb_1^{-1} \Vert \M_{1}-\M\Vert_{\mathrm{F}}^{2} + (2+C_3) nb_1^{-1}\Vert \M-\M^{*}\Vert_{\mathrm{F}} \Vert\M_{1}-\M\Vert_{\mathrm{F}} +2C_3nb_{1}^{-1}\Vert \M-\M^{*}\Vert_{\mathrm{F}}^{2},
	\end{align*}
	where the second inequality is from the condition $\fro{\M-\M^*}\geq C_2\sqrt{\frac{d_1r}{n}}b_1$. Since $\M_1 = \M + \frac{b_1}{2n}\text{SVD}_r(\G)$, it becomes
	\begin{equation}
		\label{eq18}
		\begin{split}
			f(\M+\frac{b_1}{2n}\text{SVD}_{r}(\G))&-f(\M)\\&\leq \frac{b_1}{4n} \Vert\G\Vert_{\mathrm{F,r}}^{2} + (1+C_1/2) \Vert \M-\M^{*}\Vert_{\mathrm{F}} \Vert\G\Vert_{\mathrm{F,r}} +2C_1\frac{n}{b_1}\Vert \M-\M^{*}\Vert_{\mathrm{F}}^{2}.
		\end{split}
	\end{equation}
	On the other hand, by the definition of sub-gradient, $f(\M_{1})-f(\M)$ has lower bound
	\begin{align}\label{eq17}
		f(\M+\frac{b_1}{2n}\text{SVD}_{r}(\G))-f(\M)\geq \frac{b_1}{2n}\fror{\G}^2.
	\end{align}
	Combine Equation~\ref{eq18} with Equation~\ref{eq17} and then solve the quadratic inequality which leads to $$\fror{\G}\leq C_3nb_{1}^{-1}\fro{\M-\M^*}.$$
\end{proof}

\subsection{Proof of Huber loss with heavy-tailed noise Lemma~\ref{lem:heavytail-huber}}
\label{proof:lem:heavytail-huber}
First consider upper bound of sug-gradient $\G\in f(\M)$, where matrix $\M$ has rank at most $r$. We have 
\begin{align*}
	|f(\M+\text{SVD}_r(\G)) - f(\M)| &= |\sum_{i=1}^n\rho_{H,\delta}(\xi_i - \inp{\X_i}{\M+\text{SVD}_r(\G)}) -\rho_{H,\delta}(\xi_i - \inp{\X_i}{\M})|\\
	&\leq \sum_{i=1}^n 2\delta|\inp{\X_i}{\text{SVD}_r(\G)} \leq 4\delta n\fror{\G},
\end{align*}
where the last line comes from  Corollary~\ref{cor: empirical process}. Note that it implies $\Lc=4\delta n$  and the reason has been stated in Section~\ref{proof:lem:Gaussian-l1}.

Then consider the lower bound. Use the inequality $2\delta|x| - \delta^2\leq\rho_{H,\delta}(x)\leq 2\delta\vert x\vert$ and then we have
\begin{align*}
	f(\M)-f(\M^*) &= \sum_{i=1}^n\rho_{H,\delta}(\xi_i - \inp{\X_i}{\M-\M^*}) -\rho_{H,\delta}(\xi_i)\\
	&\geq 2\delta \sum_{i=1}^n(|\xi_i - \inp{\X_i}{\M-\M^*}| - |\xi_i|) - n\delta^2\\
	&\geq\delta\sqrt{2/\pi}n\fro{\M-\M^*} - 4\delta \lone{\boldsymbol{\xi}} - n\delta^2,
\end{align*}
where the last inequality comes from Corollary~\ref{cor: empirical process}. Also, Lemma~\ref{teclem:Contraction of Heavy Tailed Random Variables} proves with high probability  $\frac{1}{n}\lone{\boldsymbol{\xi}}\leq 3\EE\vert\xi\vert=3\gamma$, which implies
\begin{align*}
	f(\M)-f(\M^*) \geq \delta\sqrt{2/\pi}n\fro{\M-\M^*} - 12\delta n\gamma - n\delta^2.
\end{align*}
Therefore when $\fro{\M-\M^*}\geq 24\gamma + 2\delta$, we have $f(\M)-f(\M^*) \geq \frac{\delta n}{4}\fro{\M-\M^*}$. It shows with $\tauc=24\gamma + 2\delta$ we have $\muc=\delta n/4$.

Now consider the second phase. First take expectation against $\xi$ and define $g_{\X_i}(\M) := \EE_\xi[\rho_{H,\delta}(\xi_i - \inp{\X_i}{\M-\M^*})]$. Denote $z = \inp{\X_i}{\M-\M^*}$. Then with a little abuse of notation, define
\begin{align}\label{eq19}
	g(z) &:= 2\delta\int_{\delta+z}^{+\infty} (x - z) \,dH_{\xi}(x) + 2\delta\int_{-\infty}^{z-\delta} (z- x) dH_{\xi}(x) + \int_{z-\delta}^{z+\delta}\left(z -x \right)^{2}dH_{\xi}(x)\\
	&= g_{\X_i}(\M).
\end{align}
Then simple computation shows 
\begin{align*}
	g'(z) &= 2\int_{z-\delta}^{z+\delta}H_{\xi}(x)dx - 2\delta,\\
	g''(z) &= 2(H_{\xi}(z+\delta) - H_{\xi}(z-\delta)).
\end{align*}
Assumption~\ref{assump:heavytail-huber} guarantees $g'(0) = 0$. Therefore the Taylor expansion at $0$ is
\begin{align*}
	g(z) = g(0) + 2\int_{0}^z(H_{\xi}(y+\delta) - H_{\xi}(y-\delta))y\,dy.
\end{align*}
Notice $z =\inp{\X_i}{\M-\M^*}\sim N(0,\fro{\M-\M^*}^2)$ and denote its density by $f_z(\cdot)$. Take expectation w.r.t. $z$ and we have 
\begin{align*}
	\EE g(z) - \EE g(0) &= 2\int_{-\infty}^{\infty}\int_{0}^t(H_{\xi}(y+\delta) - H_{\xi}(y-\delta))yf_z(t)\,dy\,dt\\
	&\geq 2\int_{-\fro{\M-\M^*}}^{\fro{\M-\M^*}}\int_{0}^tf_z(t)(H_{\xi}(y+\delta) - H_{\xi}(y-\delta))y\,dy\,dt\\
	&\geq 4\delta b_0^{-1}\int_{-\fro{\M-\M^*}}^{\fro{\M-\M^*}}t^2f_z(t)\,dt = 4\delta b_0^{-1}\fro{\M-\M^*}^2\int_{-1}^{1} t^2 \frac{1}{\sqrt{2\pi}}\exp(-t^2/2)\,dt\\
	&\geq \frac{2}{3}\delta b_{0}^{-1}\fro{\M-\M^*}^2,
\end{align*}
where the second inequality uses $H_{\xi}(y+\delta)-H_{\xi}(y-\delta)\geq2\delta b_0^{-1}$ when $y\leq \tauc$. And this implies $$\EE f(\M)-\EE f(\M^*)=n\left(\EE g(z) - \EE g(0) \right)\geq 2\delta n b_{0}^{-1}\fro{\M-\M^*}^2.$$ And together with Theorem~\ref{thm:empirical process}, we have 
\begin{align*}
	f(\M)-f(\M^*) &\geq \EE[f(\M)-f(\M^*)] - (C_2+C_3\delta)\sqrt{nd_1r}\fro{\M-\M^*}.
\end{align*}
And thus when $\fro{\M-\M^*}\geq b_0\frac{C_2+C_3\delta}{\delta}\sqrt{d_1r/n}$, we have $$f(\M)-f(\M^*)\geq \frac{1}{3}\delta n b_{0}^{-1}\fro{\M-\M^*}^2.$$ It proves with $\taus=b_0\frac{C_2+C_3\delta}{\delta}\sqrt{d_1r/n}$ it has $\mus= \frac{1}{3}\delta n b_0^{-1}$.

And the bound for $\Ls$ is given as follows.

\begin{lemma}[Upper bound for sub-gradient]\label{lemma:upperboundsubgradient:huber}
	Let $\M\in\RR^{d_1\times d_2}$ be rank at most $r$ matrix and satisfy $\fro{\M-\M^*}\geq c\sqrt{d_1r/n}b_0$. Let $\G\in\partial f(\M)$ be the sub-gradient.  Under the event $\bcalE=\{\sup_{\M,\M_1\in\MM_r}|f(\M+\M_1)-f(\M) - \EE(f(\M+\M_1)-f(\M))|\cdot\fro{\M_1}^{-1}\leq C_1\sqrt{nd_1r} \}$, we have $\fror{\G}\leq C_{\delta}\fro{\M-\M^*}$ for $C_{\delta} = C_1 + C_2\delta$ and some absolute constant $C_1,C_2>0$.
\end{lemma}
\begin{proof}
	Take $\M_1 = \M + \frac{b_1}{4n\delta}\text{SVD}_r(\G)$. First consider $\EE f(\M_1)-\EE f(\M)$. Continue from Equation~\ref{eq19},
	\begin{align*}
		&~~~~\EE_\xi \left[\rho_{H,\delta}(\xi_i - \inp{\X_i}{\M_1-\M^*}) \right] - \EE_{\xi}\left[\rho_{H,\delta}(\xi_i - \inp{\X_i}{\M-\M^*}) \right] \\
		&= 2\int_{\inp{\X_i}{\M-\M^*}}^{\inp{\X_i}{\M_1-\M^*}}[H_{\xi}(y+\delta) - H_{\xi}(y-\delta)]y\,dy\\
		&\leq 4\delta b_1^{-1}\int_{\inp{\X_i}{\M-\M^*}}^{\inp{\X_i}{\M_1-\M^*}}y\,dy\\
		&= 2\delta b_{1}^{-1}[\inp{\X_i}{\M_1-\M}^2 + 2\inp{\X_i}{\M_1-\M}\inp{\X_i}{\M-\M^*}].
	\end{align*}
	Take expectation on each side and sum up ove $i$:
	\begin{align*}
		\EE f(\M_1)-\EE f(\M) &=n \left( \EE_\xi \left[\rho_{H,\delta}(\xi_i - \inp{\X_i}{\M_1-\M^*}) \right] - \EE_{\xi}\left[\rho_{H,\delta}(\xi_i - \inp{\X_i}{\M-\M^*}) \right]\right)\\
		&\leq 2\delta nb_{1}^{-1}[\fro{\M_1-\M}^2 + 2\fro{\M_1-\M}\fro{\M-\M^*}].
	\end{align*}
	Then invoke Crorllary~\ref{cor:thm:empirical process} and bound $f(\M_1)-f(\M)$:
	\begin{align*}
		&{~~~~}f(\M_1)-f(\M)\\&\leq\EE f(\M_1)-\EE f(\M) + C_{\delta}\sqrt{nd_1r}(\fro{\M_{1}-\M^*}+\fro{\M-\M^*})\\
		&\leq 2\delta n b_1^{-1}\fro{\M_1-\M}^2 +\delta n b_1^{-1}\fro{\M_1-\M}\fro{\M-\M^*}+C_{\delta}\sqrt{nd_1r}(\fro{\M_{1}-\M}+2\fro{\M-\M^*})\\
		&\leq2\delta n b_1^{-1}\fro{\M_1-\M}^2 + c_{\delta} n b_1^{-1}\fro{\M_1-\M}\fro{\M-\M^*} + C_{\delta} n b_1^{-1}\fro{\M-\M^*}^2,
	\end{align*}
	where the last inequality uses $\fro{\M-\M^*}\geq \sqrt{d_1r/n}b_1$. Insert $\M_1=\M+\frac{b_1}{4n\delta}\text{SVD}_r(\G)$ and then it is
	\begin{align}\label{eq22}
		f(\M_1)-f(\M)\leq \frac{b_1}{8n\delta}\fror{\G}^2 + \frac{c_{\delta}}{4\delta}\fror{\G}\fro{\M-\M^*} + \frac{C_{\delta}n}{b_1}\fro{\M-\M^*}^2.
	\end{align}
	On the other hand, from the definition of the sub-gradient, we have 
	\begin{align}\label{eq23}
		f(\M_1) -f(\M) = f(\M+\frac{b_1}{4n\delta}\text{SVD}_r(\G)) - f(\M)\geq \frac{b_1}{4n\delta}\fror{\G}^2.
	\end{align}
	Combine Equation~\ref{eq22} with Equation~\ref{eq23} and then the quadratic inequality of $\fror{\G}$ gives $$\fror{\G}\leq C_{\delta}n b_1^{-1}\fro{\M-\M^*}.$$
\end{proof}
\subsection{Proof of quantile loss with heavy-tailed noise Lemma~\ref{lem:heavytail-quantile}}
\label{proof:lem:heavytail-quantile}
Notice that quantile loss function $\rho_{Q,\delta}(\cdot)$ is Lipschitz continuous with $\max\{\delta,1-\delta\}$. 	First consider upper bound of sug-gradient $\G\in f(\M)$, where matrix $\M$ has rank at most $r$. We have 
\begin{align*}
	|f(\M+\text{SVD}_r(\G)) - f(\M)| &= |\sum_{i=1}^n\rho_{Q,\delta}(\xi_i - \inp{\X_i}{\M+\text{SVD}_r(\G)}) -\rho_{Q,\delta}(\xi_i - \inp{\X_i}{\M})|\\
	&\leq \sum_{i=1}^n \max\{\delta,1-\delta\}|\inp{\X_i}{\text{SVD}_r(\G)}|\\
	&\leq 2 \max\{\delta,1-\delta\}n\fror{\G},
\end{align*}
where the last line comes from Corollary~\ref{cor: empirical process}. Note that it verifies $\Lc=2\max\{\delta, 1-\delta\} n$ and its proof is same as Section~\ref{proof:lem:Gaussian-l1}.

Then consider the lower bound. Notice that quantile loss function satisfies triangle inequality, $\rho_{Q,\delta}(x_1+x_2)\leq\rho_{Q,\delta}(x_1)+\rho_{Q,\delta}(x_2)$ and then it has
\begin{align*}
	f(\M)-f(\M^*) &= \sum_{i=1}^n\rho_{Q,\delta}(\xi_i - \inp{\X_i}{\M-\M^*}) -\rho_{Q,\delta}(\xi_i)\\
	&\geq \sum_{i=1}^n\rho_{Q,\delta}(-\inp{\X_i}{\M-\M^*})-\rho_{Q,\delta}(\xi_{i})-\rho_{Q,\delta}(-\xi_{i}) \\
	&\geq\min\{\delta,1-\delta\}\sqrt{1/2\pi}n\fro{\M-\M^*} - \lone{\boldsymbol{\xi}},
\end{align*}
where the last inequality comes from empirical process results Corollary~\ref{cor: empirical process}. Lemma~\ref{teclem:Contraction of Heavy Tailed Random Variables} proves with probability over $1-{}^*\!c_1n^{-\min\{1,\varepsilon\}}$, $\frac{1}{n}\lone{\boldsymbol{\xi}}\leq 3\gamma$. Then it leads to \begin{align*}
	f(\M)-f(\M^*)\geq\min\{\delta,1-\delta\}\sqrt{1/2\pi}n\fro{\M-\M^*} - 3n\gamma.
\end{align*}
Therefore when $\fro{\M-\M^*}\geq 15\gamma\max\{\frac{1}{\delta},\frac{1}{1-\delta}\}$, it has \begin{align*}
	f(\M)-f(\M^*)\geq\frac{1}{6}\min\{\delta,1-\delta\}n\fro{\M-\M^*},
\end{align*}
which proves with $\muc=\frac{1}{6}\min\{\delta,1-\delta\}n$.

Then consider the second phase where $c\sqrt{d_1r/n}b_{0}\leq\fro{\M-\M^*}\leq15\gamma\max\{\frac{1}{\delta},\frac{1}{1-\delta}\}$. Notice for all $t_0\in\RR$, 
\begin{align*}
	\EE\rho_{Q,\delta}(\xi-t_0) &= \delta\int_{s\geq t_0}(s-t_0)\,dH_{\xi}(s) + (1-\delta)\int_{s< t_0}(t_0-s)\,dH_{\xi}(s)\\
	&=\int_{s\geq t_0}(s-t_0)\,dH_{\xi}(s) +(1-\delta) \int_{-\infty}^{+\infty}(t_0-s)\,dH_{\xi}(s)\\
	&=\int_{s\geq t_0}(1-H_{\xi}(s))\,ds + (1-\delta)t_0 - (1-\delta)\int_{-\infty}^{\infty}s \,dH_{\xi}(s).
\end{align*}
Set $t_0 = \inp{\X_i}{\M-\M^*}$ and define $g_{\X_i}(\M) := \EE_{\xi_i}\rho_{Q,\delta}(\xi_i - \inp{\X_i}{\M-\M^*})$. It implies
\begin{align*}
	g_{\X_i}(\M) = \int_{s\geq \inp{\X_i}{\M-\M^*}}(1-H_{\xi}(s))\,ds + (1-\delta)\inp{\X_i}{\M-\M^*} - (1-\delta)\int_{-\infty}^{+\infty}s \,dH_{\xi}(s).
\end{align*}
This leads to 
\begin{align}
	g_{\X_i}(\M) -g_{\X_i}(\M^*) = \int_{0}^{\inp{\X_i}{\M-\M^*}}H_{\xi}(s)\,ds -\delta \inp{\X_i}{\M-\M^*}.
	\label{eq24}
\end{align}
Note that $\langle \mathbf{X}_{i}, \mathbf{M}-\mathbf{M}^{*}\rangle\sim\mathcal{N}(0,\Vert \mathbf{M}-\mathbf{M}^{*}\Vert_{\mathrm{F}}^{2}) $ and denote $z:=\langle \mathbf{X}_{i}, \mathbf{M}-\mathbf{M}^{*}\rangle$. Let $f_z(\cdot)$ be the density of $z$. Take expectation of $z$ on both sides of Equation~\ref{eq24}
\begin{align*}
	\mathbb{E}\left[g_{\mathbf{X}_{i}}(\mathbf{M})-g_{\mathbf{X}_{i}}(\mathbf{M}^{*}) \right]
	&=\int_{-\infty}^{+\infty}\int_{0}^{t}\left(H_{\xi}(s)-\delta\right)f_{z}(t)\,ds\, dt\\
	&=\int_{-\infty}^{+\infty}\int_{0}^{t}\int_{0}^{s}h_{\xi}(w)f_{z}(t)\,dw\,ds\,dt\\
	&\geq b_{0}^{-1}\int_{-\Vert \mathbf{M} - \mathbf{M}^{*} \Vert_{\mathrm{F}}}^{\Vert \mathbf{M} - \mathbf{M}^{*} \Vert_{\mathrm{F}} }\int_{0}^{t}f_{z}(t)s\,ds\,dt\\
	&=b_0^{-1}\int_{-\fro{\M-\M^*}}^{\fro{\M-\M^*}}t^2f_z(t)\,dt = b_0^{-1}\fro{\M-\M^*}^2\int_{-1}^{1}t^2\cdot \frac{1}{\sqrt{2\pi}}e^{-t^2/2}\,dt\\
	&\geq\frac{1}{6b_0}\fro{\M-\M^*}^2,
\end{align*}
where the first inequality follows from Assumption~\ref{assump:heavy-tailed quantile} and the last line is from $\int_{-1}^{1}t^2\cdot\frac{1}{\sqrt{2}}e^{-t^2/2}dt\geq1/6$. Therefore, \begin{align*}
	\EE f(\M)-\EE f(\M^*)=\sum_{i=1}^n \EE g_{\X_i}(\M) - g_{\X_i}(\M^*)\geq\frac{n}{6b_0}\fro{\M-\M^*}^2.
\end{align*}
Invoke Theorem~\ref{thm:empirical process}, \begin{align*}
	f(\M)-f(\M^*)&\geq \EE f(\M)-\EE f(\M^*) - C\sqrt{nd_1r}\fro{\M-\M^*}\\
	&\geq \frac{n}{6b_0}\fro{\M-\M^*}^2-C\sqrt{nd_1r}\fro{\M-\M^*}\\
	&\geq \frac{n}{12b_0}\fro{\M-\M^*}^2,
\end{align*}
where the last inequality uses $\fro{\X-\X^*}\geq c\sqrt{d_1r/n}b_0$. This proves the lower bound in the second phase and shows $\mus=\frac{n}{12b_0}$.

Finally, from the following lemma, we see that $\Ls\leq C_2nb_1^{-1}$. And this finished the proof of the lemma.

\begin{lemma}[Upper bound for sub-gradient]\label{lemma:upperboundsubgradient:quantile}
	Let $\M\in\RR^{d_1\times d_2}$ be at most rank $r$ matrix such that $\fro{\M-\M^*}\geq \sqrt{d_1r/n}b_1$. Let $\G\in\partial f(\M)$ be the sub-gradient.  Under the event $\bcalE=\{\sup_{\M,\M_1\in\MM_r}|f(\M+\M_1)-f(\M) - \EE(f(\M+\M_1)-f(\M))|\cdot\fro{\M_1}^{-1}\leq C_1\sqrt{nd_1r} \}$,, we have $\fror{\G}\leq Cnb_1^{-1}\fro{\M-\M^*}$ for some absolute constant $C>0$.
\end{lemma}
\begin{proof}
	Take $\M_1 = \M + \frac{b_1}{2n}\text{SVD}_r(\G)$, where $\text{SVD}_r(\G)$ is the best rank $r$ approximation of $\G$ and then $\rank(\M_1) \leq2r$. First consider $\EE f(\M_1)-\EE f(\M)$.
	Use the same notation as in the above proof and similarly,
	\begin{align*}
		g_{\X_{i}}(\M_{1}) - g_{\X_{i}}(\M)=&\int_{\langle \X_{i}, \M-\M^{*}\rangle}^{\langle \X_{i}, \M-\M^{*}\rangle+\langle \X_{i}, \M_{1}-\M\rangle}\int_{0}^{\xi} h_{\xi}(x)\,dx\,d\xi\\
		\leq& b_{1}^{-1}\int_{\langle \X_{i}, \M-\M^{*}\rangle}^{\langle \X_{i}, \M-\M^{*}\rangle+\langle \X_{i}, \M_{1}-\M\rangle}\xi \,d\xi\\
		=&\frac{1}{2b_1} \left[(\langle \X_{i}, \M-\M^{*}\rangle+\langle \X_{i}, \M_{1}-\M\rangle )^2 - (\langle \X_{i}, \M-\M^{*}\rangle)^2\right],
	\end{align*}
	where the inequality follows from the upper bound for $h_{\xi}(x)\leq b_{1}^{-1}$. Then take expectation on each side and sum up:
	\begin{align*}
		\EE f(\M_{1}) - \EE f(\M) =& \sum_{i=1}^{n}\EE[g_{\X_{i}}(\M_{1}) - g_{\X_{i}}(\M) ]\\
		\leq&\frac{n}{2b_{1}} \EE\left[(\langle \X_{i}, \M-\M^{*}\rangle+\langle \X_{i}, \M_{1}-\M\rangle )^2 - (\langle \X_{i}, \M-\M^{*}\rangle)^2\right]\\
		\leq&\frac{n}{2b_{1}} \left[ \Vert \M_{1} -\M\Vert_{\mathrm{F}}^{2} + 2\Vert \M_{1} -\M\Vert_{\mathrm{F}} \Vert \M -\M^{*}\Vert_{\mathrm{F}}\right]
	\end{align*}
	Bound $f(\M_{1})-f(\M)$ using Corollary~\ref{cor:thm:empirical process}:
	\begin{align*}
		f(\M_{1}) - f(\M)\leq&\EE f(\M_{1})-\EE f(\M)+C_1\sqrt{ nd_1r}(\Vert \M_{1}-\M^{*}\Vert_{\mathrm{F}} + \Vert \M-\M^{*}\Vert_{\mathrm{F}})\\
		\leq&\EE f(\M_{1})-\EE f(\M)+C_1nb_{1}^{-1}\Vert \M-\M^{*}\Vert_{\mathrm{F}}(2\Vert \M-\M^{*}\Vert_{\mathrm{F}} + \Vert\M_{1}-\M\Vert_{\mathrm{F}} )\\
		\leq& \frac{n}{2b_1} \Vert \M_{1}-\M\Vert_{\mathrm{F}}^{2} + (1+C_1) nb_{1}^{-1}\Vert \M-\M^{*}\Vert_{\mathrm{F}} \Vert\M_{1}-\M\Vert_{\mathrm{F}} +2C_1nb_1^{-1}\Vert \M-\M^{*}\Vert_{\mathrm{F}}^{2},
	\end{align*}
	where the second inequality is from the condition $\fro{\M-\M^*}\geq c\sqrt{\frac{d_1r}{n}}b_{1}$. Since $\M_1 = \M + \frac{b_1}{2n}\text{SVD}_r(\G)$, it becomes
	\begin{equation}
		\label{eq28}
		\begin{split}
			f(\M+\frac{b_1}{2n}\text{SVD}_{r}(\G))&-f(\M)\\&\leq \frac{b_1}{8n} \Vert\G\Vert_{\mathrm{F,r}}^{2} + (1/2+C_1/2) \Vert \M-\M^{*}\Vert_{\mathrm{F}} \Vert\G\Vert_{\mathrm{F,r}} +2C_1\frac{n}{b_1}\Vert \M-\M^{*}\Vert_{\mathrm{F}}^{2}.
		\end{split}
	\end{equation}
	On the other hand, by the definition of sub-gradient, $f(\M_{1})-f(\M)$ has lower bound
	\begin{align}
		f(\M+\frac{b_1}{2n}\text{SVD}_{r}(\G))-f(\M)\geq \frac{b_1^2}{2n}\fror{\G}^2.
		\label{eq27}
	\end{align}
	Combine Equation~\ref{eq28}, Equation~\ref{eq27} and then solve the quadratic inequality which leads to $$\fror{\G}\leq Cnb_{1}^{-1}\fro{\M-\M^*}.$$
\end{proof}
\subsection{Proof of tensor case, absolute loss with Gaussian noise Lemma~\ref{lem:tensor:Gaussian-l1}}
\label{proof:lem:tensor:Gaussian-l1}
\begin{proof}
	We shall use $\OO_{d, r}:=\{\U\in\RR^{d\times r}:\, \U^{\top}\U=\I_{r}\}$ to refer to set of all $d$-by-$r$ matrices with orthonormal columns and $\OO_r$ to be set of dimension $r$ orthonormal matrices.
	
	Note that when $n\geq C(2^m\cdot r_1r_2\cdots r_m+2\sum_{j=1}^{m} r_jd_j)$, with probability exceeding $1-4\exp(-c (2^m\cdot r_1r_2\cdots r_m+2\sum_{j=1}^{m} r_jd_j))$, $\Lc=2n$. The reason is similar to the one presented in Section~\ref{proof:lem:Gaussian-l1} and here we use Gaussian process Corollary~\ref{cor:tensor:empirical process} to cope with tensor.
	
	Next consider the lower bound for $f(\M) - f(\M^*)$. Its expectation $\EE[f(\M)-f(\M^{*})]$ is calculated in Lemma~\ref{teclem:gaussian-l1}. We shall proceed assuming the event $\bcalE = \{\sup_{\M\in\MM_\r}|f(\M)-f(\M^*) - \EE(f(\M)-f(\M^*))|\cdot\fro{\M-\M^*}^{-1}\leq C_2\sqrt{n(2^m\cdot r_1r_2\cdots r_m+2\sum_{j=1}^{m}d_jr_j)} \}$ holds and specifically Theorem \ref{thm:tensor:empirical process} proves $\bcalE$ holds with probability exceeding $1-\exp(-c(2^m\cdot r_1r_2\cdots r_m+2\sum_{j=1}^{m}d_jr_j))$.
	
	We discuss the two phases respectively, namely, phase one when $\fro{\M-\M^*}\geq \sigma$ and phase two when $C\sqrt{( 2^m\cdot r_1r_2\cdots r_m+2\sum_{j=1}^{m}d_jr_j)/n}\sigma\leq\fro{\M-\M^*}\leq \sigma$.
	
	\noindent\textit{Case 1:}
	When $\Vert \mathbf{M}-\mathbf{M}^{*}\Vert_{\mathrm{F}}\geq\sigma$, we have 
	\begin{align*}
		f(\mathbf{M})-f(\mathbf{M}^*)&\geq\mathbb{E}\left[f(\mathbf{M})-f(\mathbf{M}^*) \right] - C_2\sqrt{n(2^m\cdot r_1r_2\cdots r_m+2\sum_{j=1}^{m}d_jr_j)}\Vert \mathbf{M}-\mathbf{M}^{*}\Vert_{\mathrm{F}}\\
		&= n\sqrt{2/\pi}\frac{1}{\sqrt{\Vert \mathbf{M}-\mathbf{M}^{*}\Vert_{\mathrm{F}}^{2}+\sigma^{2}}+\sigma}\Vert \mathbf{M}-\mathbf{M}^{*}\Vert_{\mathrm{F}}^{2} - C_2\sqrt{n(2^m\cdot r_1r_2\cdots r_m+2\sum_{j=1}^{m}d_jr_j)}\Vert \mathbf{M}-\mathbf{M}^{*}\Vert_{\mathrm{F}}\\
		&\geq n\sqrt{2/\pi}\frac{1}{\sqrt{2}+1}\Vert \mathbf{M}-\mathbf{M}^{*}\Vert_{\mathrm{F}} -C_2\sqrt{n(2^m\cdot r_1r_2\cdots r_m+2\sum_{j=1}^{m}d_jr_j)}\Vert \mathbf{M}-\mathbf{M}^{*}\Vert_{\mathrm{F}}\\
		&\geq \frac{1}{12}n\Vert \mathbf{M}-\mathbf{M}^{*}\Vert_{\mathrm{F}},
	\end{align*}
	where the penultimate line is from $\Vert \mathbf{M}-\mathbf{M}^{*}\Vert_{\mathrm{F}}\geq\sigma$ and the last line uses $n\geq C(2^m\cdot r_1r_2\cdots r_m+2\sum_{j=1}^{m}d_jr_j)$ for some large absolute constant $C>0$. It verifies $\muc=\frac{1}{12}n$.

	\noindent\textit{Case 2:}
	When $C\sqrt{(2^m\cdot r_1r_2\cdots r_m+2\sum_{j=1}^{m}d_jr_j)/n}\sigma\leq\Vert \mathbf{M}-\mathbf{M}^{*}\Vert_{\mathrm{F}}<\sigma$, we have \begin{equation*}
		\begin{split}
			f(\mathbf{M})-f(\mathbf{M}^*)&\geq\mathbb{E}\left[f(\mathbf{M})-f(\mathbf{M}^*) \right] - C_2\sqrt{n(2^m\cdot r_1r_2\cdots r_m+2\sum_{j=1}^{m}d_jr_j)}\Vert \mathbf{M}-\mathbf{M}^{*}\Vert_{\mathrm{F}}\\
			&= n\sqrt{2/\pi}\frac{1}{\sqrt{\Vert \mathbf{M}-\mathbf{M}^{*}\Vert_{\mathrm{F}}^{2}+\sigma^{2}}+\sigma}\Vert \mathbf{M}-\mathbf{M}^{*}\Vert_{\mathrm{F}}^{2} - C_2\sqrt{n(2^m\cdot r_1r_2\cdots r_m+2\sum_{j=1}^{m}d_jr_j)}\Vert \mathbf{M}-\mathbf{M}^{*}\Vert_{\mathrm{F}}\\
			&\geq\frac{1}{6}\frac{n}{\sigma}\Vert \mathbf{M}-\mathbf{M}^{*}\Vert_{\mathrm{F}}^{2} -C_2\sqrt{n(2^m\cdot r_1r_2\cdots r_m+2\sum_{j=1}^{m}d_jr_j)}\Vert \mathbf{M}-\mathbf{M}^{*}\Vert_{\mathrm{F}}\\
			&\geq \frac{1}{12}\frac{n}{\sigma}\Vert \mathbf{M}-\mathbf{M}^{*}\Vert_{\mathrm{F}}^{2},
		\end{split}
	\end{equation*}
	where the penultimate line uses $\Vert \mathbf{M}-\mathbf{M}^{*}\Vert_{\mathrm{F}}< \sigma$ and the last line is from $\fro{\M-\M^*}\geq C\sqrt{(2^m\cdot r_1r_2\cdots r_m+2\sum_{j=1}^{m}d_jr_j)/n}\sigma$. It shows $\mus=\frac{1}{12}\frac{n}{\sigma}$.
	
	Finally, from the following Lemma \ref{lemma:tensor:upperboundsubgradient:gaussian}, we see that $\Ls \leq C_4n\sigma^{-1}$. And this finishes the proof of the lemma.
\end{proof}

\begin{lemma}[Upper bound for sub-gradient]\label{lemma:tensor:upperboundsubgradient:gaussian}
	Let tensor $\M\in\RR^{d_1\times \cdots\times d_m}$ have rank at most $\r$ and satisfy $\fro{\M-\M^*}\geq \sqrt{n^{-1}(2^m\cdot r_1r_2\cdots r_m+2\sum_{j=1}^{m}d_jr_j)}\sigma$. Let $\G\in\partial f(\M)$ be the sub-gradient. Under the event $\bcalE=\{\sup_{\M\in\MM_\r,\M_1\in\MM_{2\r}}|f(\M+\M_1)-f(\M) - \EE(f(\M+\M_1)-f(\M))|\cdot\fro{\M_1}^{-1}\leq C_1\sqrt{n(2^m\cdot r_1r_2\cdots r_m+2\sum_{j=1}^{m}d_jr_j)} \}$, we have $\Vert\G\Vert_{\mathrm{F,2\r}}\leq Cn\sigma^{-1}\fro{\M-\M^*}$ for some absolute constant $C>0$.
\end{lemma}
\begin{proof}
	There exists orthonormal matrices $\U_1\in\OO_{d_1, 2r_1},\cdots,\U_m\in\OO_{d_m, 2r_m}$ such that $\Vert\G\Vert_{\mathrm{F,2\r}}=\fro{\G\times_{1} \U_1^{\top} \times_2 \cdots \times_m \U_m^{\top}}$ and denote $\tilde{\G}:= \G\times_{1} \U_1\U_1^{\top} \times_2 \cdots \times_m \U_m\U_m^{\top}$. Take $\M_1 = \M + \frac{\sigma}{2n}\tilde{\G}$. So $\rank(\M_1-\M) \leq 2\r$. Then we have
	\begin{align*}
		\EE f(\M_{1})-\EE f(\M)=&\sqrt{2/\pi}n\frac{\Vert \M_{1}-\M^{*}\Vert_{\mathrm{F}}^{2} - \Vert \M-\M^{*}\Vert_{\mathrm{F}}^{2} }{\sqrt{\sigma^{2} + \Vert \M_{1}-\M^{*}\Vert_{\mathrm{F}}^{2} } + \sqrt{\sigma^2 + \Vert \M-\M^{*}\Vert_{\mathrm{F}}^{2} } }\\
		=&\sqrt{2/\pi}n\frac{\Vert \M_{1}-\M\Vert_{\mathrm{F}}^{2}+2\langle \M-\M^{*}, \M_{1}-\M\rangle }{\sqrt{\sigma^{2} + \Vert \M_{1}-\M^{*}\Vert_{\mathrm{F}}^{2} } + \sqrt{\sigma^2 + \Vert \M-\M^{*}\Vert_{\mathrm{F}}^{2} } }\\
		\leq&\sqrt{2/\pi}\frac{n}{\sigma}\left( \Vert \M_{1}-\M\Vert_{\mathrm{F}}^{2}+2\Vert \M-\M^{*}\Vert_{\mathrm{F}} \Vert\M_{1}-\M\Vert_{\mathrm{F}}\right),
	\end{align*}
	where the first equality comes from Lemma~\ref{teclem:gaussian-l1}. And from Corollary~\ref{cor:thm:tensor:empirical process}, we have 
	\begin{align*}
		f(\M_{1}) - f(\M)\leq&\EE f(\M_{1})-\EE f(\M)+C_1\sqrt{ n(2^m\cdot r_1r_2\cdots r_m+2\sum_{j=1}^{m}d_jr_j)}(\Vert \M_{1}-\M^{*}\Vert_{\mathrm{F}} + \Vert \M-\M^{*}\Vert_{\mathrm{F}})\\
		\leq&\EE f(\M_{1})-\EE f(\M)+C_1\frac{n}{\sigma}\Vert \M-\M^{*}\Vert_{\mathrm{F}}(2\Vert \M-\M^{*}\Vert_{\mathrm{F}} + \Vert\M_{1}-\M\Vert_{\mathrm{F}} )\\
		\leq& \frac{n}{\sigma} \Vert \M_{1}-\M\Vert_{\mathrm{F}}^{2} + (2+C_1) \frac{n}{\sigma}\Vert \M-\M^{*}\Vert_{\mathrm{F}} \Vert\M_{1}-\M\Vert_{\mathrm{F}} +2C_1\frac{n}{\sigma}\Vert \M-\M^{*}\Vert_{\mathrm{F}}^{2},
	\end{align*}
	where the second inequality is from the condition $\fro{\M-\M^*}\geq \sqrt{\frac{2^m\cdot r_1r_2\cdots r_m+2\sum_{j=1}^{m}d_jr_j}{n}}\sigma$. Since $\M_1 = \M + \frac{\sigma}{2n}\tilde{\G}$, we have 
	\begin{equation}\label{eq33}
		\begin{split}
			f(\M+\frac{\sigma}{2n}\tilde{\G})&-f(\M)\\&\leq \frac{\sigma}{4n} \Vert\G\Vert_{\mathrm{F,2\r}}^{2} + (1+C_1/2) \Vert \M-\M^{*}\Vert_{\mathrm{F}} \Vert\G\Vert_{\mathrm{F,2\r}} +2C_1\frac{n}{\sigma}\Vert \M-\M^{*}\Vert_{\mathrm{F}}^{2}.
		\end{split}
	\end{equation}
	On the other hand, by the definition of sub-gradient, we have
	\begin{align}\label{eq34}
		f(\M+\frac{\sigma}{2n}\tilde{\G})-f(\M)\geq \frac{\sigma}{2n}\Vert\G\Vert_{\mathrm{F,2\r}}^2.
	\end{align}
	Combine Equation~\ref{eq33} with Equation~\ref{eq34} and by solving the quadratic inequality we get $$\Vert\G\Vert_{\mathrm{F,2\r}}\leq Cn\sigma^{-1}\fro{\M-\M^*}.$$
\end{proof}
\subsection{Proof of tensor case, absolute value loss with heavy-tailed noise Lemma~\ref{lem:tensor:heavytail-l1}}

The $\Lc=2n$ proof is the same as the one given in Section~\ref{proof:lem:tensor:Gaussian-l1}. Now we focus on lower bound of $f(\M)-f(\M^*)$.

First consider phase one when $\fro{\X-\X^*}\geq 30\EE\vert\xi\vert=30\gamma$. Use triangle inequality and then get
\begin{align}
	f(\mathbf{M})-f(\mathbf{M}^{*}) &= \sum_{i=1}^{n}|\xi_i-\inp{\M-\M^*}{\X_i}|-\sum_{i=1}^{n}|\xi_i| \geq\sum_{i=1}^{n}|\inp{\M-\M^*}{\X_i}|-2\Vert \boldsymbol{\xi}\Vert_{1}.
	\label{eq35}
\end{align}
Corollary~\ref{cor:tensor:empirical process} proves $ \sum_{i=1}^{n}|\inp{\M-\M^*}{\X_i}| \geq \sqrt{1/2\pi}n\fro{\M-\M^*}$ and Lemma~\ref{teclem:Contraction of Heavy Tailed Random Variables} proves with probability over $1- {}^*\!c_1n^{-\min\{1,\varepsilon\}}$, $\frac{1}{n}\lone{\boldsymbol{\xi}}\leq 3\EE\vert\xi\vert=3\gamma$. Then combined with $\fro{\M-\M^*}\geq 30\gamma$, Equation~\ref{eq35} becomes
\begin{align*}
	f(\mathbf{M})-f(\mathbf{M}^{*})\geq \sqrt{1/2\pi}n\Vert \mathbf{M}-\mathbf{M}^{*}\Vert_{\mathrm{F}}-6n\gamma\geq \frac{n}{6}\fro{\M-\M^*},
\end{align*}
which verifies $\muc = \frac{n}{6}$.

Then consider the second phase when $C_2\sqrt{(2^m\cdot r_1r_2\cdots r_m+2\sum_{j=1}^{m}d_jr_j)/n}b_0 \leq \fro{\M-\M^*}<30\gamma$. To bound expectation of $f(\M)-f(\M^*)$, we first consider the conditional expectation $\EE_\xi f(\M)-\EE_\xi f(\M^*)$. Notice for all $t_0\in\RR$,
\begin{align*}
	\EE|\xi-t_0| &= \int_{s\geq t_0}(s-t_0)dH_{\xi}(s) + \int_{s< t_0}(t_0-s)dH_{\xi}(s)\\
	&=2\int_{s\geq t_0}(s-t_0)dH_{\xi}(s) + \int_{-\infty}^{+\infty}(t_0-s)dH_{\xi}(s)\\
	&=2\int_{s\geq t_0}(1-H_{\xi}(s))ds + t_0 - \int_{-\infty}^{\infty}s dH_{\xi}(s).
\end{align*}
Set $t_0 = \inp{\X_i}{\M-\M^*}$ and define $g_{\X_i}(\M) := \EE_{\xi_i}|\xi_i - \inp{\X_i}{\M-\M^*}|$. It implies
\begin{align*}
	g_{\X_i}(\M) = 2\int_{s\geq \inp{\X_i}{\M-\M^*}}(1-H_{\xi}(s))ds + \inp{\X_i}{\M-\M^*} - \int_{-\infty}^{+\infty}s dH_{\xi}(s).
\end{align*}
This leads to 
\begin{align}
	g_{\X_i}(\M) -g_{\X_i}(\M^*) = 2\int_{0}^{\inp{\X_i}{\M-\M^*}}H_{\xi}(s)ds -  \inp{\X_i}{\M-\M^*}.
	\label{eq36}
\end{align}
Note that $\langle \mathbf{X}_{i}, \mathbf{M}-\mathbf{M}^{*}\rangle$ follows Gaussian distribution with mean zero variance $\Vert \mathbf{M}-\mathbf{M}^{*}\Vert_{\mathrm{F}}^{2}$ and denote it as $z:=\langle \mathbf{X}_{i}, \mathbf{M}-\mathbf{M}^{*}\rangle$. Let $f_z(\cdot)$ be the density of $z$. Take expectation of $z$ on both sides of Equation~\ref{eq36}
\begin{align*}
	\mathbb{E}\left[g_{\mathbf{X}_{i}}(\mathbf{M})-g_{\mathbf{X}_{i}}(\mathbf{M}^{*}) \right]
	&=2\int_{-\infty}^{+\infty}\int_{0}^{t}\left(H_{\xi}(s)-0.5\right)f_{z}(t)\,ds\, dz\\
	&=2\int_{-\infty}^{+\infty}\int_{0}^{t}\int_{0}^{s}h_{\xi}(w)f_{z}(t)\,dw\,ds\,dt\\
	&\geq2\int_{-\Vert \mathbf{M} - \mathbf{M}^{*} \Vert_{\mathrm{F}}}^{\Vert \mathbf{M} - \mathbf{M}^{*} \Vert_{\mathrm{F}} }\int_{0}^{t}f_{z}(t)b_0^{-1}s\,ds\,dt\\
	&=b_0^{-1}\int_{-\fro{\M-\M^*}}^{\fro{\M-\M^*}}t^2f_z(t)\,dt = b_0^{-1}\fro{\M-\M^*}^2\int_{-1}^{1}t^2\cdot \frac{1}{\sqrt{2\pi}}e^{-t^2/2}\,dt\\
	&\geq\frac{1}{6b_0}\fro{\M-\M^*}^2,
\end{align*}
where the first inequality follows from Assumption \ref{assump:heavy-tailed} and the last is from $\int_{-1}^1t^2\cdot\frac{1}{\sqrt{2\pi}}e^{-t^2/2}\,dt\geq 1/6$. Therefore,
\begin{align*}
	\EE f(\M)- \EE f(\M^*) = \sum_{i=1}^n \EE [g_{\X_i}(\M) - g_{\X_i}(\M^*)] \geq \frac{n}{6b_0}\fro{\M-\M^*}^2.
\end{align*}
Invoke Theorem~\ref{thm:tensor:empirical process},
\begin{align*}
	f(\M)-f(\M^*)&\geq \EE[f(\M)-f(\M^*)] - C\sqrt{n(2^m\cdot r_1r_2\cdots r_m+2\sum_{j=1}^{m}d_jr_j )}\fro{\M-\M^*}\\
	&\geq \frac{n}{6b_{0}} \fro{\M-\M^*}^2 - C\sqrt{n(2^m\cdot r_1r_2\cdots r_m+2\sum_{j=1}^{m}d_jr_j)}\fro{\M-\M^*}\\
	&\geq\frac{n}{12b_0}\fro{\M-\M^*}^2,
\end{align*}
where the last inequality uses $\fro{\M-\M^{*}}\geq C_2\sqrt{(2^m\cdot r_1r_2\cdots r_m+2\sum_{j=1}^{m}d_jr_j)/n}b_{0}$. This proves $\mus=\frac{n}{12b_0}$.

Finally, from the following lemma, we see that $\Ls \leq C_3nb_1^{-1}$.
And this finishes the proof.

\begin{lemma}[Upper bound for sub-gradient]\label{lemma:tensor:upperboundsubgradient:heavytail}
	Let $\M\in\RR^{d_1\times \cdots\times d_m}$ be rank at most $\r$ tensor such that $\fro{\M-\M^*}\geq C_2\sqrt{n^{-1}(2^m\cdot r_1r_2\cdots r_m+2\sum_{j=1}^{m}d_jr_j)}b_1$. Let $\G\in\partial f(\M)$ be the sub-gradient. Under the event $\bcalE=\{\sup_{\M\in\MM_\r,\M_1\in\MM_{2\r}}|f(\M+\M_1)-f(\M) - \EE(f(\M+\M_1)-f(\M))|\cdot\fro{\M_1}^{-1}\leq C_1\sqrt{n(2^m\cdot r_1r_2\cdots r_m+2\sum_{j=1}^{m}d_jr_j)} \}$, we have $\Vert\G\Vert_{\mathrm{F,2\r}}\leq C_3nb_{1}^{-1}\fro{\M-\M^*}$ for some absolute constant $C_3>0$.
\end{lemma}
\begin{proof}
	There exists orthonormal matrices $\U_1\in\OO_{d_1, 2r_1},\cdots,\U_m\in\OO_{d_m, 2r_m}$ such that $\Vert\G\Vert_{\mathrm{F,2\r}}=\fro{\G\times_{1} \U_1^{\top} \times_2 \cdots \times_m \U_m^{\top}}$ and denote $\tilde{\G}:= \G\times_{1} \U_1\U_1^{\top} \times_2 \cdots \times_m \U_m\U_m^{\top}$. Take $\M_1 = \M + \frac{b_1}{2n}\tilde{\G}$ and then $\rank(\M_1-\M) \leq2\r$. We finish the proof via combining lower bound and upper bound of $f(\M_{1})-f(\M)$. First consider $\EE f(\M_1)-\EE f(\M)$.
	Use the same notation as the above proof and similarly,
	\begin{align*}
		g_{\X_{i}}(\M_{1}) - g_{\X_{i}}(\M)=&2\int_{\langle \X_{i}, \M-\M^{*}\rangle}^{\langle \X_{i}, \M-\M^{*}\rangle+\langle \X_{i}, \M_{1}-\M\rangle}\int_{0}^{\xi} h_{\xi}(x)\,dx\,d\xi\\
		\leq& 2b_1^{-1}\int_{\langle \X_{i}, \M-\M^{*}\rangle}^{\langle \X_{i}, \M-\M^{*}\rangle+\langle \X_{i}, \M_{1}-\M\rangle}\xi \,d\xi\\
		=&b_{1}^{-1}\left[(\langle \X_{i}, \M-\M^{*}\rangle+\langle \X_{i}, \M_{1}-\M\rangle )^2 - (\langle \X_{i}, \M-\M^{*}\rangle)^2\right],
	\end{align*}
	where the inequality follows from the upper bound for $h_{\xi}(x)$. Then take expectation of $\X_{i}$ on each side and sum up over $i$ :
	\begin{align*}
		\EE f(\M_{1}) - \EE f(\M) =& \sum_{i=1}^{n}\EE[g_{\X_{i}}(\M_{1}) - g_{\X_{i}}(\M) ]\\
		\leq&nb_1^{-1} \EE\left[(\langle \X_{i}, \M-\M^{*}\rangle+\langle \X_{i}, \M_{1}-\M\rangle )^2 - (\langle \X_{i}, \M-\M^{*}\rangle)^2\right]\\
		\leq&nb_1^{-1}\left[ \Vert \M_{1} -\M\Vert_{\mathrm{F}}^{2} + 2\Vert \M_{1} -\M\Vert_{\mathrm{F}} \Vert \M -\M^{*}\Vert_{\mathrm{F}}\right]
	\end{align*}
	Then bound $f(\M_{1})-f(\M)$ with empirical process Corollary~\ref{cor:thm:tensor:empirical process}:
	\begin{align*}
		f(\M_{1}) - f(\M)\leq&\EE f(\M_{1})-\EE f(\M)+C_1\sqrt{ n(2^m\cdot r_1r_2\cdots r_m+2\sum_{j=1}^{m}d_jr_j)}(\Vert \M_{1}-\M^{*}\Vert_{\mathrm{F}} + \Vert \M-\M^{*}\Vert_{\mathrm{F}})\\
		\leq&\EE f(\M_{1})-\EE f(\M)+C_3nb_1^{-1}\Vert \M-\M^{*}\Vert_{\mathrm{F}}(2\Vert \M-\M^{*}\Vert_{\mathrm{F}} + \Vert\M_{1}-\M\Vert_{\mathrm{F}} )\\
		\leq& nb_1^{-1} \Vert \M_{1}-\M\Vert_{\mathrm{F}}^{2} + (2+C_3) nb_1^{-1}\Vert \M-\M^{*}\Vert_{\mathrm{F}} \Vert\M_{1}-\M\Vert_{\mathrm{F}} +2C_3nb_{1}^{-1}\Vert \M-\M^{*}\Vert_{\mathrm{F}}^{2},
	\end{align*}
	where the second inequality is from the condition $\fro{\M-\M^*}\geq C_2\sqrt{n^{-1}(2^m\cdot r_1r_2\cdots r_m+2\sum_{j=1}^{m}d_jr_j)}b_1$. Since $\M_1 = \M + \frac{b_1}{2n}\tilde{\G}$, it becomes
	\begin{equation}
		\label{eq37}
		\begin{split}
			f(\M+\frac{b_1}{2n}\tilde{\G})&-f(\M)\\&\leq \frac{b_1}{4n} \Vert\G\Vert_{\mathrm{F,2\r}}^{2} + (1+C_1/2) \Vert \M-\M^{*}\Vert_{\mathrm{F}} \Vert\G\Vert_{\mathrm{F,2\r}} +2C_1\frac{n}{b_1}\Vert \M-\M^{*}\Vert_{\mathrm{F}}^{2}.
		\end{split}
	\end{equation}
	On the other hand, by the definition of sub-gradient, $f(\M_{1})-f(\M)$ has lower bound
	\begin{align}\label{eq38}
		f(\M+\frac{b_1}{2n}\tilde{\G})-f(\M)\geq \frac{b_1}{2n}\Vert\G\Vert_{\mathrm{F,2\r}}^2.
	\end{align}
	Combine Equation~\ref{eq37} with Equation~\ref{eq38} and then solve the quadratic inequality which leads to $$\Vert\G\Vert_{\mathrm{F,2\r}}\leq C_3nb_{1}^{-1}\fro{\M-\M^*}.$$
\end{proof}
\section{Proof of Initialization Theorem~\ref{thm:initialization}}
\label{proof:initialization}
Notice that
\begin{align*}
	\Vert \frac{1}{n}\sum_{i=1}^{n}Y_{i}\X_{i} -\M^{*}\Vert=&\Vert \frac{1}{n}\sum_{i=1}^{n}(\xi_{i}+\langle \X_{i}, \M^{*}\rangle)\X_{i} -\M^{*}\Vert\\
	\leq&\Vert \frac{1}{n}\sum_{i=1}^{n}\langle \X_{i}, \M^{*}\rangle\X_{i}-\M^{*}\Vert + \frac{1}{n}\Vert \sum_{i=1}^{n} \xi_{i}\X_{i}\Vert.
\end{align*}
Use Theorem~\ref{thm:Bernstein Ineq} to bound the first term. Note that $$\EE\left[\langle \X_{i}, \M^{*}\rangle\X_{i}-\M^{*} \right]=0,$$
$$\Vert \Lambda_{\max}\left(\langle \X_{i}, \M^{*}\rangle\X_{i}-\M^{*} \right)\Vert_{\Psi_1}\leq \Vert\langle \X_{i}, \M^{*}\rangle\Vert_{\Psi_2} \Vert \Lambda_{\max}\left( \X_{i} \right)\Vert_{\Psi_2}=c\sqrt{d_1}\fro{\M^*},$$
$$\EE\left[\left(\langle \X_{i}, \M^{*}\rangle\X_{i}-\M^{*} \right)\left(\langle \X_{i}, \M^{*}\rangle\X_{i}-\M^{*}\right)^{\top}\right]=\M^*\M^{*\top}+d_2\fro{\M^*}^2\cdot \I_{d_1},$$
$$\EE\left[\left(\langle \X_{i}, \M^{*}\rangle\X_{i}-\M^{*} \right)^{\top}\left(\langle \X_{i}, \M^{*}\rangle\X_{i}-\M^{*}\right)\right]=\M^{*\top}\M^*+d_1\fro{\M^*}^2\cdot \I_{d_2}.$$
Thus,
\begin{align*}
	S^2:=&\max\left\{\Lambda_{\max}\left(\sum_{i=1}^n\EE\left[\left(\langle \X_{i}, \M^{*}\rangle\X_{i}-\M^{*} \right)\left(\langle \X_{i}, \M^{*}\rangle\X_{i}-\M^{*}\right)^{\top}\right]\right)\big/n,\right.\\
	&{~~~~~~~~}\left.\Lambda_{\max}\left(\sum_{i=1}^n\EE\left[\left(\langle \X_{i}, \M^{*}\rangle\X_{i}-\M^{*} \right)^{\top}\left(\langle \X_{i}, \M^{*}\rangle\X_{i}-\M^{*}\right)\right] \right)\big/n\right\}\\
	&=\sigma_1^2+d_{1}\fro{\M^*}^2.
\end{align*} 

Then, by Theorem~\ref{thm:Bernstein Ineq}, we have
\begin{align*}
	\Vert \frac{1}{n}\sum_{i=1}^{n}\langle \X_{i}, \M^{*}\rangle\X_{i}-\M^{*}\Vert&\leq c\sqrt{\sigma_1^2+d_1\fro{\M^*}^2}\sqrt\frac{\log d_1}{n} + c\frac{\log d_1}{n}\log\left(\frac{\sqrt{d_1}\fro{\M^*}}{\sqrt{\sigma_1^2+d_1\fro{\M^*}^2 }}\right)\\
	&\leq c_0\sqrt\frac{d_1r\log d_1}{n}\sigma_{1}
\end{align*}
holds with  probability over $1-c_1d_{1}^{-1}$.
Then consider the second term. Note that with probability over $1-c_2\exp(-d_1)$, the following holds
\begin{align}\label{eq15}
	\frac{1}{n}\Vert \sum_{i=1}^{n} \xi_{i}\X_{i}\Vert\leq c_3\frac{\sqrt{d_1}}{n}\left(\sum_{i=1}^{n}\xi_{i}^2\right)^{1/2}.
\end{align}
Lemma~\ref{teclem:bound of second moment heavy tailed} analyzes tail bound for sum of squared heavy-tailed random variables and  proves with probability exceeding $1-(\log d_1)^{-1}$, the following holds
\begin{align*}
	\sum_{i=1}^{n}\xi_{i}^2\leq  (2n\gamma_1\log d_1)^{\frac{2}{1+\epsilon}}
	%	c_5n^{\frac{2}{1+\varepsilon}}\left(\log d_1\right)^{\frac{2}{1+\varepsilon}}\gamma^2.
\end{align*}
Instert it into Equation~\ref{eq15}: \begin{align*}
	\frac{1}{n}\Vert \sum_{i=1}^{n} \xi_{i}\X_{i}\Vert\leq  \tilde c_3\sqrt{d_1}n^{-\frac{\varepsilon}{1+\varepsilon}}\left(\log d_1\right)^{\frac{1}{1+\varepsilon}}\gamma_1^{\frac{1}{1+\epsilon}}.
\end{align*}
Thus, we have 
\begin{align*}
	\Vert \frac{1}{n}\sum_{i=1}^{n}Y_{i}\X_{i} -\M^{*}\Vert\leq c_0\sqrt{\frac{d_1r\log d_1}{n}}\sigma_{1} + \tilde c_3\sqrt{d_1}n^{-\frac{\varepsilon}{1+\varepsilon}}\left(\log d_1\right)^{\frac{1}{1+\varepsilon}}\gamma_1^{\frac{1}{1+\epsilon}},
\end{align*}
which implies $\Vert \frac{1}{n}\sum_{i=1}^{n}Y_{i}\X_{i} -\M^{*}\Vert\leq \sigma_{r}/4$. Then by Lemma~\ref{teclem:perturbation}, we have \begin{align*}
	\Vert \frac{1}{n}\operatorname{SVD}_r(\sum_{i=1}^{n}Y_{i}\X_{i}) -\M^{*}\Vert\leq c_4\sqrt{\frac{d_1r\log d_1}{n}}\sigma_{1} + c_5 \sqrt{d_1}n^{-\frac{\varepsilon}{1+\varepsilon}}\left(\log d_1\right)^{\frac{1}{1+\varepsilon}}\gamma_1^{\frac{1}{1+\epsilon}}.
\end{align*}
Note that $\rank(\M_0-\M^*)\leq 2r$. Then, it has\begin{align*}
	\Vert \M_{0} - \M^{*}\Vert_{\mathrm{F}}\leq&\sqrt{2r}\Vert \M_{0}-\M^{*}\Vert\\
	=&\sqrt{2r}\Vert\operatorname{SVD}_{r}(\frac{1}{n}\sum_{i=1}^{n}Y_{i}\X_{i})-\M^{*}\Vert_{2}\\
	\leq&\tilde c_4\sqrt{\frac{d_1r^2\log d_1}{n}}\sigma_{1} + \tilde c_5\sqrt{d_1r}n^{-\frac{\varepsilon}{1+\varepsilon}}\left(\log d_1\right)^{\frac{1}{1+\varepsilon}}\gamma_1^{\frac{1}{1+\epsilon}},
\end{align*}
which completes the proof.
\begin{theorem}(Bernstein's  Inequality, Proposition 2 in \cite{koltchinskii2011nuclear})\label{thm:Bernstein Ineq}
	Let $\A,\A_1,\cdots,\A_n$ be \iid $p\times q$ matrices that satisfy for some $\alpha\geq1$ (and all $i$) $$\EE \A_i=0,\qquad \Vert \Lambda_{\max}(\A_i)\Vert_{\Psi_\alpha}=:K<+\infty.$$
	Define $$S^2:=\max\left\{\Lambda_{\max}\left(\sum_{i=1}^n\EE\A_{i}\A_{i}^{\top}\right)\big/n, \Lambda_{\max}\left(\sum_{i=1}^n\EE\A_{i}^{\top}\A_{i}\right)\big/n\right\}.$$
	Then for some constant $C>0$ and for all $t>0$,
	\begin{align*}
		\PP\left(\Lambda_{\max}\left(\sum_{i=1}^{n}\A_{i}\right)\big/n\geq CS\sqrt{\frac{t+\log(p+q)}{n}}+CK\log^{1/\alpha}\left(\frac{K}{S}\right)\frac{t+\log(p+q)}{n} \right)\leq\exp(-t).
	\end{align*}
\end{theorem}
\section{Proof of Tensor Regression Initialization Theorem~\ref{thm:init:tensor}}
For any two orthogonal matrices $\U,\V\in\OO_{d,r}$, their distance with optimal rotation is defined to be $d(\U,\V) := \min_{\R\in\OO_r}\fro{\U-\V\R}$. First estimate distance between $\U_i^{(0)}$ and $\U_i^*$. Denote 
\begin{align}\label{eq:tildeN}
	\hat{\N} = \frac{1}{2n(n-1)}\sum_{1\leq i<i'\leq n}^n\tilde Y_i\tilde Y_{i'}(\X_i\X_{i'}^T + \X_{i'}\X_i^T).
\end{align}

From Lemma \ref{lemma:init:tensor} and Wedin's sin$\Theta$ theorem, we obtain that with probability exceeding $1 - c_0m\dmax^{-10}$, for all $i\in[m]$,
\begin{align}\label{eq:dist}
	d(\U_i^{(0)},\U_i^*)\leq \sqrt{2r_i}\op{\tilde\N_i - \M_{(i)}^{*}\M_{(i)}^{*\top}}/\mins^2 \leq  C_1\sqrt{\log \dmax}
	\frac{(d^*)^{1/4}r_i}{n^{1/2}}\kappa^2 + C_2\sqrt{\log \dmax}
	\frac{(d^*)^{1/2}r_i^{1/2}}{n}\cdot \frac{\op{\xi}_{2+\eps}^2}{\mins^2}.
\end{align}
Decompose $\M_0 - \M^*$ and we have
\begin{align}\label{eq:m0-m}
	\fro{\M_0-\M^*} \leq \fro{\C^{(0)}-\C^*}
	%	d(\C^{(0)}, \C^*) 
	+ m\maxs \max_{j=1,\ldots,m} d(\U_j^{(0)},\U_{j}^{*}).
\end{align}
We now bound $\fro{\C^{(0)}-\C^*}$. Definition of $\C^{(0)}$ implies $\gcl(\C^{(0)}\times_i\U_i^{(0)}) = 0$. Then expand square of $\fro{\C^{(0)}-\C^*}$.
\begin{align}\label{eq:c0-c}
	\fro{\C^{(0)}-\C^*}^2 &= \inp{\gcL(\C^{(0)}\times_{j\in[m]}\U_j^{(0)}) - \gcL(\C^*\times_{j\in[m]}\U_j^{(0)})}{\C^{(0)}-\C^*}\notag\\
	&= \underbrace{\inp{(\gcL-\gcl)(\C^{(0)}\times_{j\in[m]}\U_j^{(0)})}{\C^{(0)}-\C^*}}_{\beta_1} - \underbrace{\inp{\gcL(\C^*\times_{j\in[m]}\U_j^{(0)})}{\C^{(0)}-\C^*}}_{\beta_2}.
\end{align}
The first term on the right hand side can be analyzed as follows:
\begin{align*}
	&\inp{(\gcL-\gcl)(\C^{(0)}\times_{j\in[m]}\U_j^{(0)})}{\C^{(0)}-\C^*} =
	\inp{(\gL-\gl)(\C^{(0)}\times_{j\in[m]}\U_j^{(0)})}{(\C^{(0)}-\C^*)\times_{j\in[m]}\U_j^{(0)}} \\
	&{~~~~~~~~~~~~~}= \inp{(\gL-\EE_\xi\gl)(\C^{(0)}\times_{j\in[m]}\U_j^{(0)})}{(\C^{(0)}-\C^*)\times_{j\in[m]}\U_j^{(0)}}\\
	&{~~~~~~~~~~~~~~~~~~~~~~~~~}+ \inp{(\EE_\xi\gl-\gl)(\C^{(0)}\times_{j\in[m]}\U_j^{(0)})}{(\C^{(0)}-\C^*)\times_{j\in[m]}\U_j^{(0)}}. 
\end{align*}
Notice that the first term is 
\begin{align*}
	&{~~~~}\inp{(\gL-\EE_\xi\gl)(\C^{(0)}\times_{j\in[m]}\U_j^{(0)})}{(\C^{(0)}-\C^*)\times_{j\in[m]}\U_j^{(0)}}\\
	&=\frac{1}{n}\sum_{i=1}^{n} \inp{\M_0-\M^*}{\X_{i}}\inp{(\C^{(0)}-\C^*)\times_{j\in[m]}\U_j^{(0)} }{\X_{i}} - \inp{\M_0-\M^*}{(\C^{(0)}-\C^*)\times_{j\in[m]}\U_j^{(0)}}.
\end{align*}
Using a standard empirical process argument, we get as long as $n\gtrsim \sum_j^m d_jr_j + r_1r_2\cdots r_m$, with probability exceeding $1-\exp(-C( \sum_j^m d_jr_j + r_1r_2\cdots r_m))$,
$$\sup_{\A,\B\in\MM_{\r},\fro{\A-\M^*},\fro{\B}\leq 1}|\frac{1}{n}\sum_{i=1}^n\inp{\A-\M^*}{\X_i}\inp{\B}{\X_i} - \inp{\A-\M^*}{\B}|\lesssim \sqrt{\frac{\sum_{j=1}^{m} d_jr_j + r_1r_2\cdots r_m}{n}},$$
by which the first term could be bounded with: \begin{equation}\label{eq44}
	\begin{split}
		&~~~~| \inp{(\gL-\EE_\xi\gl)(\C^{(0)}\times_{j\in[m]}\U_j^{(0)})}{(\C^{(0)}-\C^*)\times_{j\in[m]}\U_j^{(0)}}|\\
		&\leq c_0 \sqrt{\frac{\sum_{j=1}^{m} d_jr_j + r_1r_2\cdots r_m}{n}} \fro{\M_0-\M^*}\fro{\C^{(0)}-\C^*}.
	\end{split}
\end{equation}
On the other hand, 
\begin{align}\label{eq41}
	\inp{(\EE_y\gl-\gl)(\C^{(0)}&\times_{j\in[d]}\U_j^{(0)})}{(\C^{(0)}-\C^*)\times_{j\in[d]}\U_j^{(0)T}} = \inp{\frac{1}{n}\sum_{i=1}^n\xi_i\X_i}{(\C^{(0)}-\C^*)\times_{j\in[d]}\U_j^{(0)}}.
	%	&\leq \frac{1}{n}\op{\sum_{i}\xi_i\X_i}\cdot\nuc{(\C^{(0)}-\C^*)\times_i\U_i^{(0)}}\\
	%	&\lesssim_m \frac{\sqrt{r^*}}{n}\op{\sum_{i}\xi_i\X_i}\fro{\C^{(0)}-\C^*}.
\end{align}
First fix the noise term $\{\xi_i\}_{i=1}^n$.
Then with a standard covering argument, we have with probability exceeding $1-c_1'\exp(-(\sum_{j=1}^{m}r_jd_j+r_1r_2\cdots r_m))$, 
\begin{align}\label{eq42}
	\sup_{\A\in\MM_{\r},\fro{\A}\leq 1} \inp{\sum_{i}^{n}\xi_i\X_i}{\A}\leq c_1\sqrt{\sum_{j=1}^m d_jr_j + r_1r_2\cdots r_m}\left(\sum_{i=1}^n\xi_i^2\right)^{1/2},
\end{align}
where $c_1,c_1'>0$ are some constants. By Lemma~\ref{teclem:Contraction of Heavy Tailed Random Variables}, we obtain upper bound of sum of heavy-tailed variables,
\begin{align}\label{eq43}
	\PP\left(\sum_{i=1}^n\xi_i^2\geq 3n\op{\xi}_{2+\eps}^2\right)\leq 1-c_2'n^{-\min\{\eps/2,1\}}.
\end{align}
Insert Equation~\ref{eq42}, \ref{eq43} into Equation~\ref{eq41} and it leads to
\begin{align}\label{eq45}
	&|\inp{(\EE_y\gl-\gl)(\C^{(0)}\times_{j\in[m]}\U_j^{(0)})}{(\C^{(0)}-\C^*)\times_{j\in[m]}\U_i^{(0)T}}|\notag\\
	&\hspace{3cm}\leq 2\sqrt{\frac{\sum_{j=1}^md_jr_j + r_1\cdots r_m}{n}}\cdot\op{\xi}_{2+\eps}\fro{\C^{(0)}-\C^*}.
\end{align} 
with probability exceeding $1-c_1'\exp(-(\sum_{j=1}^{m}r_jd_j+r_1r_2\cdots r_m))-c_2'n^{-\min\{\eps/2,1\}}$. Then use Equation~\ref{eq44}, \ref{eq45} to bound the first term on the R.H.S. of Equation~\ref{eq:c0-c}:
\begin{align}\label{eq:beta1}
	|\beta_1| &\leq C_1\sqrt{\frac{\sum_{j=1}^{m}d_jr_j + r_1r_2\cdots r_m}{n}}(\fro{\M_0-\M^*} + \op{\xi}_{2+\eps})\fro{\C^{(0)} - \C^*}\notag\\
	&\leq C_1\sqrt{\frac{\sum_{j=1}^{m}d_jr_j + r_1r_2\cdots r_m}{n}}\bigg(\fro{\C^{(0)}-\C^*} + m\maxs\max_{j=1,\ldots,m} d(\U_{j}^{(0)}, \U_j^*)  + \op{\xi}_{2+\eps}\bigg)\fro{\C^{(0)} - \C^*}.
\end{align}
Now we consider $\beta_2$. Notice that 
\begin{align*}
	\fro{\gcL(\C^*\times_{j\in[m]}\U_j^{(0)})} &= \fro{\C^*(\times_{j\in[m]}\U_j^{(0)} - \times_{j\in[m]}\U_j^*)\times_{k\in[m]} \U_k^{(0)\top}} \leq m\maxs\max_{j=1,\ldots,m} d(\U_{j}^{(0)}, \U_j^*).
\end{align*}
Therefore
\begin{align}\label{eq:beta2}
	|\beta_2|\leq \fro{\gcL(\C^*\times_i\U_i^{(0)})}\fro{\C^{(0)}-\C^*}\leq  m\maxs\max_{j=1,\ldots,m} d(\U_{j}^{(0)}, \U_j^*)\fro{\C^{(0)} - \C^*}.
\end{align}
Now from \eqref{eq:c0-c}, \eqref{eq:beta1} and \eqref{eq:beta2}, we obtain
\begin{align*}
	\fro{\C^{(0)} - \C^*} \leq C_2 \sqrt{\frac{\sum_{j=1}^{m}d_jr_j + r_1r_2\cdots r_m}{n}} \op{\xi}_{2+\eps}+C_3m\maxs\max_{j=1,\ldots,m} d(\U_{j}^{(0)}, \U_j^*).
\end{align*}
Now we go back to \eqref{eq:m0-m}, and we have
\begin{align*}
	\fro{\M_0-\M^*} \leq C_2 \sqrt{\frac{\sum_{j=1}^{m}d_jr_j + r_1r_2\cdots r_m}{n}} \op{\xi}_{2+\eps}+(C_3+1)m\maxs\max_{j=1,\ldots,m} d(\U_{j}^{(0)},\U_j^*).
\end{align*}
Now from \eqref{eq:dist} and under the sample size condition and SNR condition, we verified for some contant $C_4>0$, $$\fro{\M_0-\M^*}\leq C_4\mins,$$
holds with probability over $1-c_1'\exp(-(\sum_{j=1}^{m}r_jd_j+r_1r_2\cdots r_m))-c_2'n^{-\min\{\eps/2,1\}}-c_3'm\dmax^{-10}$.

\begin{lemma}\label{lemma:init:tensor}
	Let $\M\in\RR^{d_1\times d_2}$ be a rank $r$ matrix and suppose $d_1\leq d_2$. Suppose $\{\X_i\}_{i=1}^n$ is a sequence of i.i.d. sensing matrices with i.i.d. standard normal entries and $\{\xi_i\}_{i=1}^n$ is a sequence of i.i.d. noise and then observe $$Y_i = \inp{\M}{\X_i} + \xi_i,\, i=1,\cdots,n.$$
	Denote $\sigma_1 = \op{\M}$ and $d = d_1\vee d_2$.
	Suppose the following condition is satisfied:
	\begin{itemize}
		\item[(C1)] Sample size satisfies: $ n\geq \sqrt{rd_1d_2}\log d$.
		\item[(C2)] Exists constant $\eps>0$, such that $\EE|\xi_1|^{2+\eps}  < +\infty$. Denote $\op{\xi_{1}}_{2+\eps}:=(\EE|\xi_1|^{2+\eps} )^{1/(2+\eps)}$.
	\end{itemize}
	Truncate $Y_i$ with $\tau = \frac{n^{1/2}}{(d_1d_2)^{1/4}}(\sqrt{r}\sigma_1 + \gamma)$, $\tilde Y_i = \text{sign}(Y_i)(|Y_i|\vee \tau)$ and define $$\tilde{\N}:=\frac{1}{n(n-1)}\sum_{1\leq i\neq i'\leq n}\tilde Y_i\tilde Y_{i'}(\X_i\X_{i'}^{\top} + \X_{i'}\X_i^{\top}).$$
	Then there exist some constants $c_0,C_1,C_2>0$, such that with probability exceeding $1-c_0d^{-10}$,
	$$\op{\tilde\N - \M\M^T}\leq C_1\frac{ (d_1d_2)^{1/4} \sqrt{\log d} }{\sqrt{n}}\sqrt{r}\sigma_{1}^2+C_2 \frac{ (d_1d_2)^{1/2} \sqrt{\log d} }{n} \op{\xi}_{2+\eps}^2.$$
\end{lemma}
\begin{proof}
	For convenience, let $Y,\X$ be i.i.d. with $\{Y_i,\X_{i}\}_{i=1}^{n}$ and $\tilde{Y}:=\text{sign}(Y)\cdot (|Y|\vee \tau)$. 
	Notice that $(2+\eps)$  moment of $Y$ could be bounded with
	\begin{align}\label{eq39}
		\op{Y}_{2+\eps}:=(\EE|Y|^{2+\eps})^{1/(2+\eps)}\leq \sqrt{2+\eps}\fro{\M}+\op{\xi}_{2+\eps}\leq c_0\sqrt{r}\sigma_1  + \op{\xi}_{2+\eps}.
	\end{align} 
	Notice that
	\begin{align}\label{eq40}
		\op{\tilde\N - \M\M^{\top}} \leq \op{\tilde{\N} - \EE\tilde{Y}\X \cdot \EE\tilde{Y}\X^{\top}} + \op{\EE\tilde{Y}\X \cdot \EE\tilde{Y}\X^{\top}- \M\M^{\top}}.
	\end{align}
	By decoupling techniques \citep{de2012decoupling}, it has for any $t>0$,
	$$\PP(\op{\tilde{\N} - \EE\tilde{Y}\X \cdot \EE\tilde{Y}\X^{\top}}\geq t) \leq 15\PP(\op{\hat{\N} - \EE\tilde{Y}\X \cdot \EE\tilde{Y}\X^{\top}}\geq 15t),$$
	where $\hat\N = \frac{1}{2n(n-1)}\sum_{1\leq i\neq j\leq n}\tilde Y_i\tilde Y_j'(\X_i\X_j'^T + \X_j'\X_i^T)$ and $\{(\X_i',\tilde Y_i') \}_{i=1}^n$ is an independent copy of $\{(\X_i,\tilde Y_i) \}_{i=1}^{n}$ such that $$Y_i'=\inp{\M}{\X_{i}'}+\xi_i',\quad i=1,\cdots,n.$$
	Denote
	\begin{align*}
		\bd_1 := \frac{1}{n}\sum_{i=1}^n\tilde Y_i\X_i - \EE \tilde{Y}\X,
		\quad \bd_2 := \frac{1}{n}\sum_{j=1}^n\tilde Y_j'\X_j' -\EE \tilde{Y}\X.
	\end{align*}
	Then the first term of Equation~\ref{eq40} could be expressed in the way of
	\begin{align*}
		\hat\N - \EE \tilde{Y}\X\cdot\EE \tilde{Y}\X^{\top} &= \frac{n}{2(n-1)}\bigg(\bd_1\bd_2^T +\bd_2\bd_1^T + (\bd_1+\bd_2)\cdot \EE\tilde{Y}\X^{\top} + \EE\tilde{Y}\X\cdot(\bd_1+\bd_2)^{\top}\bigg)\\
		&~~~~+ \frac{1}{n-1}\bigg(\frac{1}{2n}\sum_{i=1}^n\tilde Y_i\tilde Y_i'(\X_i\X_i'^T + \X_i'\X_i^T - 2\EE \tilde{Y}\X\cdot \EE \tilde{Y}\X^{\top})\bigg).
	\end{align*}
	
	\hspace{1cm}
	
	\noindent\textit{Step 1. Bound $\op{\bd_1},\op{\bd_2}$.} The proof follows Bernstein's inequality Theorem~\ref{thm:Bernstein Ineq}. Note that for any $\u\in\SS^{d_1-1}$, $\v\in\SS^{d_2-1}$, $|\u^{\top} \EE\tilde{Y}_i \X_{i}\v|\leq \tau \EE |\u^{\top}\X_{i}\v|\leq 2\tau$, which implies $$\op{\EE\tilde{Y}_i \X_{i}}\leq 2\tau.$$ Furthermore, we have for some constant $c_2>0$,
	$$\bpsi{\op{\tilde Y_i\X_i - \EE \tilde{Y}_i\X_i}} \leq  c_2\tau\sqrt{d}. $$
	Then consider $\op{\EE\tilde Y_i^2\X_i\X_i^T}$:
	\begin{align*}
		\op{\EE\tilde Y_i^2\X_i\X_i^T}&\leq\EE\op{\tilde Y_i^2\X_i\X_i^T }\leq (\EE |Y_i|^{2+\eps})^{2/(2+\eps)}(\EE\op{\X_i}^{2\cdot\frac{2+\eps}{\eps}})^{\frac{\eps}{2+\eps}}\\
		&\leq c_2\op{Y}_{2+\eps}^2 d.
	\end{align*}
	Then we have \begin{align*}
		\op{\EE(\tilde Y_i\X_i - \EE\tilde Y_i\X_i)(\tilde Y_i\X_i - \EE\tilde Y_i\X_i)^\top}&\leq \op{\EE\tilde Y_i^2\X_i\X_i^T}\leq c_2\op{Y}_{2+\eps}^2 d.
	\end{align*}
	Similarly,  $	\op{\EE(\tilde Y_i\X_i - \EE\tilde Y_i\X_i)^{\top}(\tilde Y_i\X_i - \EE\tilde Y_i\X_i)}$ could be bounded with $c_2\op{Y}_{2+\eps}^2 d $. Thus by Bernstein's inequality Theorem~\ref{thm:Bernstein Ineq}, the following equation holds with probability exceeding $1-c_4d^{-10}$,
	\begin{align*}
		\op{\frac{1}{n}\sum_{i=1}^n\tilde Y_i\X_i - \EE \tilde{Y}\X}\leq c_3 \op{Y}_{2+\eps} \sqrt{\frac{d\log d}{n}}+c_3\tau\sqrt{d}\sqrt{\log \frac{c'\tau}{\op{Y}_{2+\eps} }}\cdot\frac{\log d}{n},
	\end{align*}
	which finishes bounding $\op{\bd_1}$. It's similar to $\op{\bd_2}$ and with a union bound, the following equation holds 
	\begin{align*}
		\max\{\op{\bd_1},\op{\bd_2}\}\leq  c_3 \op{Y}_{2+\eps} \sqrt{\frac{d\log d}{n}}+c_3\tau\sqrt{d}\sqrt{\log \frac{c'\tau}{\op{Y}_{2+\eps} }}\cdot \frac{\log d}{n}
	\end{align*}
	with probability exceeding $1-2c_4d^{-10}$.
	
	\hspace{1cm}
	
	\noindent\textit{Step 2. Bound $\op{\bd_1\cdot \EE\tilde{Y}\X^T},\op{\bd_2\cdot \EE\tilde{Y}\X^T}$.} WLOG, we only consider $\op{\bd_1\cdot\EE\tilde{Y}\X^T}$ and for convenience denote $\tilde\M = \EE \tilde{Y}\X$. Recall that $\bd_1\cdot \EE\tilde{Y}\X^T=\bd_1\tm^T = \frac{1}{n}\sum_{i=1}^n(\tilde Y_i\X_i - \tm)\tm^T$. The proof also follows Bernstein's inequality Theorem~\ref{thm:Bernstein Ineq}. First, note that
	$$\EE[(\tilde Y_i\X_i - \tm)\tm^{\top}]=0,\quad\bpsi{\op{(\tilde Y_i\X_i - \tm)\tm^\top}}\leq c_1\tau\op{\tilde{\M}}\sqrt{d_1},$$ which uses  $\op{\tm}\leq2\tau $.
	Then consider $ \op{\EE\tilde Y_i^2\X_i\tm^T\tm\X_i^T }$:
	\begin{align*}
		\op{\EE\tilde Y_i^2\X_i\tm^T\tm\X_i^T }&\leq \EE\op{ Y_i^2\X_i\tm^T\tm\X_i^T }\leq (\EE |Y_i|^{2+\eps})^{2/2+\eps}(\EE\op{\tm\X_i^\top}^{2\cdot\frac{2+\eps}{\eps}})^{\frac{\eps}{2+\eps}}.
	\end{align*}
	With standard $\eps$-net techniques (see \cite{vershynin2018high}), we have for any $t>0$, $\op{\tm\X_{i}^\top}\leq c_0 (\sqrt{d_1}+t)\op{\tm}$ holds with probability at least $1-2\exp(-t^2)$ and $\EE\op{\tm\X_i^\top}\leq c_0' \sqrt{d_1}\op{\tm} $. Then combined with $\op{\tilde{\M}}=\op{\EE\tilde{Y}\X}\leq c_1'\op{Y}_{2+\eps}$, we have
	\begin{align*}
		\op{\EE\tilde Y_i^2\X_i\tm^T\tm\X_i^T }\leq c_2\op{Y}_{2+\eps}^2\op{\M}^2 d_1.
	\end{align*}
	Thus we have
	\begin{align*}
		\op{\EE ((\tilde Y_i\X_i - \tm)\tm^T)((\tilde Y_i\X_i - \tm)\tm^T)^T}&\leq \op{\EE\tilde Y_i^2\X_i\tm^T\tm\X_i^T}\leq  c_2\op{Y}_{2+\eps}^2\op{\tilde{\M}}^2 d_1.
	\end{align*}
	Similarly, it has \begin{align*}
		\op{\EE ((\tilde Y_i\X_i - \tm)\tm^T)^{\top}((\tilde Y_i\X_i - \tm)\tm^T)}&\leq \op{\EE\tilde Y_i^2\tm\X_{i}^{\top} \X_{i}\tm^{\top}}\leq  c_2\op{Y}_{2+\eps}^2\op{\tilde{\M}}^2 d_1.
	\end{align*}
	By Bernstein's inequality Theorem~\ref{thm:Bernstein Ineq}, with probability exceeding $1-2d^{-10}$, the following equation holds
	\begin{align*}
		\op{\bd_1\tm^T}\leq c_3\op{Y}_{2+\eps}\op{\tilde{\M}} \sqrt{\frac{d_1 \log d}{n}}+c_3\tau\op{\tilde{\M}} \sqrt d_1\sqrt{\log\frac{\tau }{\op{Y}_{2+\eps} }}\cdot\frac{\log d}{n}
	\end{align*}
	Term $\op{\bd_2\tm^T}$ could be bounded in a similar fashion. Take a union bound and then we get with probability exceeding $1-4d^{-10}$, 
	\begin{equation}
		\begin{split}
			\max&\{\op{\bd_1\tm^T},\op{\bd_2\tm^T}\}\\&\leq c_3\op{Y}_{2+\eps}\op{\tilde{\M}} \sqrt{\frac{d_1 d\log d}{n}}+c_3\tau \op{\tilde{\M}}\sqrt d_1\sqrt{\log\frac{\tau}{\op{Y}_{2+\eps}}}\cdot\frac{\log d}{n}
		\end{split}\label{eq:step2}
	\end{equation}
	
	\hspace{1cm}
	
	\noindent\textit{Step 3. Bound $\op{\bd_1\bd_2^T}$.} We shall bound $\op{\bd_1\bd_2^T}$ under event $$\bcalE_1 = \{ \op{\bd_2}\leq  c_3 \op{Y}_{2+\eps} \sqrt{\frac{d\log d}{n}}+c_3\tau\sqrt{d}\sqrt{\log \frac{c'\tau}{ \op{Y}_{2+\eps}}}\cdot \frac{\log d}{n}\},$$ which is verified to hold with probability over $1-4d^{-10}$ in Step 1. Replace $\tilde{\M}$ with $\bd_2$ of $\op{\bd_1\bd_2}$ in Step 2 and follow Step 2 procedure. Then we get with probability exceeding $1-2d^{-10}$,
	\begin{equation}
		\begin{split}
			&\op{\bd_1\bd_2^T} \leq c_4\bigg(\op{Y}_{2+\eps} \sqrt{\frac{d_1 \log d}{n}}+\tau \sqrt{d_1}\sqrt{\log\frac{\tau }{\op{Y}_{2+\eps}}}\cdot\frac{\log d}{n}\bigg)\\
			&{~~~~~~~~~~~~~~~~~~}\times\bigg(\op{Y}_{2+\eps} \sqrt{\frac{d\log d}{n}}+\tau\sqrt{d}\sqrt{\log \frac{c'\tau}{\op{Y}_{2+\eps} }}\cdot \frac{\log d}{n} \bigg).
		\end{split}\label{eq:step3}
	\end{equation}
	
	\hspace{1cm}

	\noindent\textit{Step 4. Bound $\op{\frac{1}{n}\sum_{i=1}^n\tilde Y_i\tilde Y_i'\X_i\X_i'^T - \tm\tm^T }$.} Notice that
	$$\frac{1}{n}\sum_{i=1}^n\tilde Y_i\tilde Y_i'\X_i\X_i'^T - \tm\tm^T = \frac{1}{n}\sum_{i=1}^n\tilde Y_i\X_i(\tilde Y_i'\X_i'-\tm)^T + \frac{1}{n}\sum_{i=1}^n(\tilde Y_i\X_i-\tm)\tm^T.$$
	Note that the second term $\bd_1\tm^T=\frac{1}{n}\sum_{i=1}^n(Y_i\X_i-\tm)\tm^T $ is analyzed in Step 2. Analysis of the first term also follows Bernstein's inequality Theorem~\ref{thm:Bernstein Ineq}. Notice that
	$$\EE[\tilde Y_i\X_i(\tilde Y_i'\X_i'-\tm)^T]=0,\quad\bpsi{\op{\tilde Y_i\X_i(\tilde Y_i'\X_i'-\tm)^T}}\leq c_2\tau^2 d, $$
	\begin{align*}
		\op{\EE \tilde Y_i^2\X_i(\tilde Y_i'\X_i'-\tm)^T(\tilde Y_i'\X_i'-\tm)\X_i^T}&\leq \op{\EE \tilde Y_i^2 \tilde Y_i^{'2} \X_i\X_i^{'\top}\X_i^{'}\X_i^{\top}}\\
		&\leq (\EE |\tilde Y_i|^{2+\eps} |\tilde Y_i^{'}|^{2+\eps} )^{2/(2+\eps)} (\EE \op{\X_{i}^{'}\X_{i}^{\top}}^{2\cdot\frac{2+\eps}{\eps}})^{\eps/(2+\eps)}\\
		&\leq c_1\op{\tilde{Y}_i}_{2+\eps}^{4}d_1d_2,
	\end{align*}
	\begin{align*}
		\op{\EE \tilde Y_i^2(\tilde Y_i'\X_i'-\tm)\X_i^T\X_i(\tilde Y_i'\X_i'-\tm)^T}&\leq \op{\EE \tilde Y_i^2 \tilde Y_i^{'2} \X_i^{'}\X_i^{\top}\X_i\X_i^{'\top}}\\
		&\leq (\EE |\tilde Y_i|^{2+\eps} |\tilde Y_i^{'}|^{2+\eps} )^{2/(2+\eps)} (\EE \op{\X_{i}\X_{i}^{'\top}}^{2\cdot\frac{2+\eps}{\eps}})^{\eps/(2+\eps)}\\
		&\leq c_1\op{Y_i}_{2+\eps}^{4}d_1d_2.
	\end{align*}
	Then by Bernstein's inequality Theorem~\ref{thm:Bernstein Ineq}, we have with probability exceeding $1-2d^{-10}$, 
	\begin{align}\label{eq:step4}
		\op{\frac{1}{n}\sum_{i=1}^n\tilde Y_i\X_i(\tilde Y_i'\X_i'-\tm)^T}\leq c_3 \op{Y}_{2+\eps}^2\sqrt{d_1 d_2} \sqrt{\frac{\log d}{n}}+c_3 \tau^2d \sqrt{\log \frac{\tau^2 d}{\op{Y }_{2+\eps}^2\sqrt{d_1 d_2}}}\cdot \frac{\log d}{n}.
	\end{align}
	
	\hspace{1cm}
	
	\noindent\textit{Step 5. Bound $\op{\EE\tilde{Y}\X \cdot \EE\tilde{Y}\X^{\top}- \M\M^{\top}}$.} Notice that 
	$\op{\M\M^T-\tm\tm^T}\leq \op{\M-\tm}(\op{\M} + \op{\tm})$. Furthermore, we have for any $\u\in\SS^{d_1-1}$, $\v\in\SS^{d_2-1}$,
	\begin{align*}
		|\u^{\top}(\M-\tm )\v|&\leq\EE|Y|\cdot 1_{|Y|>\tau}\cdot|\u^{\top}\X\v|\\
		&\leq \tau^{-1} \EE|Y|^2|\u^{\top}\X\v|\\
		&\leq \tau^{-1} \op{Y}_{2+\eps}^2\cdot (\EE |\u^{\top}\X\v|^{1+2/\eps})^{\eps/(2+\eps)}\leq c_1 \tau^{-1} \op{Y}_{2+\eps}^2,
	\end{align*}
	where $c_1>0$ is some constant and it implies $ \op{\M-\tm}\leq c_1 \tau^{-1} \op{Y}_{2+\eps}^2$. Then, we have$$\op{\M}=\sigma_{1},\quad \op{\tm}\leq \op{\M}+\op{\M-\tm}\leq \sigma_{1}+ c_1 \tau^{-1} \op{Y}_{2+\eps}^2.$$
	Therefore,
	\begin{align}\label{eq:step5}
		\op{\M\M^\top-\tm\tm^\top}\leq c_1\tau^{-1} \op{Y}_{2+\eps}^2\left(2\sigma_{1} + c_1 \tau^{-1} \op{Y}_{2+\eps}^2\right).
	\end{align}
	
	\hspace{1cm}
	
	\noindent\textit{Step 6. Put everything together.} Combine Equation~\ref{eq40} with Equation~\ref{eq:step2}, \ref{eq:step3}, \ref{eq:step4} and \ref{eq:step5} and then we get with probability exceeding $1-c_0d^{-10}$, 
	\begin{align*}
		&\quad\op{\tilde\N - \M\M^T} \notag\\
		&\leq \frac{n}{n-1}\bigg( \op{\bd_1\bd_2^{\top}}+2\op{\bd_1\cdot \EE\tilde{Y}\X^{\top} }\bigg)+\frac{1}{n-1}\op{\sum_{i=1}^{n} \tilde{Y}_i\tilde{Y}_i' \X_{i}\X_{i}^{'\top} - \EE\tilde{Y}\X\cdot\EE\tilde{Y}\X^{\top} } + \op{\EE\tilde{Y}\X \cdot \EE\tilde{Y}\X^{\top}- \M\M^{\top}}\\
		&\leq 2\bigg(\op{Y}_{2+\eps} \sqrt{\frac{d_1 \log d}{n}}+\tau \sqrt{d_1}\sqrt{\log\frac{\tau}{\op{Y}_{2+\eps}}}\cdot\frac{\log d}{n}\bigg)\bigg(\op{Y}_{2+\eps} \sqrt{\frac{d\log d}{n}}+\tau\sqrt{d}\sqrt{\log \frac{c'\tau}{\op{Y}_{2+\eps} }}\cdot \frac{\log d}{n} \bigg)\\
		&~~~~+3\op{Y}_{2+\eps}\op{\tilde{\M}} \sqrt{\frac{d_1\log d}{n}}+3\tau\op{\tilde{\M}} \sqrt d_1\sqrt{\log\frac{\tau }{\op{\tilde{\M}} }}\cdot\frac{\log d}{n}+c_1\tau^{-1} \op{Y}_{2+\eps}^2\left(2\sigma_{1} + c_1 \tau^{-1} \op{Y}_{2+\eps}^2\right)\\
		&~~~~+\frac{1}{n-1}\bigg(\op{Y}_{2+\eps}^2\sqrt{d_1 d_2} \sqrt{\frac{\log d}{n}}+ \tau^2 d \sqrt{\frac{\tau^2 d}{\op{Y }_{2+\eps}^2\sqrt{d_1 d_2}}}\cdot \frac{\log d}{n} \bigg).
	\end{align*}
	Insert $\op{\tm}\leq \sigma_{1}+ c_1 \tau^{-1} \op{Y}_{2+\eps}^2$, $\op{Y}_{2+\eps}\leq c_0\sqrt{r}\sigma_{1}+\op{\xi}_{2+\eps}$ and take $\tau\asymp \sqrt{n} (c_0\sqrt{r}\sigma_{1}+\op{\xi}_{2+\eps} )\left(d_1d_2\right)^{-1/4}$. Also, with sample size $n\geq \sqrt{d_1d_2}\log d$, we finally have\begin{align*}
		\op{\tilde\N - \M\M^T} \leq C_1\frac{ (d_1d_2)^{1/4} \sqrt{\log d} }{\sqrt{n}}\sqrt{r}\sigma_{1}^2+C_2 \frac{ (d_1d_2)^{1/2} \sqrt{\log d} }{n} \op{\xi}_{2+\eps}^2,
	\end{align*}
	where $C_1,C_2>0$ are some constants.
\end{proof}
\section{Technical Lemmas}
\begin{lemma}[Partial Frobenius Norm, Lemma 28 of \cite{tong2021accelerating}]
	Define partial Frobenius norm of matrix $\M\in\mathbb{R}^{d_1\times d_2}$ to be $$\fror{\M}:=\sqrt{\sum_{i=1}^{r}\sigma_{i}^{2}(\M)}=\fro{{\rm SVD}_r (\M)}.$$ For any $\M, \hat{\M}\in\mathbb{R}^{d_1\times d_2}$ with $\operatorname{rank}(\hat{\M})\leq r$, one has \begin{equation}
		\vert\langle \M,\hat{\M}\rangle\vert\leq\fror{\M}\fro{\hat\M}.
	\end{equation}
	For any $\M\in\mathbb{R}^{d_1\times d_2}$, $\V\in\mathbb{R}^{d_1\times r}$, one has\begin{equation}
		\fro{\M\V}\leq \fror{\M}\op{\V}.
	\end{equation}
	\label{teclem:partial F norm}
\end{lemma}
\begin{lemma}(Frobenius norm of projected sub-gradient)\label{teclem:projected subgradient norm}
	Let $\M\in\RR^{d_1\times d_2}$ be a matrix with rank at most $r$. Suppose $\mathbf{G}\in \partial f(\M)$ is the sub-gradient at $\M$ and $\mathbb{T}$ is the tangent space for $\MM_\r$ at $\M$. Then we have
	$$
	\Vert \mathcal{P}_{\mathbb{T}}(\mathbf{G})\Vert_{\mathrm{F}}\leq\sqrt{2}\fror{\G}.
	$$
\end{lemma}
\begin{proof}
	Notice that 
	\begin{align*}
		\fro{\calP_{\TT}(\G)}^2 = \fro{\U\U^T\G}^2 + \fro{(\I - \U\U^T)\G\V\V^T}^2. 
	\end{align*}
	Then by properties of partial Frobenius norm Lemma~\ref{teclem:partial F norm}, one has
	\begin{align*}
		\fro{\U\U^T\G} = \fro{\U^T\G} &\leq \fror{\G} \op{\U}=\fror{\G},\\
		\fro{(\I - \U\U^T)\G\V\V^T} &\leq \fror{\G},
	\end{align*}
	which implies $ \fro{\calP_{\TT}(\G)}^2\leq 2\fror{\G}^2$.
\end{proof}

\begin{lemma}(Matrix Perturbation)
	Suppose matrix $\M^{*} \in\mathbb{R}^{d_1\times d_2}$ has rank $r$ and has singular value decomposition  $\M^{*}=\mathbf{U}\boldsymbol{\Sigma}\mathbf{V}^{\top}$ where $\boldsymbol{\Sigma}=\operatorname{diag}\{\sigma_{1},\sigma_{2},\cdots,\sigma_{r}\}$ and $\sigma_{1}\geq\sigma_{2}\geq\cdots\geq\sigma_{r}> 0$. Then for any $\hat{\M}\in\mathbb{R}^{d\times d}$ satisfying $\Vert\hat{\M}-\M\Vert_{\mathrm{F}}<\sigma_{r}/4$, with $\hat{\mathbf{U}}_{r}\in\mathbb{R}^{d_1\times r}$ and $\hat{\mathbf{V}}_{r}\in\RR^{d_2\times r}$ the left and right singular vectors of $r$ largest singular values, we have
	\begin{align*}
		\Vert \hat{\mathbf{U}}_{r}\hat{\mathbf{U}}_{r}^{\top} -\mathbf{U}\mathbf{U}^{\top}\Vert&\leq \frac{4}{\sigma_{r}}\Vert \hat{\mathbf{M}}-\mathbf{M}\Vert,\quad \Vert \hat{\mathbf{V}}_{r}\hat{\mathbf{V}}_{r}^{\top} -\mathbf{V}\mathbf{V}^{\top}\Vert\leq \frac{4}{\sigma_{r}}\Vert \hat{\mathbf{M}}-\mathbf{M}\Vert,\\
		\Vert\operatorname{SVD}_{r}(\hat{\M})-\M^{*}\Vert&\leq\Vert\mathbf{\hat{\M}-\M^{*}}\Vert+20\frac{\Vert \hat{\M}-\M^{*}\Vert^2}{\sigma_{r}},\\
		\Vert \operatorname{SVD}_{r}(\hat{\M})-\M^{*}\Vert_{\mathrm{F}}&\leq\Vert\mathbf{\hat{\M}-\M^{*}}\Vert_{\mathrm{F}}+20\frac{\Vert \hat{\M}-\M^{*}\Vert\Vert \hat{\M}-\M^{*}\Vert_{\mathrm{F}}}{\sigma_{r}}.
	\end{align*}
	\label{teclem:perturbation}
\end{lemma}
\begin{proof}
	See Section~\ref{proof:teclem:perturbation}.
\end{proof}

\begin{theorem}(Empirical process in matrix recovery)\label{thm:empirical process}
	Suppose $f(\mathbf{M})$  is defined in Equation~\ref{eq:loss}  and $\rho(\cdot)$ is $\tilde{L}$-Lipschitz continuous. Then there exists constants $C,C_1>0$, such that  
	\begin{equation}
		\big{|} (f(\mathbf{M})-f(\mathbf{M}^{*})) - (\mathbb{E}f(\mathbf{M}) - \mathbb{E}f(\mathbf{M}^{*}))\big{|}\leq (C_1+C\tilde{L})\sqrt{nd_1r}\Vert \mathbf{M}-\mathbf{M}^{*}\Vert_{\mathrm{F}}
	\end{equation} holds for all at most $r$ rank matrix $\M\in\mathbb{R}^{d_1\times d_2}$ with probability over $1-\exp\left(-\frac{C^{2}d_1r}{3}\right)-3\exp\left(-\frac{n}{\log n}\right)$. 
\end{theorem}
\begin{proof}
	See Section~\ref{proof:thm:empirical}.
\end{proof}

\begin{corollary}\label{cor:thm:empirical process}
	Suppose $f(\mathbf{M})$  is defined in Equation~\ref{eq:loss}  and $\rho(\cdot)$ is $\tilde{L}$-Lipschitz continuous. Then there exists constants $C,C_1>0$, such that  
	\begin{equation}
		\big{|} (f(\mathbf{M}+\M_1)-f(\mathbf{M})) - (\mathbb{E}f(\mathbf{M}+\M_1) - \mathbb{E}f(\mathbf{M}))\big{|}\leq (C_1+C\tilde{L})\sqrt{nd_1r}\Vert \mathbf{M}_1\Vert_{\mathrm{F}}
	\end{equation} holds for all at most $r$ rank matrix $\M,\M_1\in\mathbb{R}^{d_1\times d_2}$ with probability over $1-\exp\left(-\frac{C^{2}d_1r}{3}\right)-3\exp\left(-\frac{n}{\log n}\right)$.
\end{corollary}
\begin{proof}
	It is similar to proof of Theorem~\ref{thm:empirical process}.
\end{proof}

\begin{corollary}\label{cor: empirical process}
	Suppose conditions in Theorem~\ref{thm:empirical process} hold and let $\X_1,\dots,\X_n$ have i.i.d. $\mathcal{N}(0,1)$ entries. For any at most rank $r$ matrix $\M\in\RR^{d_1\times d_2}$, There exist constants $c,c_1>0$, such that when $n\geq cd_1r$, the following inequality \begin{align*}
		\sqrt\frac{1}{2\pi}n\fro{\M}\leq\sum_{i=1}^{n}|\inp{\X_{i}}{\M}| \leq 2n\fro{\M}
	\end{align*}
	holds with high probability over $1-\exp(-c_1d_1r)$.
\end{corollary}
\begin{proof}
	Take the absolute loss function in Theorem~\ref{thm:empirical process} , namely, $\rho(\cdot)=\vert\cdot\vert$ and set $\M^*$ to be zero and instert $\EE|\inp{\X_{i}}{\M}|=\sqrt{\frac{2}{\pi}}\fro{\M}$ (see Lemma~\ref{teclem:gaussian-l1}) into it.
\end{proof}

\begin{lemma}(Expectation of absolute loss under Gaussian noise) 	\label{teclem:gaussian-l1} Suppose the noise term follows Gaussian distribution $\xi_{1},\cdots,\xi_{n}\stackrel{i.i.d.}{\sim}\mathcal{N}(0,\sigma^{2})$ and take the absolute value loss function $$f(\M)=\sum_{i=1}^{n}|Y_i-\inp{\M}{\X_i}|,$$
	and then we have $$\mathbb{E}[f(\mathbf{M})]=n\sqrt{\frac{2}{\pi}}\sqrt{\sigma^{2}+\Vert \mathbf{M}-\mathbf{M}^{*}\Vert_{\mathrm{F}}^{2}},\quad \mathbb{E}[f(\mathbf{M}^{*})]=n\sqrt{\frac{2}{\pi}}\sigma.$$
\end{lemma}
\begin{proof}
	Note that $f(\mathbf{M})=\sum_{i=1}^{n}\vert Y_{i}-\langle \mathbf{X}_{i}, \mathbf{M}\rangle\vert=\sum_{i=1}^{n}\vert \xi_{i} - \langle \mathbf{X}_{i}, \mathbf{M}-\mathbf{M}^{*}\rangle\vert$, $f(\M^{*})=\sum_{i=1}^{n}\vert Y_{i}-\langle \mathbf{X}_{i}, \mathbf{M}^{*}\rangle\vert=\sum_{i=1}^{n}\vert \xi_{i} \vert$. 	
	
	First calculate the conditional expectation $\mathbb{E}[[\vert \xi_{i}-\langle \mathbf{A}_{i}, \mathbf{M}-\mathbf{M}^{*}\rangle\vert] \lvert \mathbf{A}_{i}]$. Notice that conditioned on $\mathbf{A}_{i}$, $\xi_{i}-\langle \mathbf{A}_{i}, \mathbf{X}-\mathbf{X}^{*}\rangle$ is normally distributed with mean $-\langle \mathbf{A}_{i}, \mathbf{X}-\mathbf{X}^{*}\rangle$ and variance $\sigma^{2}$ and then $\vert \xi_{i}-\langle \mathbf{A}_{i}, \mathbf{X}-\mathbf{X}^{*}\rangle\vert$ has a folded normal distribution.
	\begin{equation*}
		\begin{split}
			\mathbb{E}&[[\vert \xi_{i}-\langle \mathbf{X}_{i}, \mathbf{M}-\mathbf{M}^{*}\rangle\vert] \lvert \mathbf{X}_{i}]\\ &= \sigma\sqrt{\frac{2}{\pi}}\exp(-\frac{\langle \mathbf{X}_{i}, \mathbf{M}-\mathbf{M}^{*}\rangle^{2}}{2\sigma^{2}})-\langle \mathbf{X}_{i}, \mathbf{M}-\mathbf{M}^{*}\rangle[1-2\phi(\frac{\langle \mathbf{X}_{i}, \mathbf{M}-\mathbf{M}^{*}\rangle}{\sigma})],
		\end{split}
	\end{equation*}
	where $\phi(\cdot)$ is the normal cumulative distribution function. For each $i=1,2,\cdots,n$, one has
	\begin{equation*}
		\begin{split}
			\mathbb{E}[\vert \xi_{i}-\langle \mathbf{X}_{i}, \mathbf{M}-\mathbf{M}^{*}\rangle\vert ]&=\mathbb{E}[\mathbb{E}[[\vert \xi_{i}-\langle \mathbf{X}_{i}, \mathbf{M}-\mathbf{M}^{*}\rangle\vert] \lvert \mathbf{X}_{i}] ]\\
			&=\sigma\sqrt{\frac{2}{\pi}}\mathbb{E}[\exp(-\frac{\langle \mathbf{X}_{i}, \mathbf{M}-\mathbf{M}^{*}\rangle^{2}}{2\sigma^{2}})]-\mathbb{E}[\langle \mathbf{X}_{i}, \mathbf{M}-\mathbf{M}^{*}\rangle[1-2\phi(\frac{\langle \mathbf{X}_{i}, \mathbf{M}-\mathbf{M}^{*}\rangle}{\sigma})]]
		\end{split}
	\end{equation*}
	and notice that $\langle \mathbf{X}_{i}, \mathbf{M}-\mathbf{M}^{*}\rangle\sim\mathcal{N}(0, \Vert \mathbf{M}-\mathbf{M}^{*}\Vert_{\mathrm{F}}^{2})$ and then after integration calculation it is \begin{equation*}
		\mathbb{E}[\vert \xi_{i}-\langle \mathbf{X}_{i}, \mathbf{M}-\mathbf{M}^{*}\rangle\vert = \sqrt{\frac{2}{\pi}}\sqrt{\sigma^{2}+\Vert \mathbf{M}-\mathbf{M}^{*}\Vert_{\mathrm{F}}^{2}}.
	\end{equation*}
	Specifically, when $\M=\M^*$, it leads to $\mathbb{E}[\vert \xi\vert] = \sqrt{\frac{2}{\pi}}\sigma$.
	Thus, we have $$\mathbb{E}[f(\mathbf{M})]=n\mathbb{E}[ \vert \xi_{i}-\langle \mathbf{X}_{i}, \mathbf{M}-\mathbf{M}^{*}\rangle\vert]=n\sqrt{\frac{2}{\pi}}\sqrt{\sigma^{2}+\Vert \mathbf{M}-\mathbf{M}^{*}\Vert_{\mathrm{F}}^{2}},\quad\mathbb{E}[f(\mathbf{\mathbf{M}}^{*})]=n\mathbb{E}[\vert \xi\vert]=n\sqrt{\frac{2}{\pi}}\sigma. $$
\end{proof}

\begin{lemma}(Contraction of Heavy Tailed Random Variables)
	Suppose $\xi,\xi_{1},\xi_{2},\cdots,\xi_{n}$ are \iid random variables and  there exists $\varepsilon>0$ such that its $1+\varepsilon$ moment is finite, that is $\EE\vert \xi\vert^{1+\varepsilon}<+\infty$. Then $\frac{1}{n}\sum_{i=1}^{n}\vert \xi_{i}\vert\leq3\EE\vert \xi\vert$ with probability exceeding $1-{}^*\!cn^{-\min\{\varepsilon,1\}}$. Here $^*\!c>0$ is some constant which depends on $\EE|\xi|, \EE|\xi|^{1+\varepsilon}$.
	\label{teclem:Contraction of Heavy Tailed Random Variables}
\end{lemma}

\begin{proof}
	First consider when $0<\varepsilon<1$. Denote $\varphi_{i}:=\vert\xi_{i}\vert$ and take truncation $\bar{\varphi}_{i}:=\vert\xi_{i}\vert\cdot 1_{\{\vert \xi_{i}\vert< n\}}=\varphi_{i}\cdot 1_{\varphi_{i}>n}$. For conveniece, denote $\op{\xi}_{1+\eps}:=\left(\EE\vert \xi\vert^{1+\varepsilon} \right)^{1/(1+\eps)}$ and introduce independent random variable $\varphi$ which has same distribution as $\varphi_{1},\cdots,\varphi_{n}$ and define $\bar{\varphi}:=\varphi\cdot 1_{\{\varphi< n\}}$. The original problem is equivalent to bound $\frac{1}{n}\sum_{i=1}^{n}\varphi_{i}$ with its expectation.
	
	Note that $\frac{1}{n} \sum_{i=1}^{n} \left(\varphi_{i}-\EE\varphi_{i}\right)$ could be written into sum of three parts, $$\frac{1}{n} \sum_{i=1}^{n} \left(\varphi_{i}-\EE\varphi_{i}\right)=\frac{1}{n} \sum_{i=1}^{n} \left(\bar{\varphi}_{i}-\EE\bar{\varphi}_{i}\right)+ \frac{1}{n} \sum_{i=1}^{n} \left(\varphi_{i}-\bar{\varphi}_{i}\right)+(\EE\varphi-\EE\bar{\varphi}).$$ Analyze the three parts respectively. Consider the first term, for any constant $s>0$,
	\begin{align*}
		\PP\left(\vert \frac{1}{n}\sum_{i=1}^{n} \left(\bar{\varphi}_{i}-\EE\bar{\varphi}_{i}\right)\vert >s\right)&\leq\frac{1}{ns^2}\Var\left(\bar\varphi\right)\leq\frac{1}{ns^2}\EE\bar\varphi^{1+\varepsilon}\bar\varphi^{1-\varepsilon}\leq \frac{\op{\xi}_{1+\eps}^{1+\eps}}{ns^2}\cdot n^{1-\varepsilon}=\op{\xi}_{1+\eps}^{1+\eps}s^{-2}n^{-\varepsilon},
	\end{align*}
	where the intermediate equation uses Markov inequality and the last inequality is because of $\bar{\varphi}<n$. Then consider the second term, the probability of $\frac{1}{n}\sum_{i=1}^{n}\bar{\varphi}_{i}\neq\frac{1}{n}\sum_{i=1}^{n}\varphi_{i} $
	\begin{align*}
		\PP\left(\frac{1}{n}\sum_{i=1}^{n}\bar{\varphi}_{i}\neq\frac{1}{n}\sum_{i=1}^{n}\varphi_{i}\right)&\leq n\PP\left(\varphi>n\right)\leq n\frac{\EE \varphi^{1+\varepsilon}}{n^{1+\varepsilon}}=\op{\xi}_{1+\eps}^{1+\eps}n^{-\varepsilon}.
	\end{align*}
	The last term measures the expectation difference and could be bounded as followed,
	\begin{align*}
		\EE\varphi_{i}-\EE\bar{\varphi}_{i}&=\EE\varphi_{i}\cdot 1_{\varphi_{i}>n}\leq \EE\varphi.
	\end{align*}
	Hence, we have
	\begin{align*}
		\PP\left(\vert \frac{1}{n}\sum_{i=1}^{n} \varphi_{i}-\EE\varphi_{i}\vert >s\right)&\leq \PP\left(\vert \frac{1}{n}\sum_{i=1}^{n} \bar\varphi_{i}-\EE\varphi_{i}\vert >s\right)+\PP\left(\frac{1}{n}\sum_{i=1}^{n}\bar{\varphi}_{i}\neq\frac{1}{n}\sum_{i=1}^{n}\varphi_{i}\right)\\
		&\leq \PP\left(\vert \frac{1}{n}\sum_{i=1}^{n} \bar\varphi_{i}-\EE\bar\varphi_{i}\vert >s-(\EE\varphi-\EE\bar\varphi)\right)+\op{\xi}_{1+\eps}^{1+\eps}n^{-\varepsilon}\\
		&\leq \PP\left(\vert \frac{1}{n}\sum_{i=1}^{n} \bar\varphi_{i}-\EE\bar\varphi_{i}\vert >s-\EE\varphi\right)+\op{\xi}_{1+\eps}^{1+\eps}n^{-\varepsilon},
	\end{align*}
	where the last inequality uses $ 	\EE\varphi_{i}-\EE\bar{\varphi}_{i}\leq \EE\varphi$. Take $s=2\EE\varphi$ and then it becomes
	\begin{align*}
		\PP\left(\vert \frac{1}{n}\sum_{i=1}^{n} \varphi_{i}-\EE\varphi_{i}\vert >2\EE\varphi\right)&\leq \PP\left(\vert \frac{1}{n}\sum_{i=1}^{n} \varphi_{i}-\EE\varphi_{i}\vert >\EE\varphi\right)+\op{\xi}_{1+\eps}^{1+\eps}n^{-\varepsilon}\\
		&\leq \op{\xi}_{1+\eps}^{1+\eps}[\EE\varphi]^{-2}n^{-\varepsilon}+\op{\xi}_{1+\eps}^{1+\eps}n^{-\varepsilon},
	\end{align*}
	which proves $\PP(\frac{1}{n}\sum_{i=1}^{n}|\xi_{i}|\geq 3\EE|\xi|)\leq {}^*\!cn^{-\varepsilon}$. When $\varepsilon\geq1$, we have $\PP(\frac{1}{n}\sum_{i=1}^{n}|\xi_{i}|\geq 3\EE|\xi|)\leq {}^*\!cn^{-1}$ whose proof is simpler and hence skipped.
\end{proof}

\begin{lemma}(Tail inequality for sum of squared heavy-tailed random variables)\label{teclem:bound of second moment heavy tailed}
	Suppose $\xi,\xi_{1},\cdots,\xi_{n}$ are \iid random variables and there exists $0<\varepsilon\leq1$ such that its $1+\varepsilon$ moment is finite, that is $\EE\vert \xi\vert^{1+\varepsilon}<+\infty$. Denote $\op{\xi}_{1+\eps}=\left( \EE\vert \xi\vert^{1+\varepsilon}\right)^{1/(1+\eps)}$. Then for any $s>0$, tail bound for sum of squares could be upper bounded $$\PP\left(\sum_{i=1}^{n}\xi_{i}^2>s\right)\leq 2n\op{\xi}_{1+\eps}^{1+\eps}s^{-\frac{1+\varepsilon}{2}}.$$
\end{lemma}
\begin{proof}
	For any $s>0$, it has $\PP\left(\xi^2\geq s\right)\leq s^{-\frac{1+\varepsilon}{2}} \EE\vert\xi\vert^{1+\varepsilon}$ which uses Markov's inequality. Introduce truncated random variable $\varphi_{i}:=\xi_{i}^21_{\{\xi_{i}^2<s\}}$. Note that $\sum_{i=1}^{n}\xi_{i}^2 $ could be written into sum of two parts $$\sum_{i=1}^{n}\xi_{i}^2=\sum_{i=1}^{n}\varphi_{i}+\sum_{i=1}^{n}(\xi_{i}^2-\varphi_{i}).$$
	Hence, we analyze these two parts respectively. First consider probability of $\sum_{i=1}^{n}\xi_{i}^2\neq\sum_{i=1}^{n}\varphi_{i}$:
	\begin{align*}
		\PP\left(\sum_{i=1}^{n}\xi_{i}^2\neq\sum_{i=1}^{n}\varphi_{i}\right)&\leq n\PP\left(\xi^2>s\right)\leq n\frac{\EE|\xi|^{1+\varepsilon}}{s^{\frac{1+\varepsilon}{2}}}=n\op{\xi}_{1+\eps}^{1+\eps}s^{-\frac{1+\varepsilon}{2}}.
	\end{align*}
	Then bound tail for sum of $\varphi_{i}$:
	\begin{align*}
		\PP\left(\sum_{i=1}^{n}\varphi_{i}>s\right)&\leq n s^{-1}\EE\varphi=ns^{-1}\EE\varphi^{\frac{1+\epsilon}{2}}\varphi^{\frac{1-\epsilon}{2}}\leq n\op{\xi}_{1+\eps}^{1+\eps}s^{-\frac{1+\varepsilon}{2}}.
	\end{align*}
	Finally, combine the above two equations, it has
	\begin{align*}
		\PP\left(\sum_{i=1}^{n}\xi_{i}^2>s\right)&\leq \PP\left(\sum_{i=1}^{n}\varphi_{i}>s\right)+\PP\left(\sum_{i=1}^{n}\xi_{i}^2\neq\sum_{i=1}^{n}\varphi_{i}\right)\leq2n\op{\xi}_{1+\eps}^{1+\eps}s^{-\frac{1+\varepsilon}{2}},
	\end{align*}
	which completes the proof.
\end{proof}
\begin{lemma}[Partial Frobenius norm for tensor]\label{teclem:tensor:partial F norm}
	Suppose $\M\in\RR^{d_1\times \dots \times d_m}$ is an $m$-th order tensor. Denote $\r=(r_1,\cdots,r_m)^{\top}$. The partial Frobenius norm of $\M$ is defined to be $$\frorr{\M}:=\sup_{\U_1\in\OO_{d_1, r_1},\cdots ,\U_m\in\OO_{d_m, r_m}}\fro{\M\times_{1} \U_1^{\top}\times_2\dots\times_m \U_m^{\top}},$$
	Then for any tensor $\hat{\M}\in\RR^{d_1\times\cdots\times d_m}$ with Tucker rank at most $\r$, one has $$|\inp{\M}{\hat{\M}}|\leq \frorr{\M} \fro{\hat\M}.$$
	Specifically, it implies $\frorr{\M}=\sup_{\hat\M\in\MM_\r, \fro{\hat\M}\leq1}|\inp{\M}{\hat\M}|$. Furthermore, for any $\V_j\in\RR^{d_1\times r_1},\ldots, \V_m\in\RR^{d_m\times r_m}$, one has $$\fro{\M\times_1\V_1^{\top}\times_2\cdots\times_m\V_m^{\top}}\leq\frorr{\M}\op{\V_1}\cdots\op{\V_m}.$$
\end{lemma}
\begin{proof}
	For matrix $\hat{\M}\in\MM_{\r}$, there exist core tensor $\hat\C\in \RR^{r_1\times\dots\times r_m}$ and orthogonal matrices $\hat\V_j\in\OO_{d_j,r_j}$ for each $j$ such that $\hat\M=\hat{\C}_1 \times_{1} \hat\V_1 \times_2\dots\times_m \hat\V_m$. Note that $\fro{\hat\M}=\fro{\hat\C}$. Then one has
	\begin{align*}
		|\inp{\M}{\hat\M}|&=|\inp{\M\times_{1}\hat\V_1^{\top}\times_2\dots\times_m \hat\V_m^{\top}}{\hat\C}|\\
		&\leq\fro{\M\times_{1}\hat\V_1^{\top}\times_2\dots\times_m \hat\V_m^{\top} }\fro{\hat\C}\\
		&\leq\frorr{\M}\fro{\hat\M}.
	\end{align*}
	Note that the equality holds when $\frorr{\M}=\fro{\M\times_{1}\hat\V_1^{\top}\times_2\dots\times_m \hat\V_m^{\top} }$ and $\hat\C=\M\times_{1}\hat\V_1^{\top}\times_2\dots\times_m \hat\V_m^{\top}/\frorr{\M}$, which shows $\frorr{\M}=\sup_{\hat\M\in\MM_\r, \fro{\hat\M}\leq1}|\inp{\M}{\hat\M}|$. It is obvious that for any matrices $\V_j\in\RR^{d_j\times r_j}$, one has $$\fro{\M\times_1\V_1^{\top}\times_2\cdots\times_m\V_m^{\top}}\leq\frorr{\M}\op{\V_1}\cdots\op{\V_m}.$$
\end{proof}
\begin{lemma}\label{teclem:tensor:projected subgrad}
	For any tensor $\M\in\MM_{\r}$, suppose $\G\in\partial f(\M)$ is the subgradient at $\M$ and $\TT$ is the tangent space for $\MM_{\r}$ at $\M$. Then we have $$\fro{\calP_{\TT}(\G)}\leq \sqrt{m+1} \frorr{\G}.$$
\end{lemma}
\begin{proof}
	Let $\M=\C\times_{1}\U_1\times_2\cdots\times_m\U_m$ be its Tucker decomposition. Then we have $$
	\calP_{\TT}(\G)=\G\times_{1} \U_1\U_1^{\top}\times_2\cdots\times_m \U_m\U_m^{\top}-\sum_{i=1}^{m}\C\times_{j\in[m]\backslash i}\U_{j}\times_{i}\dot{\U}_i,
	$$
	with $\dot{\U}_i=\left(\G\times_{j\in[m]\backslash i}\U_j^{\top}\times_{i} (\I_{d_i}-\U_i\U_i^{\top})\right)_{(i)}\left(\C_{(i)}\right)^{\dagger}$.
	It implies \begin{align*}
		\fro{\calP_{\TT}(\G) }^2&=\fro{\G\times_{1} \U_1\U_1^{\top}\times_2\cdots\times_m \U_m\U_m^{\top} }^2+\sum_{i=1}^m \fro{\C\times_{j\in[m]\backslash i}\U_{j}\times_{i}\dot{\U}_i }^2\\
		&\leq (m+1)\frorr{\G}^2.
	\end{align*}
\end{proof}
\begin{lemma}[Lemma 3.2 of \cite{koch2010dynamical}] \label{teclem:tensor:perp projected}For any $\M,\M^*\in\MM_{\r}$ with $\fro{\M-\M^*}\leq \frac{1}{16m(m+3)}\underline\lambda$, denote $\TT$ as the tangent space for $\MM_\r$ at $\M$ and then we have
	$$\fro{\calP_{\TT}^{\perp}(\M-\M^*) }\leq \frac{8m(m+3)}{\underline{\lambda}}\fro{\M-\M^*}^2.$$
\end{lemma}
\begin{theorem}(Empirical process in tensor recovery)\label{thm:tensor:empirical process}
	Suppose $f(\cdot)$ is defined in Equation~\ref{eq:tensor:loss} and $\rho(\cdot)$ is $\tilde{L}$-Lipschitz continuous. Then there exist constants $C,C_1>0$, such that  
	\begin{equation}
		\big{|} (f(\mathbf{M})-f(\mathbf{M}^{*})) - (\mathbb{E}f(\mathbf{M}) - \mathbb{E}f(\mathbf{M}^{*}))\big{|}\leq (C_1+C\tilde{L})\sqrt{n(2\sum_{j=1}^{m}r_jd_j+2^m\cdot r_1\cdots r_d)}\Vert \mathbf{M}-\mathbf{M}^{*}\Vert_{\mathrm{F}}
	\end{equation} holds for all at most Tucker rank $\r$ tensor $\M\in\mathbb{R}^{d_1\times\cdots\times d_m}$ with probability over $1-3\exp\left(-\frac{n}{\log n}\right)-\exp\left(-C(2^m\cdot r_1r_2\cdots r_m+2\sum_{j=1}^{m}r_jd_j)\right) $. 
\end{theorem}
\begin{proof}
	See Section~\ref{proof:thm:tensor:empirical process}.
\end{proof}

\begin{corollary}\label{cor:thm:tensor:empirical process}
	Suppose $f(\cdot)$ is defined in Equation~\ref{eq:tensor:loss} and $\rho(\cdot)$ is $\tilde{L}$-Lipschitz continuous. Then there exist constants $C,C_1>0$, such that  
	\begin{equation}
		\big{|} (f(\mathbf{M}+\M_1)-f(\mathbf{M})) - (\mathbb{E}f(\mathbf{M}+\M_1) - \mathbb{E}f(\mathbf{M}))\big{|}\leq (C_1+C\tilde{L})\sqrt{n(2\sum_{j=1}^{m}r_jd_j+2^m\cdot r_1\cdots r_d)}\Vert \mathbf{M}_1\Vert_{\mathrm{F}}
	\end{equation} holds for all at most Tucker rank $\r$ tensor $\M\in\mathbb{R}^{d_1\times\cdots\times d_m}$ and all at most Tucker rank $2\r$ tensor $\M_1\in\mathbb{R}^{d_1\times\cdots\times d_m}$, with probability over $1-3\exp\left(-\frac{n}{\log n}\right)-\exp\left(-C(2^m\cdot r_1r_2\cdots r_m+2\sum_{j=1}^{m}r_jd_j)\right) $. 
\end{corollary}
\begin{proof}
	It is similar to proof of Theorem~\ref{thm:tensor:empirical process}.
\end{proof}
\begin{corollary}\label{cor:tensor:empirical process}
	Suppose conditions of Theorem~\ref{thm:tensor:empirical process} hold and let $\X_1,\dots,\X_n$ have i.i.d. $N(0,1)$ entries. For any at most rank $\r$ tensors $\M\in\MM_{2\r}$, there exist constants $c,c_1>0$, such that when $n\geq c\sqrt{2^m\cdot r_1r_2\cdots r_m+2\sum_{j=1}^{m}r_jd_j}$, the following inequality \begin{align*}
		\sqrt\frac{1}{2\pi}n\fro{\M}\leq\sum_{i=1}^{n}|\inp{\X_{i}}{\M}| \leq 2n\fro{\M}
	\end{align*}
	holds with high probability over $1-\exp(-c_1(2^m\cdot r_1r_2\cdots r_m+2\sum_{j=1}^{m}r_jd_j))$.
\end{corollary}
\begin{proof}
	Take the absolute loss function in Theorem~\ref{thm:tensor:empirical process}, namely, $\rho(\cdot)=\vert\cdot\vert$  and instert $\EE|\inp{\X_{i}}{\M}|=\sqrt{\frac{2}{\pi}}\fro{\M}$ (see Lemma~\ref{teclem:gaussian-l1}) into it.
\end{proof}

\begin{lemma}[Partial Frobenius norm of random Gaussian Tensor]\label{teclem:tensor:gaussian F norm}
	Suppose $\X\in\RR^{d_1\times\dots\times d_m}$ has independent $N(0,1)$ entries. Then for rank vector $\r=(r_1,\dots,r_m)^{\top}$, it has \begin{align*}
		\PP\left(\frorr{\X}\geq 2c\left(t+\sqrt{\sum_{j=1}^{d}r_jd_j+r_1r_2\cdots r_d}\right)\right)\leq2\exp\left(-t^2\right)
	\end{align*}
\end{lemma}
\begin{proof}
	The proof uses $\eps$-net arguments. Recall that $\MM_{\r}=\{\M\in\RR^{d_1\times\dots\times d_m}: \rank(\M_{(j)})\leq r_j,\, j =1,\cdots,m\}$ is the set of tensors with Tucker rank at most $\r$. From definition of partial Frobenius norm, it has $$\frorr{\X}=\sup_{\M\in\MM_\r, \fro{\M}\leq1}|\inp{\X}{\M}|$$ and suppose it achieves the supremum at $\M_0\in\MM_{\r}$, namely, $\fro{\M_0}=1$ and $\frorr{\X}=|\inp{\X}{\M_0}|$. Then there exist a core tensor $\C_0\in\RR^{r_1\times \cdots\times r_m}$ and orthogonal matrices $ \U_1^0\in\OO_{d_1, r_1},\cdots,\U_m^0\in\OO_{d_m, r_m}$ such that $$\M_0=\C_0\times_{1} \U_1^0\times_2\cdots\times_m \U_m^0.$$ Notice that $\fro{\C_0}=1$. Define $\FF_\r=\{\C\in\RR^{r_1\times\dots\times r_m}: \fro{\C}=1\}$ to be the set of tensors with unit Frobenius norm. Then it has one $\eps/(m+1)$-net $\calN_{\eps/(m+1)}^{\FF_\r}$ with  cardinality $|\calN_{\eps/(m+1)}^{\FF_\r}|\leq (3(m+1)/\eps)^{r_1r_2\cdots r_m}$ respect to the Frobenius norm. Suppose $\calN_1,\cdots,\calN_m$ are $\eps/(m+1)$-nets of orthogonal matrix sets $\OO_{d_1,r_1},\cdots,\OO_{d_m, r_m}$ respectively with cardinalities $$|\calN_1|\leq(3(m+1)/\eps)^{d_1r_1},\cdots,|\calN_m|\leq(3(m+1)/\eps)^{d_mr_m}.$$ See \cite{rauhut2017low,vershynin2018high} for more about $\eps$-nets. 
	Hence, for tensor $\M_0=\C_0\times_{1} \U_1^0\times_2\cdots\times_m \U_m^0$, there is close approximation in the nets. Exist $\C\in\calN_{\eps/(m+1)}^{\FF_\r}, \,\U_1\in\calN_1,\cdots, \U_m\in\calN_m$ such that 
	$$\fro{\C_0-\C}\leq \eps/(m+1),\,\fro{\U_1^0-\U_1}\leq \eps/(m+1),\cdots, \fro{\U_m^0-\U_m}\leq \eps/(m+1).$$ 
	Denote the product as $\M=\C\times_{1} \U_1\times_2\cdots\times_m \U_m$. Furthermore, net $\calN:=\{\M=\C\times_{1} \U_1\times_2\cdots\times_m \U_m: \C\in\calN_{\eps/(m+1)}^{\FF_\r},\,\U_1\in\calN_1,\dots,\U_m\in\calN_m\}$ forms an $\eps$-net of $\MM_{\r}$ with cardinality $|\calN|\leq(3(m+1)/\eps)^{r_1r_2\cdots r_m+\sum_{j=1}^{m}r_jd_j}$.
	
	Then come back to partial Frobenius norm of $\X$ and we have  \begin{align*}
		\frorr{\X}&=|\inp{\X}{\M_0}|\\
		&\leq |\inp{\X}{\M}|+|\inp{\X}{\M_0-\M}|\\
		&\leq \sup_{\M\in\calN}|\inp{\X}{\M}|+\eps\frorr{\X},
	\end{align*}
	which implies $$\frorr{\X}\leq\frac{1}{1-\eps} \sup_{\M\in\calN}|\inp{\X}{\M}|.$$ 
	On the other hand, for any fixed $\M\in\calN$,  we have $\|\inp{\X}{\M}\|_{\Psi_{2}}^2=1$.
	Then the inner product between $\X$ and fixed tensor $\M$ has the following tail bound $$ \PP\left(|\inp{\X}{\M}|\geq u\right)\leq 2\exp(-u^2).$$
	Thus the union yields $$\PP\left(\sup_{\M\in\calN}|\inp{\X}{\M}|\geq u\right)\leq 2|\calN|\exp(-u^2)\leq (3(m+1)/\eps)^{r_1r_2\cdots r_m+\sum_{j=1}^{m}r_jd_j}\exp(-u^2).$$
	Take $\eps=1/2$ and $u=c(\sqrt{ r_1r_2\cdots r_m+\sum_{j=1}^{m}r_jd_j}+t)$ and it leads to $$\PP\left(\sup_{\M\in\calN}|\inp{\X}{\M}|\geq u\right)\leq2\exp(-t^2),$$which proves $$\PP\left(\frorr{\X}\geq 2u\right)\leq2\exp(-t^2). $$
\end{proof}

\section{Proof of Matrix Perturbation Lemma~\ref{teclem:perturbation}}
\label{proof:teclem:perturbation}
To prove Lemma~\ref{teclem:perturbation}, first consider simpler case when the matrix and the perturbation are symmetric.
\begin{lemma}(Symmetric Matrix Perturbation)
	Suppose symmetric matrix $\mathbf{M}^{*}\in\RR^{d\times d}$ has rank $r$. Let $\mathbf{M}^{*}=\mathbf{U}\boldsymbol{\Lambda}\mathbf{U}^{\top}$ be its singular value decomposition with $\boldsymbol{\Lambda}=\operatorname{diag}\{\lambda_{1},\lambda_{2},\cdots,\lambda_{r}\}$, $\vert \lambda_{1}\vert\geq\vert \lambda_{2}\vert\geq\cdots\geq\vert\lambda_{r}\vert>0$. Then for any matrix $\hat{\M}\in\RR^{d\times d}$ satisfying $\Vert\hat{\M}-\M\Vert_{\mathrm{F}}<\sigma_{r}/4$ with $\hat{\mathbf{U}}_{r}\in\mathbb{R}^{d\times r}$ eigenvectors of $r$ largest absolute eigenvalues of $\hat{\mathbf{M}}$, then we have 
	\begin{equation*}
		\Vert \hat{\mathbf{U}}_{r}\hat{\mathbf{U}}_{r}^{\top}-\mathbf{U}\mathbf{U}^{\top}\Vert\leq\frac{4\Vert\hat{\mathbf{M}}-\mathbf{M}\Vert}{\vert\lambda_{r}\vert},
	\end{equation*}
	\begin{equation*}
		\Vert \operatorname{SVD}_{r}(\hat{\mathbf{M}})-\mathbf{M}^{*}\Vert\leq\Vert\hat{\mathbf{M}}-\mathbf{M}\Vert+20\frac{\Vert \hat{\mathbf{M}}-\mathbf{M}\Vert^2}{\lambda_{r}},
	\end{equation*}
	\begin{equation*}
		\Vert \operatorname{SVD}_{r}(\hat{\mathbf{M}})-\mathbf{M}^{*}\Vert_{\mathrm{F}}\leq\Vert\hat{\mathbf{M}}-\mathbf{M}\Vert_{\mathrm{F}}+20\frac{\Vert \hat{\mathbf{M}}-\mathbf{M}\Vert\Vert\hat{\mathbf{M}}- \mathbf{M}\Vert_{\mathrm{F}}}{\lambda_{r}}.
	\end{equation*}
	\label{teclem:symmetric perturbation}
\end{lemma}
\begin{proof}
	Denote $\mathbf{Z}=\hat{\mathbf{M}}-\mathbf{M}$. Define $\mathbf{U}_{\perp}\in\mathbb{R}^{d\times (d-r)}$ such that $[\mathbf{U}, \mathbf{U}_{\perp}]\in\mathbb{R}^{d\times d}$ is orthonormal and then define the projector $$\mathfrak{P}^{\perp}:=\mathbf{U}_{\perp}\mathbf{U}_{\perp}^{\top},\quad \mathfrak{P}^{-1}:=\mathbf{U}\boldsymbol{\Lambda}^{-1}\mathbf{U}^{\top}.$$
	Write $\mathfrak{P}^{-k}=\mathbf{U}\boldsymbol{\Lambda}^{-k}\mathbf{U}^{\top}$, for all $k\geq1$ and for convenience when $k=0$, we write $\mathfrak{P}^{0}=\mathfrak{P}^{-1}$. Define the $k$-th order perturbation \begin{equation*}
		\mathcal{S}_{\mathbf{M}^{*},k}(\mathbf{Z}):=\sum_{\mathbf{s}:s_{1}+\cdots+s_{k+1}=k} (-1)^{1+\tau(\mathbf{s})} \mathfrak{P}^{-s_{1}}\mathbf{Z}\mathfrak{P}^{-s_{2}}\cdots\mathfrak{P}^{-s_{k}}\mathbf{Z}\mathfrak{P}^{-s_{k+1}},
	\end{equation*} where $s_{1},\cdots,s_{k}$ are non-negative integers and $\tau(\mathbf{s})=\sum_{i=1}^{k}\mathbb{I}(s_{i}>0)$ is the number of positive indices in $\mathbf{s}$. Work \cite{xia2021normal} proves \begin{equation*}
		\hat{\mathbf{U}}_{r}\hat{\mathbf{U}}_{r}^{\top}-\mathbf{U}\mathbf{U}^{\top}=\sum_{k\geq1}\mathcal{S}_{\mathbf{M}^{*},k}(\mathbf{Z}).
	\end{equation*}
	Then norm of $\mathcal{S}_{\mathbf{M}^{*},k}(\mathbf{Z})$ coulded be bounded in the way of
	\begin{equation*}
		\Vert \mathcal{S}_{\mathbf{M}^{*},k}(\mathbf{Z})\Vert\leq \binom{2k}{k} \frac{\Vert\mathbf{Z}\Vert^{k}}{\vert\lambda_{r}\vert^{k}}\leq\left(\frac{4\Vert\mathbf{Z}\Vert}{\vert\lambda_{r}\vert}\right)^{k},\quad 	\Vert \mathcal{S}_{\mathbf{M}^{*},k}(\mathbf{Z})\Vert_{\mathrm{F}}\leq \binom{2k}{k} \frac{\Vert\mathbf{Z}\Vert_{\mathrm{F}}^{k}}{\vert\lambda_{r}\vert^{k}}\leq\left(\frac{4\Vert\mathbf{Z}\Vert_{\mathrm{F}}}{\vert\lambda_{r}\vert}\right)^{k}.
	\end{equation*}
	Hence, the above three equations lead to
	\begin{equation*}
		\Vert \hat{\mathbf{U}}_{r}\hat{\mathbf{U}}_{r}^{\top}-\mathbf{U}\mathbf{U}^{\top}\Vert\leq \sum_{k\geq1}\Vert\mathcal{S}_{\mathbf{M}^{*},k}(\mathbf{Z})\Vert\leq \frac{4\Vert\mathbf{Z}\Vert}{\vert\lambda_{r}\vert},\quad \Vert \hat{\mathbf{U}}_{r}\hat{\mathbf{U}}_{r}^{\top}-\mathbf{U}\mathbf{U}^{\top}\Vert_{\mathrm{F}}\leq \sum_{k\geq1}\Vert\mathcal{S}_{\mathbf{M}^{*},k}(\mathbf{Z})\Vert_{\mathrm{F}}\leq \frac{4\Vert\mathbf{Z}\Vert_{\mathrm{F}}}{\vert\lambda_{r}\vert}.
	\end{equation*}
	Then consider $\mathcal{S}_{\mathbf{M}^{*},k}(\mathbf{Z})\mathbf{M}^{*}$:\begin{equation*}
		\begin{split}
			\mathcal{S}_{\mathbf{M}^{*},k}(\mathbf{Z})\mathbf{M}^{*}
			&=\sum_{\mathbf{s}:s_{1}+\cdots+s_{k+1}=k} (-1)^{1+\tau(\mathbf{s})} \mathfrak{P}^{-s_{1}}\mathbf{Z}\mathfrak{P}^{-s_{2}}\cdots\mathfrak{P}^{-s_{k}}\mathbf{Z}\mathfrak{P}^{-s_{k+1}}\mathbf{M}^{*}\\
			&=\sum_{\mathbf{s}:s_{1}+\cdots+s_{k+1}=k} (-1)^{1+\tau(\mathbf{s})} \mathfrak{P}^{-s_{1}}\mathbf{Z}\mathfrak{P}^{-s_{2}}\cdots\mathfrak{P}^{-s_{k}}\mathbf{Z}\mathfrak{P}^{-s_{k+1}}\mathbf{U}\boldsymbol{\Lambda}\mathbf{U}^{\top}.
		\end{split}
	\end{equation*}
	Note that when $s_{k+1}=0$, $\mathfrak{P}^{-s_{k+1}}=\mathbf{U}_{\perp}\mathbf{U}_{\perp}^{\top}$ and hence $\mathfrak{P}^{-s_{k+1}}\mathbf{M}^{*}=0$ and thus we have $$\mathcal{S}_{\mathbf{M}^{*},k}(\mathbf{Z})\mathbf{M}^{*}=\sum_{\mathbf{s}:s_{1}+\cdots+s_{k+1}=k,\, s_{k+1}>0} (-1)^{1+\tau(\mathbf{s})} \mathfrak{P}^{-s_{1}}\mathbf{Z}\mathfrak{P}^{-s_{2}}\cdots\mathfrak{P}^{-s_{k}}\mathbf{Z}\mathfrak{P}^{1-s_{k+1}}.$$  Specifically, when $k=1$, it is $\mathcal{S}_{\mathbf{M}^{*},1}(\mathbf{Z})\mathbf{M}^{*}= \mathbf{U}_{\perp}\mathbf{U}_{\perp}^{\top}\mathbf{Z}\mathbf{U}\mathbf{U}^{\top}$ and its operator norm could be bounded with
	$$\Vert \mathcal{S}_{\mathbf{M}^{*},k}(\mathbf{Z})\mathbf{M}^{*}\Vert\leq \binom{2k-1}{k}\frac{\Vert \mathbf{Z}\Vert^{k+1}}{\vert\lambda_{r}\vert^{k-1}}\leq \Vert\mathbf{Z}\Vert\left(\frac{4\Vert\mathbf{Z}\Vert}{\vert\lambda_{r}\vert}\right)^{k-1},$$
	which implies
	$$\Vert \sum_{k\geq2} \mathcal{S}_{\mathbf{M}^{*},k}(\mathbf{Z})\mathbf{M}^{*}\Vert\leq\frac{4\Vert\mathbf{Z}\Vert^{2}}{\vert\lambda_{r}\vert}.$$
	Similarly, Frobenius norm of  $\mathcal{S}_{\mathbf{M}^{*},k}(\mathbf{Z})\mathbf{M}^{*} $ has $$\Vert \mathcal{S}_{\mathbf{M}^{*},k}(\mathbf{Z})\mathbf{M}^{*}\Vert_{\mathrm{F}}\leq \binom{2k-1}{k}\frac{\Vert \mathbf{Z}\Vert^{k}}{\vert\lambda_{r}\vert^{k-1}}\Vert \mathbf{Z}\Vert_{\mathrm{F}}\leq \Vert\mathbf{Z}\Vert_{\mathrm{F}}\left(\frac{4\Vert\mathbf{Z}\Vert}{\vert\lambda_{r}\vert}\right)^{k-1},$$
	$$\Vert \sum_{k\geq2} \mathcal{S}_{\mathbf{M}^{*},k}(\mathbf{Z})\mathbf{M}^{*}\Vert_{\mathrm{F}}\leq\frac{4\Vert \mathbf{Z}\Vert_{\mathrm{F}}\Vert\mathbf{Z}\Vert}{\vert\lambda_{r}\vert}.$$
	Then expand $\operatorname{SVD}_{r}(\hat{\mathbf{M}})-\mathbf{M}^{*} $:
	\begin{equation*}
		\begin{split}
			&{~~~~}\operatorname{SVD}_{r}(\hat{\mathbf{M}})-\mathbf{M}^{*}\\&=\hat{\mathbf{U}}_{r}\hat{\mathbf{U}}_{r}^{\top}(\mathbf{M}^{*}+\mathbf{Z})\hat{\mathbf{U}}_{r}\hat{\mathbf{U}}_{r}^{\top}-\mathbf{M}^{*}\\
			&=\hat{\mathbf{U}}_{r}\hat{\mathbf{U}}_{r}^{\top}(\mathbf{M}^{*}+\mathbf{Z})\hat{\mathbf{U}}_{r}\hat{\mathbf{U}}_{r}^{\top}-\mathbf{U}\mathbf{U}^{\top}\mathbf{M}^{*}\mathbf{U}\mathbf{U}^{\top}\\
			&=(\hat{\mathbf{U}}_{r}\hat{\mathbf{U}}_{r}^{\top}-\mathbf{U}\mathbf{U}^{\top})\mathbf{M}^{*}\mathbf{U}\mathbf{U}^{\top} + \hat{\mathbf{U}}_{r}\hat{\mathbf{U}}_{r}^{\top}\mathbf{M}^{*}( \hat{\mathbf{U}}_{r}\hat{\mathbf{U}}_{r}^{\top}-\mathbf{U}\mathbf{U}^{\top})+ \hat{\mathbf{U}}_{r}\hat{\mathbf{U}}_{r}^{\top}\mathbf{Z}\hat{\mathbf{U}}_{r}\hat{\mathbf{U}}_{r}^{\top}\\
			&=\sum_{k\geq2} \mathcal{S}_{\mathbf{M}^{*},k}(\mathbf{Z})\mathbf{M}^{*}\mathbf{U}\mathbf{U}^{\top} + \hat{\mathbf{U}}_{r}\hat{\mathbf{U}}_{r}^{\top}\mathbf{M}^{*}\sum_{k\geq2} \mathcal{S}_{\mathbf{M}^{*},k}(\mathbf{Z})+\mathbf{U}_{\perp}\mathbf{U}_{\perp}^{\top}\mathbf{Z}\mathbf{U}\mathbf{U}^{\top} \\&{~~~} +\hat{\mathbf{U}}_{r}\hat{\mathbf{U}}_{r}^{\top} \mathbf{U}\mathbf{U}^{\top}\mathbf{Z}\mathbf{U}_{\perp}\mathbf{U}_{\perp}^{\top}+\hat{\mathbf{U}}_{r}\hat{\mathbf{U}}_{r}^{\top}\mathbf{Z}\hat{\mathbf{U}}_{r}\hat{\mathbf{U}}_{r}^{\top}\\
			&=\sum_{k\geq2} \mathcal{S}_{\mathbf{M}^{*},k}(\mathbf{Z})\mathbf{M}^{*}\mathbf{U}\mathbf{U}^{\top} + \hat{\mathbf{U}}_{r}\hat{\mathbf{U}}_{r}^{\top}\mathbf{M}^{*}\sum_{k\geq2} \mathcal{S}_{\mathbf{M}^{*},k}(\mathbf{Z})+\mathbf{Z}- \mathbf{U}_{\perp}\mathbf{U}_{\perp}^{\top}\mathbf{Z}\mathbf{U}_{\perp}\mathbf{U}_{\perp}^{\top}\\&{~~~}+ (\hat{\mathbf{U}}_{r}\hat{\mathbf{U}}_{r}^{\top}-\mathbf{U}\mathbf{U}^{\top}) \mathbf{U}\mathbf{U}^{\top}\mathbf{Z}\mathbf{U}_{\perp}\mathbf{U}_{\perp}^{\top}+(\hat{\mathbf{U}}_{r}\hat{\mathbf{U}}_{r}^{\top}\mathbf{Z}\hat{\mathbf{U}}_{r}\hat{\mathbf{U}}_{r}^{\top}-\mathbf{U}\mathbf{U}^{\top}\mathbf{Z}\mathbf{U}\mathbf{U}^{\top})\\
		\end{split}
	\end{equation*}
	Take Frobenius norm on each side of the above equation and it leads to
	\begin{equation*}
		\begin{split}
			&{~~~~}\Vert \operatorname{SVD}_{r}(\hat{\mathbf{M}})-\mathbf{M}^{*}\Vert_{\mathrm{F}}\\
			&\leq \Vert \sum_{k\geq2} \mathcal{S}_{\mathbf{M}^{*},k}(\mathbf{Z})\mathbf{M}^{*}\mathbf{U}\mathbf{U}^{\top}\Vert_{\mathrm{F}} +\Vert \hat{\mathbf{U}}_{r}\hat{\mathbf{U}}_{r}^{\top}\mathbf{M}^{*}\sum_{k\geq2} \mathcal{S}_{\mathbf{M}^{*},k}(\mathbf{Z})\Vert_{\mathrm{F}}+\Vert\mathbf{Z}- \mathbf{U}_{\perp}\mathbf{U}_{\perp}^{\top}\mathbf{Z}\mathbf{U}_{\perp}\mathbf{U}_{\perp}^{\top}\Vert_{\mathrm{F}}\\&{~~~}+ \Vert(\hat{\mathbf{U}}_{r}\hat{\mathbf{U}}_{r}^{\top}-\mathbf{U}\mathbf{U}^{\top}) \mathbf{U}\mathbf{U}^{\top}\mathbf{Z}\mathbf{U}_{\perp}\mathbf{U}_{\perp}^{\top}\Vert_{\mathrm{F}}+\Vert\hat{\mathbf{U}}_{r}\hat{\mathbf{U}}_{r}^{\top}\mathbf{Z}\hat{\mathbf{U}}_{r}\hat{\mathbf{U}}_{r}^{\top}-\mathbf{U}\mathbf{U}^{\top}\mathbf{Z}\mathbf{U}\mathbf{U}^{\top}\Vert_{\mathrm{F}}\\
			&\leq \frac{4\Vert \mathbf{Z}\Vert_{\mathrm{F}}\Vert\mathbf{Z}\Vert}{\vert\lambda_{r}\vert}+\frac{4\Vert \mathbf{Z}\Vert_{\mathrm{F}}\Vert\mathbf{Z}\Vert}{\vert\lambda_{r}\vert}+\Vert\mathbf{Z}\Vert_{\mathrm{F}}+3\Vert \hat{\mathbf{U}}_{r}\hat{\mathbf{U}}_{r}^{\top}-\mathbf{U}\mathbf{U}^{\top}\Vert\Vert\mathbf{Z}\Vert_{\mathrm{F}}\\
			&\leq \Vert\mathbf{Z}\Vert_{\mathrm{F}} + 20\frac{\Vert \mathbf{Z}\Vert_{\mathrm{F}}\Vert\mathbf{Z}\Vert}{\vert\lambda_{r}\vert}.
		\end{split}
	\end{equation*}
	Similarly, one has the bound of the operator norm \begin{equation*}
		\Vert \operatorname{SVD}_{r}(\hat{\mathbf{M}})-\mathbf{M}^{*}\Vert\leq\Vert\hat{\mathbf{M}}-\mathbf{M}\Vert+20\frac{\Vert \hat{\mathbf{M}}-\mathbf{M}\Vert^2}{\lambda_{r}}.
	\end{equation*}
\end{proof}

\begin{proof}[\textbf{Proof of Lemma~\ref{teclem:perturbation}}]
	Here we construct symmetrization of $\mathbf{M}^{*}$ and $\hat{\mathbf{M}}$ and then we could apply Lemma~\ref{teclem:symmetric perturbation}. First construct symmetric matrices $$\mathbf{Y}^{*}:=\left(\begin{matrix}
		0&\mathbf{M}^{*}\\
		\mathbf{M}^{*\top}&0
	\end{matrix}\right),\,\hat{\mathbf{Y}}:=\left(\begin{matrix}
		0&\hat{\mathbf{M}}\\
		\hat{\mathbf{M}}^{\top}&0
	\end{matrix}\right).$$
	Denote the perturbation matrix as $\mathbf{Z}:=\hat{\mathbf{M}}-\mathbf{M}^{*}$ and similarly define
	$$\hat{\mathbf{Z}}:=\left(\begin{matrix}
		0&\mathbf{Z}\\
		\mathbf{Z}^{\top}&0
	\end{matrix}\right)=\left(\begin{matrix}
		0&\hat{\mathbf{M}}-\mathbf{M}^{*}\\
		\hat{\mathbf{M}}^{\top}-\mathbf{M}^{*\top}&0
	\end{matrix}\right).$$
	Note that $\operatorname{rank}(\mathbf{Y}^{*})\leq2r$.
	By Lemma~\ref{teclem:symmetric perturbation}, one has\begin{equation}
		\Vert \operatorname{SVD}_{2r}(\hat{\mathbf{Y}})-\mathbf{Y}^{*}\Vert_{\mathrm{F}}\leq\Vert\hat{\mathbf{Z}}\Vert_{\mathrm{F}}+20\frac{\Vert \hat{\mathbf{Z}}\Vert\Vert \hat{\mathbf{Z}}\Vert_{\mathrm{F}}}{\vert\lambda_{2r}(\mathbf{Y}^{*})\vert}
		\label{eq26}
	\end{equation}
	Suppose $\M^*$ has singular value decomposition $\mathbf{M}^{*}=\mathbf{U}\boldsymbol{\Sigma}\mathbf{V}^{\top}$, where $\boldsymbol{\Sigma}=\operatorname{diag}(\sigma_{1},\cdots,\sigma_{r})\in\mathbb{R}^{r\times r}$ with $\sigma_{1}\geq\cdots\geq\sigma_{r}>0$. Notice that $$\frac{1}{\sqrt{2}}\left(\begin{matrix}
		\mathbf{U}&\mathbf{U}\\
		\mathbf{V}&-\mathbf{V}
	\end{matrix}\right)\in\mathbb{R}^{2d\times 2r}$$ is an orthogonal matrix and its columns contain $2r$ linearly independent eigenvectors of $\mathbf{Y}^{*}$ \begin{equation*}
		\frac{1}{\sqrt{2}}\left(\begin{matrix}
			\mathbf{U}&\mathbf{U}\\
			\mathbf{V}&-\mathbf{V}
		\end{matrix}\right)^{\top}\mathbf{Y}^{*}\frac{1}{\sqrt{2}}\left(\begin{matrix}
			\mathbf{U}&\mathbf{U}\\
			\mathbf{V}&-\mathbf{V}
		\end{matrix}\right)=\left(\begin{matrix}
			\boldsymbol{\Sigma}&0\\
			0&-\boldsymbol{\Sigma}
		\end{matrix}\right)
	\end{equation*}
	Thus, we have $\vert\lambda_{2r}(\mathbf{Y}^{*})\vert=\sigma_{r}$. Similarly, we could prove $$\operatorname{SVD}_{2r}(\hat{\mathbf{Y}})=\frac{1}{\sqrt{2}}\left(\begin{matrix}
		\hat{\mathbf{U}}_{r}&\hat{\mathbf{U}}_{r}\\
		\hat{\mathbf{V}}_{r}&-\hat{\mathbf{V}}_{r}
	\end{matrix}\right)^{\top}\hat{\mathbf{Y}}\frac{1}{\sqrt{2}}\left(\begin{matrix}
		\hat{\mathbf{U}}_{r}&\hat{\mathbf{U}}_{r}\\
		\hat{\mathbf{V}}_{r}&-\hat{\mathbf{V}}_{r}
	\end{matrix}\right) =\left(\begin{matrix}
		0&\operatorname{SVD}_{r}(\hat{\mathbf{M}})\\
		\operatorname{SVD}_{r}(\hat{\mathbf{M}}^{\top})&0
	\end{matrix}\right),$$
	which implies \begin{equation*}
		\Vert \operatorname{SVD}_{2r}(\hat{\mathbf{Y}})-\mathbf{Y}^{*}\Vert_{\mathrm{F}}=	\sqrt{2}\Vert \operatorname{SVD}_{r}(\hat{\mathbf{M}})-\mathbf{M}^{*}\Vert_{\mathrm{F}}.
	\end{equation*}
	Combine the above equation with $\Vert\hat{\mathbf{Z}}\Vert_{\mathrm{F}}=\sqrt{2}\Vert\mathbf{Z}\Vert_{\mathrm{F}}$,
	$\Vert\hat{\mathbf{Z}}\Vert=\Vert\mathbf{Z}\Vert$ and Equation \ref{eq26} becomes
	$$\Vert \operatorname{SVD}_{r}(\hat{\mathbf{M}})-\mathbf{M}^{*}\Vert_{\mathrm{F}}\leq \Vert\mathbf{Z}\Vert_{\mathrm{F}} + 20\frac{\Vert \mathbf{Z}\Vert\Vert \mathbf{Z}\Vert_{\mathrm{F}}}{\sigma_{r}}.$$
	Notice that proof of $\Vert \operatorname{SVD}_{r}(\hat{\mathbf{M}})-\mathbf{M}^{*}\Vert_{\mathrm{F}}\leq \Vert\mathbf{Z}\Vert_{\mathrm{F}} + 20\frac{\Vert \mathbf{Z}\Vert\Vert \mathbf{Z}\Vert_{\mathrm{F}}}{\sigma_{r}}$ is similar to the Frobinius norm case. Apply the perturbated eigenvector results of Lemma~\ref{teclem:symmetric perturbation} here and it becomes \begin{equation}
		\bigg\Vert \frac{1}{2}\left(\begin{matrix}
			\hat{\mathbf{U}}_{r}&\hat{\mathbf{U}}_{r}\\
			\hat{\mathbf{V}}_{r}&-\hat{\mathbf{V}}_{r}
		\end{matrix}\right)^{\top}\left(\begin{matrix}
			\hat{\mathbf{U}}_{r}&\hat{\mathbf{U}}_{r}\\
			\hat{\mathbf{V}}_{r}&-\hat{\mathbf{V}}_{r}
		\end{matrix}\right) - \frac{1}{2}\left(\begin{matrix}
			\mathbf{U}&\mathbf{U}\\
			\mathbf{V}&-\mathbf{V}
		\end{matrix}\right)^{\top}\left(\begin{matrix}
			\mathbf{U}&\mathbf{U}\\
			\mathbf{V}&-\mathbf{V}
		\end{matrix}\right)\bigg\Vert\leq \frac{4\Vert \mathbf{Z}\Vert}{\sigma_{r}},
	\end{equation}
	which implies \begin{equation}
		\Vert \hat{\mathbf{U}}_{r}\hat{\mathbf{U}}_{r}^{\top} -\mathbf{U}\mathbf{U}^{\top}\Vert\leq \frac{4\Vert \mathbf{Z}\Vert}{\sigma_{r}}=\frac{4}{\sigma_{r}}\Vert \hat{\mathbf{M}}-\mathbf{M}\Vert.
	\end{equation}
	\begin{equation}
		\Vert \hat{\mathbf{V}}_{r}\hat{\mathbf{V}}_{r}^{\top} -\mathbf{U}\mathbf{U}^{\top}\Vert\leq \frac{4\Vert \mathbf{Z}\Vert}{\sigma_{r}}=\frac{4}{\sigma_{r}}\Vert \hat{\mathbf{M}}-\mathbf{M}\Vert.
	\end{equation}
	
\end{proof}
\section{Proof of Matrix Recovery Empirical Process Theorem~\ref{thm:empirical process}}
\label{proof:thm:empirical}
\begin{proof}
	The proof follows from Theorem~\ref{tecthm:orlicz norm empirical}. Denote $\MM_r:=\{\M\in\RR^{d_1\times d_2}: \rank(\M)\leq r\}$ the set of matrices with rank bounded by $r$. And consider 
	\begin{align*}
		&~~~~Z:= \sup_{\M\in\MM_r} |f(\M)-f(\M^*) - \EE(f(\M)-f(\M^*))|/\fro{\M-\M^*}.
	\end{align*}
	Together with $f(\M)= \sum_{i=1}^n\rho(\inp{\X_i}{\M}-y_i)$, we have 
	\begin{align*}
		\EE Z &\leq 2\EE \sup_{\M\in\MM_r} \left|\sum_{i=1}^n\wt\epsilon_i\cdot\big(\rho(\inp{\X_i}{\M}-y_i) - \rho(\inp{\X_i}{\M^*}-y_i)\big)\right|\cdot\fro{\M-\M^*}^{-1}\notag\\
		&\leq 4\tilde L\cdot\EE \sup_{\M\in\MM_r}\left|\sum_{i=1}^n\wt\epsilon_i\cdot\inp{\X_i}{\M-\M^*}\right|\cdot\fro{\M-\M^*}^{-1}\notag\\
		&\leq 4\tilde L\cdot\EE\|\sum_{i=1}^n\wt\epsilon_i\X_i\|_{\mathrm{F,2r}}
		\leq 4\tilde L\cdot\sqrt{2r}\EE\|\sum_{i=1}^n\wt\epsilon_i\X_i\|,
	\end{align*}
	where $\{\wt\epsilon_{i}\}_{i=1}^n$ is Rademacher sequence independent of $\X_1,\ldots,\X_n$. The first inequality is from Theorem \ref{Symmetrization of Expectation} and  the second inequality is from Theorem \ref{Contraction Theorem}. Notice that $\sum_{i=1}^n\wt\epsilon_i\X_i$ has i.i.d. $N(0,n)$ entries and therefore 
	$\EE\|\sum_{i=1}^n\wt\epsilon_i\X_i\|_2\leq C_1\sqrt{nd}$ for some absolute constant $C_1>0$. And thus we have $$\EE Z\leq 4\sqrt{2}C_1\tilde L\sqrt{nd_1r}.$$
	Moreover,
	\begin{align*}
		&~~~~\Vert\sup_{\M\in\MM_r} \vert \rho(Y_{i}-\langle \mathbf{X}_{i}, \mathbf{M}\rangle) - \rho(Y_{i} - \langle \mathbf{X}_{i}, \mathbf{M}^{*}\rangle)\vert\cdot\fro{\M-\M^*}^{-1}\Vert_{\Psi_{2}}\\
		&\leq \tilde{L} \Vert\sup_{\M\in\MM_r} \vert \langle \mathbf{X}_{i},\mathbf{M}-\mathbf{M}^{*}\rangle\vert\cdot\fro{\M-\M^*}^{-1}\Vert_{\Psi_{2}}\\
		&\leq \tilde{L}\|\X_i\|_{\mathrm {F,2r}}
		\leq C_2\tilde{L}\sqrt{d_1r},
	\end{align*}
	for some absolute constant $C_2>0$.
	Then, by Lemma~\ref{orlicz norm max}, the Orlcz norm of its maximum in $i=1,\cdots,n$ could be bounded:
	\begin{align*}
		\Vert\max_{i=1,\ldots,n}\sup_{\M\in\MM_r}\vert \rho(Y_{i}-\langle \mathbf{X}_{i}, \mathbf{M}\rangle) - \rho(Y_{i} - \langle \mathbf{X}_{i}, \mathbf{M}^{*}\rangle)\vert\cdot\fro{\M-\M^*}^{-1}\Vert_{\Psi_{2}}\leq C_2\tilde{L}\sqrt{d_1r\log n},
	\end{align*}
	\begin{align*}
		\mathbb{E}\left[\max_{i=1,\ldots,n}\sup_{\M\in\MM_r}\vert \rho(Y_{i}-\langle \mathbf{X}_{i}, \mathbf{M}\rangle) - \rho(Y_{i} - \langle \mathbf{X}_{i}, \mathbf{M}^{*}\rangle)\vert\cdot\fro{\M-\M^*}^{-1}\right]\leq 3C_2\tilde{L}\sqrt{d_1r\log n}.
	\end{align*}
	Also, note that
	\begin{align*}
		\mathbb{E}\big[\big(\rho(Y_{i}-\langle \mathbf{X}_{i}, \mathbf{M}\rangle) - \rho(Y_{i} - \langle \mathbf{X}_{i}, \mathbf{M}^{*}\rangle)\big)^2\big]/\Vert \mathbf{M}-\mathbf{M}^{*}\Vert^{2}\leq \tilde{L}^{2}\mathbb{E}\langle \mathbf{X}_{i}, \mathbf{M}-\mathbf{M}^{*}\rangle^2/\Vert \mathbf{M}-\mathbf{M}^{*}\Vert^{2} = \tilde{L}^{2}.
	\end{align*}
	Invoke Theorem \ref{tecthm:orlicz norm empirical} and take $\eta=1$, $\delta=0.5$ and there exists some constant $C>0$, and we have 
	\begin{align*}
		\mathbb{P}(Z&\geq 8\sqrt{2}C_{2}\sqrt{d_1r}+t)\leq \exp \left(-\frac{t^{2}}{3 n\tilde{L}^{2}}\right)+3 \exp \left(-\left(\frac{t}{C\tilde{L}\sqrt{d_1r\log n}}\right)^{2}\right)
	\end{align*}
	holds for any $t>0$.
	Take $t=C\sqrt{nd_1r}\tilde{L}$ and then we have 
	\begin{equation*}
		\mathbb{P}(Z\geq 8\sqrt{2}C_1\sqrt{nd_1r}+C\tilde{L}\sqrt{nd_1r})\leq \exp\left(-\frac{C^{2}d_1r}{3}\right)+3\exp\left(-\frac{n}{\log n}\right),
	\end{equation*}
	which completes the proof.
\end{proof}
\begin{theorem}[Symmetrization of Expectations, \citep{van1996weak}]\label{Symmetrization of Expectation} 
	Consider $\mathbf{X}_{1},\mathbf{X}_{2},\cdots,\mathbf{X}_{n}$ independent matrices in $\chi$ and let $\mathcal{F}$ be a class of real-valued functions on $\chi$. Let $\tilde{\varepsilon}_{1},\cdots,\tilde{\varepsilon}_{n}$ be a Rademacher sequence independent of $\mathbf{X}_{1},\mathbf{X}_{2},\cdots,\mathbf{X}_{n}$, then
	\begin{equation}
		\mathbb{E}\big[\sup_{f\in\mathcal{F}}\big{|} \sum_{i=1}^{n} (f(\mathbf{X}_{i}) - \mathbb{E}f(\mathbf{X}_{i}))\big{|}\big]\leq 2\mathbb{E}\big[\sup_{f\in\mathcal{F}} \big{|} \sum_{i=1}^{n} \tilde{\varepsilon}_{i}f(\mathbf{X}_{i})\big{|}\big]
	\end{equation}
\end{theorem}

\begin{theorem}[Contraction Theorem, \citep{ludoux1991probability}]\label{Contraction Theorem}
	
	Consider the non-random elements $x_{1}, \ldots, x_{n}$ of $\chi$. Let $\mathcal{F}$ be a class of real-valued functions on $\chi$. Consider the Lipschitz continuous functions $\rho_{i}: \mathbb{R} \rightarrow \mathbb{R}$ with Lipschitz constant $L$, i.e.
	$$
	\left|\rho_{i}(\mu)-\rho_{i}(\tilde{\mu})\right| \leq L|\mu-\tilde{\mu}|, \text { for all } \mu, \tilde{\mu} \in \mathbb{R}
	$$
	Let $\tilde{\varepsilon}_{1}, \ldots, \tilde{\varepsilon}_{n}$ be a Rademacher sequence $.$ Then for any function $f^{*}: \chi \rightarrow \mathbb{R}$, we have
	
	\begin{equation}
		\mathbb{E}\left[\sup _{f \in \mathcal{F}}\left|\sum_{i=1}^{n} \tilde{\varepsilon}_{i}\left\{\rho_{i}\left(f\left(x_{i}\right)\right)-\rho_{i}\left(f^{*}\left(x_{i}\right)\right)\right\}\right|\right] \leq 2 \mathbb{E}\left[L\sup_{f \in \mathcal{F}} \mid \sum_{i=1}^{n} \tilde{\varepsilon}_{i}\left(f\left(x_{i}\right)-f^{*}\left(x_{i}\right)\right)\mid \right]
	\end{equation}
\end{theorem}

\begin{theorem}[Tail inequality for suprema of empirical process \citep{adamczak2008tail}]
	Let $\mathbf{X}_{1}, \ldots, \mathbf{X}_{n}$ be independent random variables with values in a measurable space $(\mathcal{S}, \mathcal{B})$ and let $\mathcal{F}$ be a countable class of measurable functions $f: \mathcal{S} \rightarrow \mathbb{R}$. Assume that for every $f \in \mathcal{F}$ and every $i, \mathbb{E} f\left(\mathbf{X}_{i}\right)=0$ and for some $\alpha \in(0,1]$ and all $i,\left\|\sup _{\mathcal{F}}\left|f\left(\mathbf{X}_{i}\right)\right|\right\|_{\Psi_{\alpha}}<\infty .$ Let
	$$
	Z=\sup _{f \in \mathcal{F}}\left|\sum_{i=1}^{n} f\left(\mathbf{X}_{i}\right)\right|
	$$
	Define moreover
	$$
	\sigma^{2}=\sup _{f \in \mathcal{F}} \sum_{i=1}^{n} \mathbb{E} f\left(\mathbf{X}_{i}\right)^{2}
	$$
	Then, for all $0<\eta<1$ and $\delta>0$, there exists a constant $C=C(\alpha, \eta, \delta)$, such that for all $t \geq 0$
	\begin{equation*}
		\begin{split}
			\mathbb{P}(Z&\geq(1+\eta) \mathbb{E}Z+t) \\
			& \leq \exp \left(-\frac{t^{2}}{2(1+\delta) \sigma^{2}}\right)+3 \exp \left(-\left(\frac{t}{C\left\|\max _{i} \sup _{f \in \mathcal{F}}\left|f\left(\mathbf{X}_{i}\right)\right|\right\|_{\Psi_{\alpha}}}\right)^{\alpha}\right)
		\end{split}
	\end{equation*}
	and
	$$
	\begin{aligned}
		\mathbb{P}(Z&\leq(1-\eta) \mathbb{E} Z-t) \\
		& \leq \exp \left(-\frac{t^{2}}{2(1+\delta) \sigma^{2}}\right)+3 \exp \left(-\left(\frac{t}{C\left\|\max _{i} \sup _{f \in \mathcal{F}}\left|f\left(\mathbf{X}_{i}\right)\right|\right\|_{\Psi_{\alpha}}}\right)^{\alpha}\right)
	\end{aligned}
	$$
	\label{tecthm:orlicz norm empirical}
\end{theorem}
\begin{remark}
	Notice that here we require $\mathbb{E}f(\mathbf{X}_{i})=0$ and when $\mathbb{E}f(\mathbf{X}_{i})\neq 0$, $\mathbf{Z}$ should be $$
	Z=\sup _{f \in \mathcal{F}}\left|\sum_{i=1}^{n} f\left(\mathbf{X}_{i}\right)- \mathbb{E}f(\mathbf{X}_{i})\right|,
	$$
	and tail inequality of $Z$ would be \begin{equation*}
		\begin{split}
			\mathbb{P}(Z&\geq(1+\eta) \mathbb{E} Z+t) \\
			& \leq \exp \left(-\frac{t^{2}}{2(1+\delta) \sigma^{2}}\right)+3 \exp \left(-\left(\frac{t}{C(\left\|\max _{i} \sup _{f \in \mathcal{F}}\left|f\left(\mathbf{X}_{i}\right)\right|\right\|_{\Psi_{\alpha}}+\mathbb{E}[\max_{i} \vert f(\mathbf{X}_{i})\vert])}\right)^{\alpha}\right)
		\end{split}
	\end{equation*}
\end{remark}
The following lemma is upper bound of Orlicz norm of maximum for $n$ random variables.
\begin{lemma}(Maximum of Sub-Gaussian)
	Let $X_{1},\cdots,X_{n}$ be a sequence of centered sub-gaussian random variables. Then the expectation and the Orlicz norm of the maximum could be upper bounded $$\mathbb{E}[\max_{i} X_{i}]\leq K\sqrt{2\log n},$$
	$$\mathbb{P}(\max_{i} X_{i}>t)\leq n\exp(-\frac{t^{2}}{2K^{2}}),$$
	where $K=\max_{i}\Vert X_{i}\Vert_{\Psi_{2}}$.
	\label{orlicz norm max}
\end{lemma}
\begin{proof}
	For any $s>0$, we have
	\begin{equation}
		\begin{split}
			\mathbb{P}(\max_{i} X_{i}>t)&=\mathbb{P}(\max_{i} \exp (sX_{i})>\exp st)\\
			&\leq \exp(-st)\mathbb{E}[\max_{i} \exp(sX_{i})]\\
			&\leq \exp(-st)\mathbb{E}[\sum_{i}^{n} \exp(sX_{i})]\\
			&\leq n\exp(-st+s^{2}K^{2}).
		\end{split}
	\end{equation}
	Take $s=t/K^{2}$ and then it becomes $ \mathbb{P}(\max_{i} X_{i}>t)\leq n\exp(-\frac{t^{2}}{2K^{2}})$.
	Then consider $\mathbb{E}[\max_{i}X_{i}]$. For any $c>0$,
	\begin{equation}
		\begin{split}
			\mathbb{E}[\max_{i} X_{i}]&=\frac{1}{c}\mathbb{E}[\log\exp(\max_{i}cX_{i})]\\
			&\leq \frac{1}{c}\log\mathbb{E}[\exp(\max_{i}cX_{i})]\\
			&= \frac{1}{c}\log\mathbb{E}[\max_{i}\exp cX_{i}]\\
			&\leq \frac{1}{c}\log\mathbb{E}[\sum_{i}\exp cX_{i}]\\
			&\leq \frac{\log n}{c} + \frac{cK^{2}}{2}.
		\end{split}
	\end{equation}
	Take $c=\sqrt{2\log n}/K$ and then get $\mathbb{E}[\max_{i} X_{i}]\leq K\sqrt{2\log n}$.
\end{proof}

\begin{remark}
	When $X_{1},\ldots,X_{n}$ are not centered, the upper bound of expectation term becomes $$\mathbb{E}[\max_{i} X_{i}]\leq \max_{i}\mathbb{E}X_{i} + K\sqrt{2\log n}.$$
\end{remark}

\section{Proof of Tensor Recovery Empirical Process Theorem~\ref{thm:tensor:empirical process}}
\label{proof:thm:tensor:empirical process}
\begin{proof}
	The proof follows from Theorem~\ref{tecthm:orlicz norm empirical}. Recall that $\MM_{\r}=\{\M\in\RR^{d_1\times\dots\times d_m}: \rank(\M_{(j)})\leq r_j,\, j =1,\cdots,m\}$ is the set of tensors with Tucker rank at most $\r$. Then consider 
	\begin{align*}
		&~~~~Z:= \sup_{\M\in\MM_\r} |f(\M)-f(\M^*) - \EE(f(\M)-f(\M^*))|/\fro{\M-\M^*}.
	\end{align*}
	Together with $f(\M)= \sum_{i=1}^{n}\rho(\inp{\X_i}{\M-\M^*}-\xi_{i})$, we bound expectation of $Z$
	\begin{align*}
		\EE Z &\leq 2\EE \sup_{\M\in\MM_\r} \left|\sum_{i=1}^{n}\wt\epsilon_{i}\cdot\big(\rho(\inp{\X_i}{\M-\M^*}-\xi_{i}) - \rho(-\xi_{i})\big)\right|\cdot\fro{\M-\M^*}^{-1}\notag\\
		&\leq 4\tilde L\cdot\EE \sup_{\M\in\MM_\r}\left|\sum_{i=1}^{n}\wt\epsilon_{i}\cdot\inp{\X_i}{\M-\M^*}\right|\cdot\fro{\M-\M^*}^{-1}\notag\\
		&\leq 4\tilde L\cdot\EE\|\sum_{i=1}^{n}\wt\epsilon_{i}\X_{i}\|_{\mathrm{F,2\r}},
	\end{align*}
	where $\{\wt\epsilon_{i}\}_{i=1}^n$ is Rademacher sequence independent of $\X_{1},\ldots,\X_{n}$. The first inequality is from Theorem \ref{Symmetrization of Expectation} and  the second inequality is from Theorem \ref{Contraction Theorem}. Notice that $\sum_{i=1}^{n}\wt\epsilon_{i}\X_{i}$ has i.i.d. Gaussian entries and therefore 
	$\EE\|\sum_{i=1}^n\wt\epsilon_i\X_i\|_{\rm F, 2\r}\leq C_1 \sqrt{n(2\sum_{j=1}^{m}r_jd_j+2^m\cdot r_1r_2\cdots r_m)}$ for some absolute constant $C_1>0$ (see Lemma~\ref{teclem:tensor:gaussian F norm}). Thus we have $$\EE Z\leq 4C_1\tilde{L} \cdot \sqrt{n(2\sum_{j=1}^{m}r_jd_j+2^m\cdot r_1r_2\cdots r_m)}.$$
	Moreover,
	\begin{align*}
		&~~~~\Vert\sup_{\M\in\MM_\r} \vert \rho(Y_{i}-\langle \mathbf{X}_{i}, \mathbf{M}\rangle) - \rho(Y_{i} - \langle \mathbf{X}_{i}, \mathbf{M}^{*}\rangle)\vert\cdot\fro{\M-\M^*}^{-1}\Vert_{\Psi_{2}}\\
		&\leq \tilde{L} \Vert\sup_{\M\in\MM_\r} \vert \langle \mathbf{X}_{i},\mathbf{M}-\mathbf{M}^{*}\rangle\vert\cdot\fro{\M-\M^*}^{-1}\Vert_{\Psi_{2}}\leq \tilde{L}\Vert \|\X_{i}\|_{\rm F, 2\r} \Vert_{\Psi_{2}}\\
		&\leq C_2\tilde{L} \cdot \sqrt{2\sum_{j=1}^{m}r_jd_j+2^m\cdot r_1r_2\cdots r_m},
	\end{align*}
	for some absolute constant $C_2>0$.
	Then, by Lemma~\ref{orlicz norm max}, the Orlicz norm of its maximum in $i=1,\cdots,n$ could be bounded:
	\begin{align*}
		&{~~~~~}\Vert\max_{i=1,\ldots,n}\sup_{\M\in\MM_\r}\vert \rho(Y_{i}-\langle \mathbf{X}_{i}, \mathbf{M}\rangle) - \rho(Y_{i} - \langle \mathbf{X}_{i}, \mathbf{M}^{*}\rangle)\vert\cdot\fro{\M-\M^*}^{-1}\Vert_{\Psi_{2}}\\
		&\leq C_2\tilde{L} \sqrt{(2\sum_{j=1}^{m}r_jd_j+2^m\cdot r_1r_2\cdots r_m)\log n} ,
	\end{align*}
	\begin{align*}
		&{~~~~~}\mathbb{E}\left[\max_{i=1,\ldots,n}\sup_{\M\in\MM_\r}\vert \rho(Y_{ij}-\langle \mathbf{X}_{ij}, \mathbf{M}\rangle) - \rho(Y_{ij} - \langle \mathbf{X}_{ij}, \mathbf{M}^{*}\rangle)\vert\cdot\fro{\M-\M^*}^{-1}\right]\\
		&\leq 3C_2\tilde{L} \sqrt{(2\sum_{j=1}^{m}r_jd_j+2^m\cdot r_1r_2\cdots r_m)\log n}.
	\end{align*}
	Also, note that
	\begin{align*}
		\mathbb{E}\big[\big(\rho(Y_{i}-\langle \mathbf{X}_{i}, \mathbf{M}\rangle) - \rho(Y_{i} - \langle \mathbf{X}_{i}, \mathbf{M}^{*}\rangle)\big)^2\big]/\Vert \mathbf{M}-\mathbf{M}^{*}\Vert^{2}\leq \tilde{L}^{2}.
	\end{align*}
	Invoke Theorem \ref{tecthm:orlicz norm empirical} and take $\eta=1$, $\delta=0.5$ and there exists some constant $C>0$ such that
	\begin{align*}
		&\mathbb{P}\left(Z\geq 8C_{2}\sqrt{n(2\sum_{j=1}^{m}r_jd_j+2^m\cdot r_1r_2\cdots r_m)}+t\right)\\
		&{~~~~~~~~~}\leq \exp \left(-\frac{t^{2}}{3 n\tilde{L}^{2}}\right)+3 \exp \left(-\left(\frac{t}{C\tilde{L} \sqrt{(2\sum_{j=1}^{m}r_jd_j+2^m\cdot r_1r_2\cdots r_m)\log n}}\right)^{2}\right)
	\end{align*}
	holds for any $t>0$.
	Take $t=C\tilde{L}\sqrt{n(2\sum_{j=1}^{m}r_jd_j+2^m\cdot r_1r_2\cdots r_m)}$ and then we have 
	\begin{align*}
		&\mathbb{P}\left(Z\geq (8C_1+C\tilde{L})\sqrt{n(2\sum_{j=1}^{m}r_jd_j+2^m\cdot r_1r_2\cdots r_m)}\right)\\
		&{~~~~~~~}\leq \exp\left(-\frac{C^{2}(2\sum_{j=1}^{m}r_jd_j+2^m\cdot r_1r_2\cdots r_m) }{3}\right)+3\exp\left(-\frac{n}{\log n}\right),
	\end{align*}
	which completes the proof.
\end{proof}

\end{document}